\documentclass{article}
\usepackage{amsmath,amssymb}
\usepackage{amsopn,amsthm}
\usepackage{subcaption}
\usepackage{fullpage}
\usepackage{graphicx}
\usepackage{color}
\usepackage{psfrag}
\usepackage{tikz}
\usepackage{paralist}

\usepackage{mathrsfs} 
\usepackage{bm}
\usepackage{cases}

\usepackage{algpseudocode} 
\usepackage{algorithm}
\floatstyle{plain} \newfloat{myalgo}{tbhp}{mya}

\usepackage[a]{esvect}
\usepackage{float}

\newenvironment{Algorithm}[1][tbh]%
{\begin{myalgo}[#1]
    \centering
    \begin{minipage}{0.9\textwidth}
      \begin{algorithm}[H]}%
      {\end{algorithm}
    \end{minipage}
  \end{myalgo}}

\usepackage[colorlinks=true]{hyperref}
\hypersetup{urlcolor=blue, citecolor=blue, linkcolor=blue}

\newtheorem{theorem}{Theorem}[section]
\newtheorem{proposition}[theorem]{Proposition}
\newtheorem{lemma}[theorem]{Lemma}

\theoremstyle{definition}

\newtheorem{example}[theorem]{Example}

\theoremstyle{remark} \newtheorem{remark}[theorem]{Remark}

\numberwithin{equation}{section}
\numberwithin{figure}{section}
\numberwithin{algorithm}{section}

\definecolor{darkgreen}{rgb}{0.0, 0.5, 0.0}

\newcommand{\field}[1]{{\mathbb{#1}}}
\newcommand{\C}{\field{C}}
 
\newcommand{\N}{\field{N}}
\newcommand{\R}{\field{R}}

\newcommand{\G}{\mathbb{G}}
\newcommand{\I}{\mathbb{I}}
\newcommand{\M}{\mathbb{M}}

\newcommand{\Ocal}{\mathcal{O}}
\newcommand{\Pcal}{\mathcal{P}}

\newcommand{\bs}{\boldsymbol} 

\newcommand{\bfb}{{\bs b}}
 
\newcommand{\bfd}{{\bs d}}
\newcommand{\bfe}{{\bs e}}

\newcommand{\bfh}{{\bs h}}
\newcommand{\bfn}{{\bs n}}
\newcommand{\bfp}{{\bs p}} 
 
\newcommand{\bfr}{{\bs r}}
\newcommand{\bft}{{\bs t}}

\newcommand{\bfv}{{\bs v}} 

\newcommand{\bfx}{{\bs x}}
\newcommand{\bfy}{{\bs y}}

\newcommand{\bfA}{{\bs A}}

\newcommand{\bfE}{{\bs E}}
\newcommand{\bfF}{{\bs F}}

\newcommand{\bfH}{{\bs H}}

\newcommand{\bfV}{{\bs V}}

\newcommand{\bfnu}{{\bs\nu}}

\newcommand{\bftheta}{{\bs\theta}}
\newcommand{\bfxi}{{\bs\xi}}

\newcommand{\bfDelta}{{\bs\Delta}}

\newcommand{\loc}{{\mathrm{loc}}}
\newcommand{\tor}{\mathrm{tor}}

\newcommand{\ol}[1]{\overline{#1}}

\newcommand{\tm}{\subseteq} 
\newcommand{\di}{\partial}
\newcommand{\trans}{{\top}}

\newcommand{\ds}{\, \dif s}

\newcommand{\dx}{\, \dif \bfx}

\newcommand{\dmu}{\,  \dif\mu}

\newcommand{\dez}{\, \dif(\eta,\zeta)}

\newcommand{\xhat}{\widehat{\bfx}}

\newcommand{\vtilde}{{\widetilde v}}
\newcommand{\wtilde}{{\widetilde w}}

\newcommand{\Atilde}{{\widetilde A}}

\newcommand{\gammatilde}{{\widetilde\gamma}}

\newcommand{\rmi}{\mathrm{i}}

\newcommand{\eps}{\varepsilon}

\newcommand{\Sone}{{S^1}}
\newcommand{\Stwo}{{S^2}}
\newcommand{\Rd}{{\R^3}}
\newcommand{\Rtwo}{{\R^2}}
\newcommand{\Cd}{{\C^3}}
\newcommand{\Rdd}{{\R^{3\times3}}}

\newcommand{\curl}{{\mathbf{curl}}}
\renewcommand{\div}{\mathrm{div}}
\newcommand{\grad}{\nabla}

\newcommand{\crossgrad}{\grad_{\eta,\zeta}'}
\newcommand{\crossdiv}{\div_{\eta,\zeta}'}

\newcommand{\Ei}{\bfE^i}
\newcommand{\Hi}{\bfH^i}

\newcommand{\Esrho}{\bfE^s_\rho}
\newcommand{\Hsrho}{\bfH^s_\rho}
\newcommand{\Etrho}{\bfE_\rho}
\newcommand{\Htrho}{\bfH_\rho}
\newcommand{\Einftyrho}{\bfE^\infty_\rho}

\newcommand{\Einftyrhotilde}{\widetilde{\bfE^\infty_{\rho}}}

\newcommand{\Esrhon}{\bfE^s_{\rho_n}}

\newcommand{\Einftyrhon}{\bfE^\infty_{\rho_n}}

\newcommand{\rhon}{{\rho_n}}

\newcommand{\Frho}{F_\rho}
\newcommand{\Trho}{T_\rho}

\newcommand{\BR}{{B_R(0)}}
\newcommand{\Br}{B_r'(0)}
\newcommand{\Brho}{B_{\rho}'(0)}
\newcommand{\Brhon}{B_{\rho_n}'(0)}
\newcommand{\Drho}{D_\rho}
\newcommand{\Drhon}{D_{\rho_n}}

\newcommand{\Omegar}{{\Omega_r}}

\newcommand{\Meps}{\M^\eps}
\newcommand{\Mmu}{\M^\mu}
\newcommand{\Mgamma}{\M^\gamma}

\newcommand{\meps}{m^\eps}
\newcommand{\mmu}{m^\mu}
\newcommand{\mgamma}{m^\gamma}

\newcommand{\Wrhonxi}{W^{(\bfxi)}_{\rhon}}
\newcommand{\Wrhonej}{W_{\rhon}^{(e_j)}}
\newcommand{\Wrhonxitilde}{\widetilde{W}^{(\bfxi)}_{\rhon}}

\newcommand{\wrhonxi}{w^{(\bfxi')}_{\rhon}}

\newcommand{\wrhonxitheta}{w^{(\Rthetarhon^{-1}\bfxi')}_{\rhon}}
\newcommand{\wrhonxithetatilde}{\widetilde{w}^{(\Rthetarhon^{-1}\bfxi')}_{\rhon}}

\newcommand{\frhon}{f_{\rhon}}

\newcommand{\grhon}{g_{\rhon}}
\newcommand{\hrho}{h_\rho}

\newcommand{\Rtheta}{R_\theta}
\newcommand{\Rthetaprime}{\frac{\di \Rtheta}{\di s}}

\newcommand{\thetarho}{{\theta}}
\newcommand{\thetarhon}{{\theta}}
\newcommand{\thetarhonprime}{\frac{\di\theta}{\di s}}

\newcommand{\Rthetarho}{R_\thetarho}
\newcommand{\Rthetarhon}{R_\thetarhon}

\newcommand{\Kext}{\Gamma}

\newcommand{\pK}{\bfp_\Kext}
\newcommand{\rK}{\bfr_\Kext}

\newcommand{\tK}{\bft_\Kext}
\newcommand{\nK}{\bfn_\Kext}
\newcommand{\bK}{\bfb_\Kext}

\newcommand{\tKprime}{\frac{\di\tK}{\di s}}
\newcommand{\nKprime}{\frac{\di\nK}{\di s}}
\newcommand{\bKprime}{\frac{\di\bK}{\di s}}

\newcommand{\tp}{\bft_\bfp}
\newcommand{\np}{\bfn_\bfp}
\newcommand{\bp}{\bfb_\bfp}

\newcommand{\JrK}{J_\Kext}

\newcommand{\rrhon}{r_{\rho_n}}
\newcommand{\wrho}{w_\rho}
\newcommand{\wrhotilde}{\widetilde{w}_\rho}

\newcommand{\kappamax}{\kappa_{\max}}

\newcommand{\tri}{\bigtriangleup}
\newcommand{\Ptri}{\Pcal_\tri}

\newcommand{\xvec}{\vv{\bfx}}
\newcommand{\RelDiff}{\mathtt{RelDiff}}
\newcommand{\smax}{s_{\mathrm{max}}}

\newcommand{\thetaprime}{\frac{\di\theta}{\di s}}

\DeclareMathOperator{\dif}{d\!}

\DeclareMathOperator{\diag}{diag}

\DeclareMathOperator{\supp}{supp}

\DeclareMathOperator{\argmin}{argmin}

\begin{document}

\title{An asymptotic representation formula for scattering
  by thin tubular structures and an application in inverse scattering} 
\author{Yves Capdeboscq\footnote{Universit\'e de Paris and Sorbonne
    Universit\'e, CNRS, Laboratoire Jacques-Louis Lions (LJLL),
    F-75006 Paris, France, 
    ({\tt yves.capdeboscq@u-paris.fr})}\,, 
  Roland Griesmaier\footnote{Institut
    f\"ur 
    Angewandte und Numerische Mathematik, Karlsruher Institut f\"ur
    Technologie, Englerstr.~2, 76131 Karlsruhe, Germany 
    ({\tt roland.griesmaier@kit.edu}).}
  and Marvin Kn\"oller\footnote{Institut
    f\"ur 
    Angewandte und Numerische Mathematik, Karlsruher Institut f\"ur
    Technologie, Englerstr.~2, 76131 Karlsruhe, Germany 
    ({\tt marvin.knoeller@kit.edu}).}
}
\date{\today}

\maketitle

\begin{abstract}
  We consider the scattering of time-harmonic electromagnetic waves by
  a penetrable thin tubular scattering object in three-dimensional
  free space. 
  We establish an asymptotic representation formula for the scattered
  wave away from the thin tubular scatterer as the radius
  of its cross-section tends to zero. 
  The shape, the relative electric permeability and the relative
  magnetic permittivity of the scattering object enter this
  asymptotic representation formula by means of the center curve of
  the thin tubular scatterer and two electric and magnetic
  polarization tensors. 
  We give an explicit characterization of these two three-dimensional
  polarization tensors in terms of the center curve and of the two
  two-dimensional polarization tensor for the cross-section of the
  scattering object. 
  As an application we demonstrate how this formula may be used to
  evaluate the residual and the shape derivative in an efficient
  iterative reconstruction algorithm for an inverse scattering problem
  with thin tubular scattering objects. 
  We present numerical results to illustrate our theoretical
  findings. 
\end{abstract}

{\small\noindent
  Mathematics subject classifications (MSC2010): 
  35C20, 	%
  (65N21,  %
  78A46) 	%
  \\\noindent
  Short title: Scattering by thin tubular structures
}

\section{Introduction}
\label{sec:Introduction}
In this work we study time-harmonic electromagnetic waves in three
dimensional free space that are being scattered by a thin tubular
object. 
We assume that this object can be described as a thin tubular
neighborhood of a smooth center curve with arbitrary, but fixed,
cross-section, possibly twisting along the center curve. 
Assuming that the electric permittivity and the magnetic permeability
of the medium inside this scatterer are real valued and
positive, we discuss an asymptotic representation formula for the
scattered field away from the thin tubular scattering object as the
radius of its cross-section tends to zero.
The goal is to describe the effective behaviour of the scattered field
due to a thin tubular scattering object. 
Our primary motivation is the application of this result to inverse
problems or shape optimization. 

Various low volume expansions for electrostatic potentials, as well as
elastic and electromagnetic fields are available in the literature, 
(see, e.g., \cite{AlbCap18,AmmKan03,AmmKanNakTan02,AmmVogVol01,BerBonFraMaz12,BerFra06,CapVog03a,CedMosVog98,DapVog17,FriVog89,Gri11}).
The framework we use in this work was first introduced in
\cite{CapVog03a,CapVog03b,CapVog06} for electrostatic potentials.
The very general low volume perturbation formula for time-harmonic
Maxwell's equations in bounded domains from \cite{AlbCap18,Gri11} can
be extended to the electromagnetic scattering problem in
unbounded free space as considered in this work using an integral
equation technique developed in \cite{AmmIakMos03,AmmVol05}. 
Applying this result to the special case of thin tubular scattering
objects, the first observation is that the scattered field away from
the scatterer converges to zero as the diameter of its cross-section
tends to zero. 
We consider the lowest order term in the corresponding asymptotic
expansion of the scattered field, which can be written as an integral
over the center curve of the thin tubular scattering object in terms
of (i) the dyadic Green's function of time-harmonic Maxwell's
equations in free space, (ii) the incident field, and (iii) two
effective polarization tensors. 
The range of integration and the electric and magnetic polarization
tensors are the signatures of the shape and of the material parameters
of the thin tubular scattering object in this lowest order term. 

The main contribution of this work is a pointwise characterization of
the eigenvalues and eigenvectors of these polarization tensors for
thin tubular scattering objects.
We show that in each point on the center curve the polarization
tensors have one eigenvector corresponding to the eigenvalue $1$ that
is tangential to the center curve.  
Since polarization tensors are symmetric $3\times3$-matrices, this
implies that there are other two eigenvectors perpendicular to the
center curve. 
We prove that in the plane spanned by these two eigenvectors 
the three-dimensional polarization tensors coincide with the
corresponding two-dimensional polarization tensors for the
cross-section of the thin tubular scattering object. 
This extends an earlier result from \cite{BerCapdeGFra09} for straight 
cylindrical scatterers with arbitrary cross-sections of small area. 
The asymptotic representation formula for the scattered field together
with this pointwise description of the polarization tensors yields an
efficient simplified model for scattering by thin tubular structures. 

For the special case, when the cross-section of the thin tubular
scattering object is an ellipse, explicit formulas for the
two-dimensional polarization tensors of the cross-section are
available, which then gives a completely explicit asymptotic
representation formula for the scattered field. 
We will exemplify how to use this asymptotic representation formula in
possible applications by discussing an inverse scattering problem
with thin tubular scattering objects with circular cross-sections. 
The goal is to recover the center curve of such a scatterer from far
field observations of a single scattered field. 
We make use the asymptotic representation formula to develop an
inexpensive iterative reconstruction scheme that does not require to 
solve a single Maxwell system during the reconstruction process. 
A similar method for electrical impedance tomography has been
considered in \cite{GriHyv11} (see also \cite{BerCapdeGFra09} for a
related inverse problem with thin straight cylinders).
Further applications of asymptotic representation formulas for
electrostatic potentials as well as elastic and electromagnetic fields
with thin objects in inverse problems, image processing, or shape
optimization can, e.g., be found in
\cite{AmmBerFra04,AmmBerFra06,BerGraMusSch14,Dap20,Gri10,ParLes09}. 

The outline of this paper is as follows.
After providing the mathematical model for electromagnetic scattering
by a thin tubular scattering object in the next section, we summarize
the results on the general asymptotic analysis from
\cite{AlbCap18,Gri11} for the special case of thin tubular scattering
objects in Section~\ref{sec:AsymptoticFormula}. 
In Section~\ref{sec:PolarizationTensor} we state and prove our main
theoretical result concerning the explicit characterization of the
polarization tensor of a thin tubular scattering object.
As an application of these theoretical results, we discuss an inverse
scattering problem with thin tubular scattering objects in
Section~\ref{sec:InverseProblem}, and in
Section~\ref{sec:NumericalResults} we provide numerical examples.

\section{Scattering by thin tubular structures}
\label{sec:Setting}
We consider time-harmonic electromagnetic wave propagation in the
unbounded domain $\Rd$ occupied by a homogenous 
\emph{background medium} with constant \emph{electric permittivity}
$\eps_0>0$ and constant \emph{magnetic permeability} $\mu_0>0$. 
Accordingly, the \emph{wave number} $k$ at \emph{frequency} $\omega>0$
is given by~$k=\omega\sqrt{\eps_0\mu_0}$, and an \emph{incident field}
$(\Ei,\Hi)$ is an entire solution to Maxwell's equations 
\begin{equation}
  \label{eq:MaxwellIncident}
  \curl\Ei - \rmi\omega\mu_0\Hi \,=\, 0 \,, \quad
  \curl\Hi + \rmi\omega\eps_0\Ei \,=\, 0 \qquad\text{in }\Rd \,.
\end{equation}

We assume that the homogeneous background medium is perturbed by a
thin tubular scattering object, which shall be given as follows.
Let $\BR\tm\Rd$ be a ball of radius $R>0$ centered at the origin, and
let $\Kext\Subset\BR$ be a simple (i.e., non-self-intersecting but
possibly closed) curve with $C^3$ parametrization by arc length
$\pK:(-L,L)\to\Rd$. 
Assuming that $\pK'(s)\times \pK''(s)\not=0$ for all $s\in (-L,L)$,
the \emph{Frenet-Serret frame} $(\tK,\nK,\bK)$ for $\Kext$ is defined
by 
\begin{equation}
  \label{eq:FrenetSerretFrame}
  \tK(s) \,:=\, \pK'(s) \,, \quad
  \nK(s) \,:=\, \frac{\pK''(s)}{|\pK''(s)|} \,, \quad
  \bK(s) \,:=\, \tK(s)\times \nK(s) \,, \qquad s\in(-L,L) \,.
\end{equation}
For any $\thetarho \in C^1([-L,L])$ let 
\begin{equation}
  \label{eq:DefRtheta}
  \Rthetarho(s) \,:=\,
  \begin{bmatrix}
    \cos(\thetarho(s)) & -\sin(\thetarho(s))\\
    \sin(\thetarho(s)) & \cos(\thetarho(s))
  \end{bmatrix}
  \in \R^{2\time 2} \,, \qquad s\in (-L,L) \,,
\end{equation}
be a two-dimensional parameter-dependent rotation matrix, 
which will be used to twist the cross-section around the curve
$\Kext$ while extruding it along the curve in the geometric 
description of the thin tubular scattering object. 

The tubular neighborhood theorem (see, e.g., 
\cite[Thm.~20, p.~467]{Spi79}) shows that there exists a radius $r>0$
sufficiently small such that the map
\begin{equation}
  \label{eq:LocalCoordsK}
  \rK: (-L,L)\times \Br \to \Rd \,, \quad 
  \rK(s,\eta,\zeta) \,:=\, \pK(s) 
  + \begin{bmatrix}
    \nK(s) & \bK(s)
  \end{bmatrix} 
  \Rthetarho(s)
  \begin{bmatrix}
    \eta\\\zeta  
  \end{bmatrix} \,,
\end{equation}
where $\Br\tm\Rtwo$ is the disk of radius $r$ centered at the origin,
defines a local coordinate system around $\Kext$.
We denote its range by
\begin{equation}
  \label{eq:DefOmegar}
  \Omegar
  \,:=\, \bigl\{ \rK(s,\eta,\zeta) \;\big|\; 
  s\in(-L,L) \,,\; (\eta,\zeta)\in \Br \bigr\} \,.
\end{equation}
Given $0<\ell<L$ and $0<\rho<r/2$ we consider a 
\emph{cross-section} $\Drho'\tm\Brho$ that is just supposed to be 
measurable, and accordingly we  define a 
\emph{thin tubular scattering object} by 
\begin{equation}
  \label{eq:DefDrho}
  \Drho 
  \,:=\, \bigl\{ \rK(s,\eta,\zeta) \;\big|\; 
  s\in(-\ell,\ell) \,,\; 
  (\eta,\zeta) \in \Drho' \bigr\} 
\end{equation}
(see Figure~\ref{fig:Sketch} for a sketch).
In the following we call
\begin{equation*}
  K \,:=\, \bigl\{ \pK(s) \;\big|\; s\in(-\ell,\ell) \bigr\}
\end{equation*}
the \emph{center curve} of $\Drho$, and the parameter $\rho$ is called
the \emph{radius} of the cross-section $\Drho'$ of $\Drho$, or
sometimes just the radius of $\Drho$. 

\begin{figure}[t]
  \centering
  \includegraphics[scale=0.5]{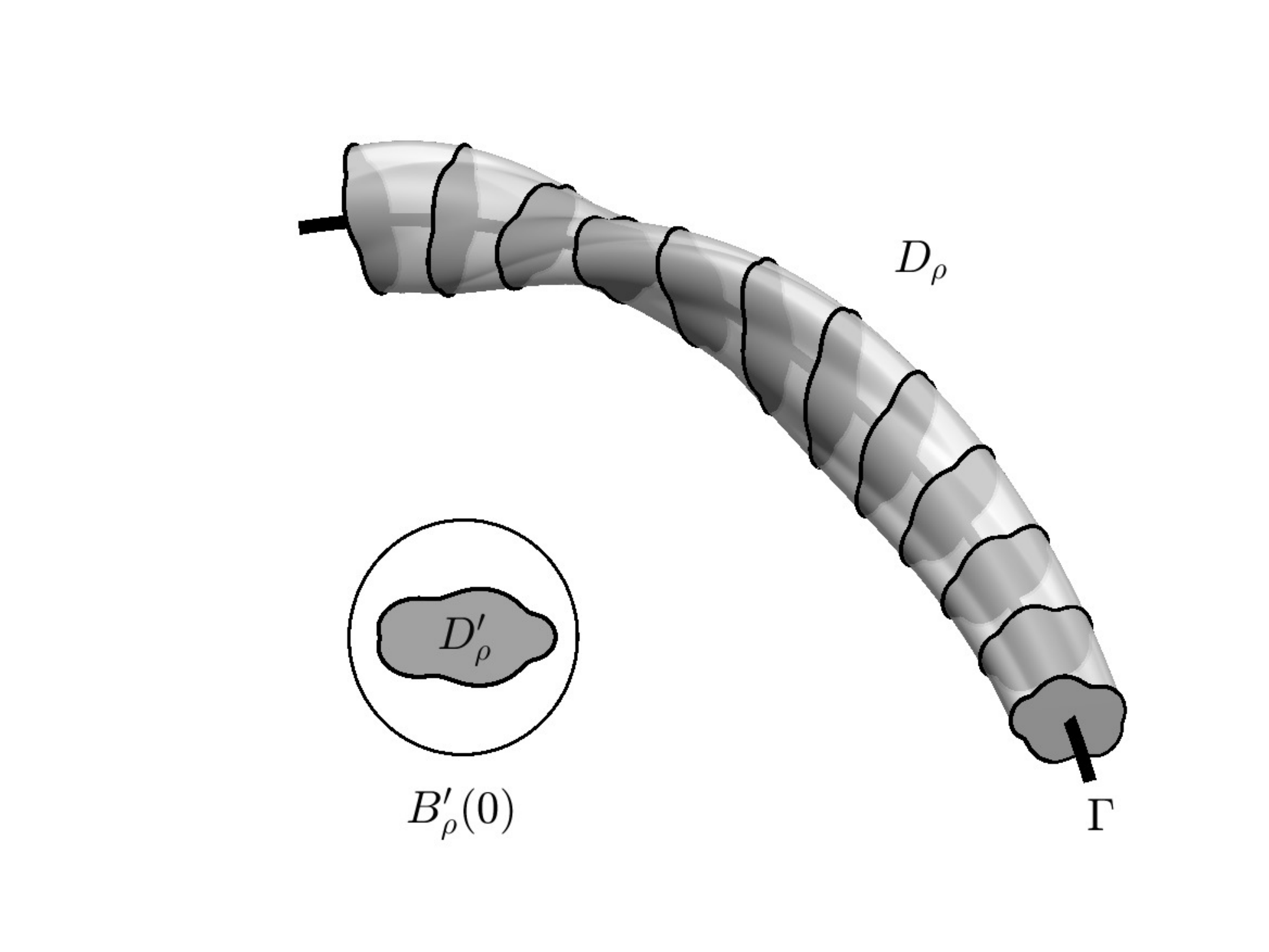}
  \caption{Sketch of a thin tubular scattering object $\Drho$ with
    cross-section $\Drho'\tm\Br$.} 
  \label{fig:Sketch}
\end{figure}

\begin{remark}
  The definition \eqref{eq:DefDrho} does not cover thin tubular
  scattering objects $\Drho$ with closed center curves~$K$. 
  However, the results established in the
  Theorems~\ref{thm:GeneralAsymptotics} and
  \ref{thm:CharacterizationPolTen} below remain valid in this case,
  and the proofs can actually be simplified because one does not have
  to take into account the ends of the tube.
  \hfill$\lozenge$
\end{remark}

We suppose that the medium inside the thin tubular scattering object
has constant electric permittivity $\eps_1>0$ and constant magnetic
permeability $\mu_1>0$. 
Accordingly, the permittivity and permeability distributions in the
entire domain are given by 
\begin{equation}
  \label{eq:Defeps/murho}
  \eps_\rho(\bfx) \,:=\,
  \begin{cases}
    \eps_1\,, & \bfx \in \Drho\,,\\
    \eps_0\,, & \bfx \in \Rd\setminus\overline{\Drho} \,,
  \end{cases}
  \qquad\text{and}\qquad
  \mu_\rho(\bfx) \,:=\,
  \begin{cases}
    \mu_1\,, & \bfx \in \Drho\,,\\
    \mu_0\,, & \bfx \in \Rd\setminus\overline{\Drho} \,.
  \end{cases}
\end{equation}
We also use the notation $\eps_r:=\eps_1/\eps_0$ and
$\mu_r:=\mu_1/\mu_0$ for the \emph{relative electric permittivity} and
the \emph{relative magnetic permeability}, respectively. 
The electromagnetic field $(\Etrho,\Htrho)$ in the perturbed medium 
satisfies 
\begin{equation}
  \label{eq:MaxwellTotal}
  \curl\Etrho - \rmi\omega\mu_\rho\Htrho \,=\, 0 \,, \quad
  \curl\Htrho + \rmi\omega\eps_\rho\Etrho \,=\, 0 \qquad\text{in }\Rd \,.
\end{equation}
Rewriting this \emph{total field} as a superposition 
\begin{equation*}
  (\Etrho,\Htrho)
  \,=\, (\Ei,\Hi)+(\Esrho,\Hsrho)
\end{equation*}
of the incident field $(\Ei,\Hi)$ and a \emph{scattered field}
$(\Esrho,\Hsrho)$, we assume that the scattered field satisfies the
Silver-M\"uller radiation condition 
\begin{equation}
  \label{eq:SilverMullerEsHs}
  \lim_{|\bfx|\to\infty} (\sqrt{\mu_0}\Hsrho(\bfx)\times \bfx 
  - |\bfx|\sqrt{\eps_0}\Esrho(\bfx)) 
  \,=\, 0 
\end{equation}
uniformly with respect to all directions 
$\xhat := \bfx/|\bfx| \in\Stwo$. 

In the following we will work with the electric field only.
Eliminating the magnetic field from the system
\eqref{eq:MaxwellIncident} gives  
\begin{subequations}
  \label{eq:ScatteringProblem}
  \begin{equation}
    \label{eq:MaxwellEi}
    \curl\curl\Ei - k^2 \Ei 
    \,=\, 0 \qquad\text{in }\Rd \,,
  \end{equation}
  while \eqref{eq:MaxwellTotal} reduces to 
  \begin{equation}
    \label{eq:MaxwellE}
    \curl \Bigl(\frac1{\mu_\rho} \curl\Etrho\Bigr) 
    - \omega^2 \eps_\rho \Etrho 
    \,=\, 0 \qquad\text{in }\Rd \,,
  \end{equation}
  and \eqref{eq:SilverMullerEsHs} turns into 
  \begin{equation}
    \label{eq:SilverMullerEs}
    \lim_{|\bfx|\to\infty} (\curl\Esrho(\bfx)\times \bfx 
    - \rmi k |\bfx|\Esrho(\bfx)) 
    \,=\, 0 \,.
  \end{equation}
\end{subequations}

\begin{remark}
  Throughout this work, Maxwell's equations are always to be
  understood in weak sense.
  For instance, $\Etrho\in H_\loc(\curl;\Rd)$ is a solution to
  \eqref{eq:MaxwellE} if and only if
  \begin{equation*}
    \int_\Rd \Bigl( \frac1{\mu_\rho}\curl\Etrho\cdot\curl\bfV 
    - \omega^2\eps_\rho\Etrho\cdot\bfV \Bigr) \dx
    \,=\, 0 \qquad
    \text{for all } \bfV\in H_0(\curl;\Rd) \,.
  \end{equation*}
  Standard regularity results yield smoothness of $\Etrho$ and
  $\Esrho$ in $\Rd\setminus\ol{\BR}$ for some $R>0$ sufficiently
  large, and the entire solution~$\Ei$ is smooth throughout~$\Rd$. 
  In particular the Silver-M\"uller radiation condition
  \eqref{eq:SilverMullerEs} is well defined.~\hfill$\lozenge$
\end{remark}

\begin{lemma}
  Suppose that the incident field $\Ei \in H_\loc(\curl;\Rd)$
  satisfies \eqref{eq:MaxwellEi}.  
  Then there exists a constant $\rho_0>0$ depending only on $R$,
  $\omega$, $\eps_0$ and $\mu_0$ such that for all
  $0<\rho<\rho_0$ the scattering problem 
  \eqref{eq:MaxwellE}--\eqref{eq:SilverMullerEs} has a unique solution
  $\Etrho \in H_\loc(\curl;\Rd)$. 
  Furthermore, the scattered field~$\Esrho$ has the asymptotic
  behaviour 
  \begin{equation*}
    \Esrho(\bfx) 
    \,=\, \frac{e^{\rmi k|\bfx|}}{4\pi|\bfx|} \bigl(\Etrho^\infty(\xhat) 
    + \Ocal(|\bfx|^{-1})\bigr) 
    \qquad \text{as } |\bf\bfx|\to\infty
  \end{equation*}
  uniformly in $\xhat = \bfx/|\bfx|$.
  The vector function $\Einftyrho$ is called the 
  \emph{electric far field pattern}. 
\end{lemma}

\begin{proof}
  The unique solvability follows, e.g., by combining the arguments in
  \cite[Sec.~10.3]{Mon03} with the uniqueness result
  \cite[Thm.~3.1]{BalCapTse12}.
  The far field expansion is, e.g., shown in \cite[Cor.~9.5]{Mon03}. 
\end{proof}

\section{The asymptotic perturbation formula}
\label{sec:AsymptoticFormula}
We derive an asymptotic perturbation formula for the scattered
electric field $\Esrho$ and the electric far field pattern
$\Einftyrho$ as the radius $\rho$ of the cross-section $\Drho'$ of the
scattering object $\Drho$ in \eqref{eq:DefDrho} tends to zero relative
to the wave length $\lambda=2\pi/k$. 
Expansions of this type are available in the literature for
time-harmonic electromagnetic fields (see, e.g.,
\cite{AlbCap18,AmmVogVol01,AmmVol05,Gri11}). 
However, the existing results for Maxwell's equations are either
formulated on bounded domains, or for scattering problems on unbounded
domains but with different geometrical assumptions on the scattering
objects than considered in this work.
In the following we combine a result for boundary value problems with
scatterers of very general geometries from \cite{AlbCap18,Gri11} and
an integral equation technique developed in
\cite{AmmIakMos03,AmmVol05} to arrive at an asymptotic perturbation
formula that applies to our setup. 

We consider a sequence of radii $(\rhon)_n\tm(0,r/2)$ converging to
zero, a sequence of measurable cross-sections $\Drhon'\tm\Brhon$,
$n\in\N$, and a two-dimensional parameter dependent rotation
matrix~$\Rtheta \in C^1([-L,L],\R^{2\times 2})$ as
in~\eqref{eq:DefRtheta}. 

Then, as $n\infty$,
\begin{equation}
  \label{eq:Defmu'}
  |\Drhon'|^{-1} \chi_{\Drhon'} \text{ converges in the sense of
    measures to } \mu'
\end{equation}
where $\mu'$ is the two-dimensional Dirac measure with support 
in $0$. 
Recalling \eqref{eq:FrenetSerretFrame} we denote
by~$\kappa(s):=|\pK''(s)|$ the \emph{curvature} of $\Kext$ and by 
$\tau(s):=-\bKprime(s)\cdot\nK(s)$ the \emph{torsion} of $\Kext$ at
$\pK(s)$. 
A short calculation (see Appendix~\ref{app:LocCoords}) shows
that the Jacobian determinant of the local coordinates $\rK$
from~\eqref{eq:LocalCoordsK} is given by 
\begin{equation}
  \label{eq:LocalCoordsKJacobian}
  \JrK(s,\eta,\zeta)
  \,:=\, \det D\rK(s,\eta,\zeta)
  \,=\, 1-\kappa \bfe_1' \cdot \Rthetarhon(s)
  \begin{bmatrix}
    \eta\\\zeta
  \end{bmatrix} \,, \qquad
  s\in(-L,L) \,,\; (\eta,\zeta)\in\Br \,,
\end{equation}
where $\bfe_1'=(1,0)^\trans\in\R^2$.
Since $\Kext$ is a $C^3$ curve, we have 
$\kappamax:=\|\kappa\|_{C(-L,L)}<\infty$, and it has, e.g., been shown
in \cite[Thm.~1]{LitSimDurRaw99} that the radius $r>0$ from
\eqref{eq:LocalCoordsK} must satisfy $r\kappamax<1$. 
In particular, $|\JrK|=\JrK>0$. 
Using the notation
\begin{equation*}
  \crossgrad u 
  \,:=\, \begin{bmatrix}
    \frac{\di u}{\di\eta} &
    \frac{\di u}{\di\zeta}
  \end{bmatrix}^\trans
  \qquad\text{and}\qquad
  \crossdiv \bfv
  \,:=\, \frac{\di\bfv_\eta}{\di\eta} + \frac{\di\bfv_\zeta}{\di\zeta}
\end{equation*}
for the two-dimensional gradient and the two-dimensional divergence
with respect to $(\eta,\zeta)$, we obtain (see
Appendix~\ref{app:LocCoords} for details) that the three-dimensional
gradient satisfies, for $s\in(-L,L)$ and $(\eta,\zeta)\in\Br$, 
\begin{multline}
  \label{eq:LocalCoordsKGradient}
  \grad u(\rK(s,\eta,\zeta))\\
  \,=\, \JrK^{-1}(s,\eta,\zeta) \biggl(
  \frac{\di u}{\di s} 
  + \Bigl(\tau+\thetaprime\Bigr)(s)
  \begin{bmatrix}
    \zeta \\ -\eta
  \end{bmatrix} \cdot \crossgrad u \biggr)\tK(s)
  + \begin{bmatrix}
    \nK(s) & \bK(s)
  \end{bmatrix}
  \Rthetarhon(s) 
  \crossgrad u \,.
\end{multline}

We note that
\begin{equation*}
  |\Drhon|
  \,=\, \int_{-\ell}^\ell \int_{\Drhon'} \JrK(s,\eta,\zeta) \dez \ds
  \,=\, 2\ell|\Drhon'| \bigl( 1 + \Ocal( \kappamax\rhon ) \bigr) \,,
\end{equation*}
and accordingly we obtain from \eqref{eq:Defmu'} that, 
for any $\psi\in C(\ol{\BR})$,
\begin{equation*}
  \begin{split}
    \int_{\BR} \psi\, |\Drhon|^{-1} \chi_{\Drhon} \dx
    &\,=\, \frac{|\Drhon'|}{|\Drhon|} \int_{-\ell}^\ell
    \frac{1}{|\Drhon'|} \int_{\Br}
    \chi_{\Drhon'}(\eta,\zeta) \psi(\rK(s,\eta,\zeta)) 
    \JrK(s,\eta,\zeta) \dez \ds\\
    &\to \frac{1}{2\ell}
    \int_{-\ell}^\ell \int_{\Br}
    \psi(\rK(s,\eta,\zeta)) \dmu' \ds
    \,=\, \frac{1}{2\ell}
    \int_{-\ell}^\ell \psi(\pK(s)) \ds%
  \end{split}
\end{equation*}
as $n\to\infty$.
This means that
\begin{equation}
  \label{eq:Defmu1}
  |\Drhon|^{-1} \chi_{\Drhon} \text{ converges in the sense of
    measures to } \mu \text{ as } n\to\infty \,,
\end{equation}
where $\mu$ is the Borel measure given by
\begin{equation}
  \label{eq:Defmu2}
  \int_{\BR} \psi \dmu
  \,=\, \frac{1}{2\ell} \int_K \psi \ds
  \qquad \text{for any $\psi\in C(\ol{\BR})$} \,.
\end{equation}

The following theorem describes the asymptotic behavior of the
scattered electric field $\Esrhon$ and of the electric far field
pattern $\Einftyrhon$ as the radius $\rhon$ of the scattering object
$\Drhon$ tends to zero. 
The matrix function
\begin{equation*}
  \G(\bfx,\bfy) 
  \,:=\, \Phi_k(\bfx-\bfy)\I_3 
  + \frac{1}{k^2} \grad_\bfx \div_\bfx ( \Phi_k(\bfx-\bfy) \I_3 ) \,, \qquad
  \bfx\not=\bfy \,,
\end{equation*}
where $\I_3\in\Rd$ is the identity matrix and 
$\Phi_k(\bfx-\bfy) := e^{\rmi k |\bfx-\bfy|}/(4\pi|\bfx-\bfy|)$
denotes the fundamental solution of the Helmholtz equation, is called
the \emph{dyadic Green's function} for Maxwell's equations (see, e.g.,
\cite[p.~303]{Mon03}). 

\begin{theorem}
  \label{thm:GeneralAsymptotics}
  Let $K\Subset\BR$ be a simple $C^3$ center curve, and let $r>0$ such
  that the local parametrization in \eqref{eq:LocalCoordsK} is well
  defined. 
  Let $(\rhon)_n\tm (0,r/2)$ be a sequence of radii converging to
  zero, and let $(\Drhon')_n$ be a sequence of measurable
  cross-sections with $\Drhon'\tm\Brhon$ for all~$n\in\N$. 
  Suppose that $(\Drhon)_n\tm\BR$ is the corresponding sequence of
  thin tubular scattering objects as in~\eqref{eq:DefDrho}, where the
  cross-section twists along the center curve subject to a parameter
  dependent rotation matrix $\Rtheta \in C^1([-L,L],\R^{2\times 2})$. 
  Denoting by $(\eps_{\rhon})_n$ and $(\mu_{\rhon})_n$ permittivity and
  permeability distributions as in~\eqref{eq:Defeps/murho}, let
  $\Esrhon$ be the associated scattered electric field solving
  \eqref{eq:ScatteringProblem} for some incident electric field~$\Ei$.
  Then there exists a subsequence, also denoted by $(\Drhon)_n$, and
  matrix valued functions 
  $\Meps,\Mmu \in L^2(K,\Rdd)$ called 
  \emph{electric} and \emph{magnetic polarization tensors},
  respectively, such that
  \begin{multline}
    \label{eq:GenAsyEs}
    \Esrhon(\bfx) 
    \,=\, \frac{|\Drhon|}{2\ell} \biggl( 
    \int_K (\mu_r-1) \curl_x\G(\bfx,\bfy) \Mmu(\bfy) 
    \curl \Ei(\bfy) \ds(\bfy)\\
    +\int_K k^2 (\eps_r-1) \G(\bfx,\bfy) \Meps(\bfy)  
    \Ei(\bfy) \ds(\bfy) \biggr) 
    + o(|\Drhon|) \,, \qquad \bfx \in \Rd\setminus\ol{\BR} \,.
  \end{multline}
  Furthermore, the electric far field pattern satisfies
  \begin{multline}
    \label{eq:GenAsyEinfty}
    \Einftyrhon(\xhat) 
    \,=\,  \frac{|\Drhon|}{2\ell} \biggl(
    \int_K (\mu_r-1) \rmi k e^{-\rmi k \xhat\cdot \bfy}
    (\xhat\times \I_3) \Mmu(\bfy) \curl \Ei(\bfy) \ds(\bfy)\\
    +\int_K k^2 (\eps_r-1) e^{-\rmi k \xhat\cdot \bfy}
    \bigl(\xhat\times(\I_3\times \xhat)\bigr) 
    \Meps(\bfy) \Ei(\bfy) \ds(\bfy)
    \biggr)
    + o(|\Drhon|) \,, \qquad \xhat\in\Stwo \,.
  \end{multline}
  The subsequence $(\Drhon)_n$ and the polarization tensors $\Meps$
  and $\Mmu$ are independent of the incident electric field $\Ei$.
  The terms $o(|\Drhon|)$ in \eqref{eq:GenAsyEs} and
  \eqref{eq:GenAsyEinfty} are such that 
  $\|o(|\Drhon|)\|_{L^\infty(\di\BR)}/|\Drhon|$ and 
  $\|o(|\Drhon|)\|_{L^\infty(\Stwo)}/|\Drhon|$ converge to zero
  uniformly for all $\Ei$ satisfying $\|\Ei\|_{H(\curl;\BR)}\leq C$
  for some fixed $C>0$.
\end{theorem}

\begin{proof}%
  An analysis similar to \cite{AmmVol05,AmmIakMos03}, using the
  asymptotic perturbation formula for the Maxwell boundary value
  problem from \cite{AlbCap18,Gri11} instead of \cite{AmmVogVol01},
  and applying \eqref{eq:Defmu2} gives the result. 
\end{proof}

All components of the leading order terms in the asymptotic
representation formulas \eqref{eq:GenAsyEs}
and~\eqref{eq:GenAsyEinfty}, except for the polarization tensors 
$\Meps,\Mmu \in L^2(K,\Rdd)$, are either known explicitly
or can be evaluated straightforwardly. 
The polarization tensors are
defined as follows (see~\cite{CapVog03a,CapVog06,Gri11}). 
Let $\gamma \in \{\eps,\mu\}$.
For $\bfxi\in\Stwo$ and $n\in\N$ let $\Wrhonxi\in H^1_0(\BR)$ be
the \emph{corrector potentials} satisfying
\begin{equation}
  \label{eq:DefWrhoxi}
  \div\bigl(\gamma_{\rhon} \grad \Wrhonxi\bigr) 
  \,=\, -\div\bigl((\gamma_{\rhon}-\gamma_0) \bfxi\bigr) \quad 
  \text{in $\BR$} \,, \qquad 
  \Wrhonxi \,=\, 0 \quad \text{on $\di\BR$} \,.
\end{equation}
Then, considering the subsequence $(\Drhon)_n$ from
Theorem~\ref{thm:GeneralAsymptotics}, the polarization tensor $\Mgamma$
is uniquely determined by 
\begin{equation}
  \label{eq:PolTenWxiconst}
  \frac{1}{2\ell} \int_K \bfxi\cdot  \Mgamma \bfxi \psi \ds
  \,=\, \frac{1}{|\Drhon|} \int_{\Drhon} |\bfxi|^2 \psi \dx
  + \frac{1}{|\Drhon|} \int_{\Drhon} 
  \bigl(\bfxi\cdot\grad\Wrhonxi \bigr)\psi \dx
  + o(1)
\end{equation}
for all $\psi\in C(\ol{\BR})$ and any $\bfxi\in\Stwo$.
Similar notions of polarization tensors appear in various contexts.
The term was introduced by Polya, Schiffer and Szeg\"o
\cite{PolSze51,SchSze49}, and they have been widely studied in the
theory of homogenization as the low volume fraction limit of the
effective properties of the dilute two phase composites (see, e.g.,
\cite{Lip93,Mil02,MovSer97}). 
For the specific form considered here, it has been shown in
\cite{CapVog03a,CapVog06} that the values of the functions $\Mgamma$
are symmetric and positive definite in the sense that 
\begin{equation}
  \label{eq:MgammaSymmetry}
  \Mgamma_{ij}(\bfx) \,=\, \Mgamma_{ji}(\bfx) \qquad
  \text{for $1\leq i,j\leq 3$ and a.e.\ } \bfx\in K \,,
\end{equation}
and
\begin{equation}
  \label{eq:MgammaBounds}
  \min \Bigl\{ 1, \frac{\gamma_0(\bfx)}{\gamma_1(\bfx)} \Bigr\}
  \,\leq\, \bfxi \cdot \Mgamma(\bfx) \bfxi 
  \,\leq\, \max \Bigl\{ 1, \frac{\gamma_0(\bfx)}{\gamma_1(\bfx)}
  \Bigr\} \qquad
  \text{for every $\bfxi\in\Stwo$ and a.e.\ $\bfx\in K$} \,.
\end{equation}
Analytic expressions for $\Mgamma$ have been derived for several basic
geometries such as, e.g., when~$(\Drhon)_n$ is a family of
diametrically small ellipsoids (see \cite{AmmKan07}), a family of thin
neighborhoods of a hypersurface (see \cite{BerFraVog03}), or a family
of thin neighborhoods of a straight line segment (see
\cite{BerCapdeGFra09}).  
In the next section we extend the result from \cite{BerCapdeGFra09} to 
thin neighborhoods of smooth curves  of $(\Drhon)_n$ as in
\eqref{eq:DefDrho}, and we derive a spectral representation of the
polarization tensor in terms of the center curve $K$ and the
two-dimensional polarization tensor of the cross-sections
$(\Drhon')_n$. 
In \cite{BerCapdeGFra09,Dap20,GriHyv11}, the authors expressed
interest such a characterization of the polarisation tensor for thin
tubular objects for various applications. 
Therewith, the leading order terms in the asymptotic representation
formulas \eqref{eq:GenAsyEs} and \eqref{eq:GenAsyEinfty} can be evaluated
very efficiently. 

\begin{remark}
  \label{rem:SmallerDomain}
  The characterization of the polarization tensor 
  $\Mgamma\in L^2(K;\Rdd)$ in \eqref{eq:PolTenWxiconst}
  remains valid when the domain $\BR$ in \eqref{eq:DefWrhoxi} is
  replaced by $\Omegar$ from \eqref{eq:DefOmegar} (see
  \cite[Rem.~1]{CapVog06}). 
  The regularity results that are used in the proof of
  \cite[Lmm.~1]{CapVog06} are applicable because $\Omegar$ is $C^2$
  away from the ends of the tube and convex in a neighborhood of the
  ends of the tube. 
  This will be used in Section~\ref{sec:PolarizationTensor} below.
  \hfill$\lozenge$
\end{remark}

\section{The polarization tensor of a thin tubular scattering object}
\label{sec:PolarizationTensor}
Let $\gamma\in\{\eps,\mu\}$.
We assign a two-dimensional polarization tensor 
${\mgamma\in\R^{2\times 2}}$ to the sequence of cross-sections
$(\Drhon')_n$ of the scattering objects $(\Drhon)_n$ as follows. 
Let
\begin{equation*}
  \gamma_{\rhon}'(\eta,\zeta) \,:=\,
  \begin{cases}
    \gamma_1\,, & (\eta,\zeta) \in \Drhon'\,,\\
    \gamma_0\,, & (\eta,\zeta) \in \Br\setminus\overline{\Drhon'} \,,
  \end{cases}
\end{equation*}
i.e., $\gamma_{\rhon}'$ is just the electric permittivity or the
magnetic permeability distribution associated to the cross-section
$\Drhon'$. 
For each $\bfxi'\in \Sone$ we denote by $\wrhonxi \in H^1_0(\Br)$ the 
unique solution to 
\begin{subequations}
  \label{eq:Defwrhoxi}
  \begin{align}
    \crossdiv \bigl( \gamma_{\rhon}' \crossgrad \wrhonxi \bigr) 
    &\,=\, -\crossdiv\bigl((\gamma_1-\gamma_0) 
      \chi_{\Drhon'} \bfxi'\bigr) &&
                                     \text{in $\Br$} \,, \\
    \wrhonxi &\,=\, 0 && \text{on $\di\Br$} \,,
  \end{align}
\end{subequations}
and accordingly we define $\mgamma\in\R^{2\times 2}$ (possibly up to
extraction of a subsequence) by
\begin{equation}
  \label{eq:PolTen2D}
  \bfxi' \cdot \mgamma \bfxi' \, \psi(0)
  \,=\, \frac{1}{|\Drhon'|} \int_{\Drhon'} |\bfxi'|^2 \psi \dx'
  + \frac{1}{|\Drhon'|} \int_{\Drhon'} 
  \bigl(\bfxi'\cdot\crossgrad\wrhonxi\bigr)\psi \dx'
  + o(1)
\end{equation}
for all $\psi\in C(\ol{\Br})$ and any $\bfxi'\in\Sone$.

The following theorem is the main result of this section.

\begin{theorem}
  \label{thm:CharacterizationPolTen}
  Let $K\Subset\BR$ be a simple $C^3$ center curve, and let $r>0$ such
  that the local parametrization in \eqref{eq:LocalCoordsK} is well
  defined. 
  Let $(\rhon)_n\tm (0,r/2)$ be a sequence of radii converging to
  zero, and let $(\Drhon')_n$ be a sequence of measurable
  cross-sections with $\Drhon'\tm\Brhon$ for all~$n\in\N$. 
  Suppose that $(\Drhon)_n\tm\BR$ is the corresponding sequence of
  thin tubular scattering objects as in~\eqref{eq:DefDrho}, where the
  cross-section twists along the center curve subject to a parameter
  dependent rotation matrix $\Rtheta \in C^1([-L,L],\R^{2\times 2})$. 
  Denoting by $(\gamma_{\rhon})_n$ a parameter distribution as
  in~\eqref{eq:Defeps/murho}, let~$\Mgamma$ be the polarization tensor
  corresponding to the thin tubular scattering objects~$(\Drhon)_n$
  from~\eqref{eq:PolTenWxiconst} (defined possibly up to extraction of
  a subsequence). 
  Denoting by $\mgamma$ the polarization tensor corresponding to the
  cross-sections~$(\Drhon')_n$ from \eqref{eq:PolTen2D} (defined
  possibly up to extraction of a subsequence), the following pointwise 
  characterization of $\Mgamma$ holds for a.e.\ $s\in(-\ell,\ell)$:
  \begin{enumerate}[(a)]
  \item The unit tangent vector $\tK(s)$ is an eigenvector of the
    matrix $\Mgamma(\pK(s))$ corresponding to the eigenvalue $1$, i.e., 
    \begin{equation}
      \label{eq:CharacterizationPolTen(a)}
      \tK(s) \cdot \Mgamma(\pK(s)) \tK(s) \,=\, 1 \qquad
      \text{for a.e.\ $s\in(-\ell,\ell)$}\,.
    \end{equation}
  \item Let $\bfxi'\in \Sone$, and let $\bfxi\in C^1(K,\Stwo)$ be
    given by $\bfxi(s) :=
    \begin{bmatrix}
      \nK(s) & \bK(s)
    \end{bmatrix} \bfxi' \in \Stwo$ for
    all~${s\in(-\ell,\ell)}$. 
    Then,
    \begin{equation}
      \label{eq:CharacterizationPolTen(b)}
      \bfxi(s) \cdot \Mgamma(\pK(s)) \bfxi(s) 
      \,=\, \bfxi' \cdot 
      \bigl( \Rthetarhon(s) \mgamma 
      \Rthetarhon^{-1}(s)\bigr)\bfxi' \qquad 
      \text{for a.e.\ $s\in(-\ell,\ell)$}\,.
    \end{equation}
  \end{enumerate}
\end{theorem}

Since the polarization tensor $\Mgamma(\pK(s))$ is symmetric, the
first part of the theorem implies that there are two more eigenvalues
in the plane orthogonal to $\tK(s)$, which is spanned by $\nK(s)$
and~$\bK(s)$. 
The second part of the theorem says that in this plane the
polarization tensor $\Mgamma(\pK(s))$ coincides with the polarization
tensor $\Rthetarhon(s) \mgamma \Rthetarhon^{-1}(s)$ of the twisted
two-dimensional cross-sections. 

The proof of Theorem~\ref{thm:CharacterizationPolTen} relies on the
following proposition, which extends the characterization of the 
polarization tensor $\Mgamma$ in \eqref{eq:PolTenWxiconst} from
constant vectors $\bfxi\in\Stwo$ to vector-valued functions 
$\bfxi\in C^1(\Omegar,\Stwo)$. 

\begin{proposition}
  \label{pro:PolTenWxi}
  Let $\bfxi\in C^1(\Omegar,\Stwo)$, and denote by 
  $\Wrhonxi \in H^1_0(\Omegar)$ the corresponding solution
  to~\eqref{eq:DefWrhoxi}. 
  Then the polarization tensor $\Mgamma$ satisfies 
  \begin{equation*}
    \frac{1}{2\ell} \int_K \bfxi\cdot \Mgamma \bfxi \psi \ds
    \,=\, \frac{1}{|\Drhon|} \int_{\Drhon} |\bfxi|^2 \psi \dx
    + \frac{1}{|\Drhon|} \int_{\Drhon} 
    \bigl(\bfxi\cdot\grad\Wrhonxi\bigr)\psi \dx
    + o(1)
  \end{equation*}
  for all $\psi\in C(\ol{\Omegar})$.
\end{proposition}

\begin{proof}
  We denote by $(\bfe_1,\bfe_2,\bfe_3)$ the standard basis of $\Rd$,
  and we consider 
  $\bfxi = \sum_{i=1}^3 \xi_i \bfe_i \in C^1(\Omegar,\Stwo)$.
  Let $\Wrhonxi \in H^1_0(\Omegar)$ be the corresponding solutions
  to~\eqref{eq:DefWrhoxi}, and let $\Wrhonej$, $1\leq j\leq 3$, be the
  solutions to \eqref{eq:DefWrhoxi} with $\bfxi=\bfe_j$.  
  Then, using \eqref{eq:PolTenWxiconst} we find that
  \begin{equation*}
    \begin{split}
      &\frac{1}{2\ell} \int_K 
      \bfxi(\bfx)\cdot \Mgamma(\bfx) \bfxi(\bfx) \psi(\bfx) \ds(\bfx)
      \,=\, \sum_{i,j=1}^3 \frac{1}{2\ell} \int_K 
      \bfe_i\cdot  \Mgamma(\bfx) \bfe_j 
      \bigl( \xi_i\xi_j\psi \bigr)(\bfx) \ds(\bfx)\\
      &\,=\, \sum_{i,j=1}^3 \frac{1}{|\Drhon|} \int_{\Drhon} 
      \bfe_i\cdot  \bfe_j 
      \bigl( \xi_i\xi_j\psi \bigr)(\bfx) \dx
      + \sum_{i,j=1}^3 \frac{1}{|\Drhon|} \int_{\Drhon} 
      \bfe_i\cdot  \grad \Wrhonej(\bfx)
      \bigl( \xi_i\xi_j\psi \bigr)(\bfx) \dx
      + o(1)\\
   &\,=\, \frac{1}{|\Drhon|} \int_{\Drhon} |\bfxi(\bfx)|^2 \psi(\bfx) \dx
      + \frac{1}{|\Drhon|} \int_{\Drhon} \bfxi(\bfx)\cdot
      \grad \Wrhonxi(\bfx) \psi(\bfx) \dx\\
      &\phantom{\,=\,}
      - \frac{1}{|\Drhon|} \int_{\Drhon} \bfxi(\bfx)\cdot 
      \Bigl( \grad \Wrhonxi 
      - \sum_{j=1}^3\xi_j\grad \Wrhonej \Bigr)(\bfx) \psi(\bfx) \dx
      + o(1) \,.
    \end{split}
  \end{equation*}
  Applying H\"older's inequality gives
  \begin{equation}
    \label{eq:ProofPolTenWxi1}
    \biggl|
    \int_{\Drhon} \bfxi\cdot
    \Bigl( \grad \Wrhonxi 
    - \sum_{j=1}^3\xi_j\grad \Wrhonej \Bigr) \psi \dx
    \biggr|
    \,\leq\, C |\Drhon|^{\frac12} \Bigl\| \grad \Wrhonxi 
    - \sum_{j=1}^3\xi_j\grad \Wrhonej \Bigr\|_{L^2(\Omegar)} \,.
  \end{equation}
  To finish the proof, we show that the right hand side of
  \eqref{eq:ProofPolTenWxi1} is $o(|\Drhon|)$ as $n\to\infty$. 

  We note that \eqref{eq:DefWrhoxi} gives
  \begin{equation*}
    \begin{split}
      \div \Bigl( \gamma_{\rhon} \grad \Bigl( \Wrhonxi 
      - \sum_{j=1}^3\xi_j\Wrhonej \Bigr) \Bigr)
      &\,=\, -\div \bigl((\gamma_{\rhon}-\gamma_0)\bfxi\bigr)\\
      &\phantom{\,=\,}
      - \sum_{j=1}^3 
      \div\bigl(\bigl(\gamma_{\rhon}\nabla \Wrhonej\bigr)\xi_j\bigr)
      - \sum_{j=1}^3 \div\bigl(\gamma_{\rhon} \Wrhonej\nabla \xi_j\bigr) \,.
    \end{split}
  \end{equation*}
  Furthermore,
  \begin{equation*}
    \begin{split}
      \sum_{j=1}^3 \div\bigl(
      \bigl(\gamma_{\rhon}\nabla \Wrhonej\bigr)\xi_j\bigr)
      &\,=\, -\sum_{j=1}^3 \div 
      \bigl((\gamma_{\rhon}-\gamma_0)\bfe_j\bigr) \xi_j
      + \sum_{j=1}^3 \gamma_{\rhon} \grad \Wrhonej \cdot \nabla \xi_j \,,
    \end{split}
  \end{equation*}
  and rewriting $\bfxi=\sum_{j=1}^3 \xi_j\bfe_j$ we obtain that
  \begin{equation*}
    \div\bigl( (\gamma_{\rhon}-\gamma_0)\bfxi \bigr)
    \,=\, \sum_{j=1}^3 \div \bigl( (\gamma_{\rhon}-\gamma_0)
    \bfe_j \bigr) \xi_j 
    + \sum_{j=1}^3 (\gamma_{\rhon}-\gamma_0) \bfe_j\cdot \grad \xi_j \,.
  \end{equation*} 
  Accordingly, $\Wrhonxi - \sum_{j=1}^3\xi_j\Wrhonej \in H^1_0(\Omegar)$
  satisfies 
  \begin{equation*}
    \begin{split}
      \div \Bigl( \gamma_{\rhon} \grad \Bigl( \Wrhonxi 
      - \sum_{j=1}^3\xi_j\Wrhonej \Bigr) \Bigr)
      &\,=\, - \sum_{j=1}^3 (\gamma_{\rhon}-\gamma_0)
      \bfe_j\cdot \grad \xi_j
      - \sum_{j=1}^3 \gamma_{\rhon} \grad \Wrhonej \cdot \nabla \xi_j\\
      &\phantom{\,=\,}
      - \sum_{j=1}^3 \div\bigl(\gamma_{\rhon} \Wrhonej\nabla \xi_j\bigr) \,.
    \end{split}
  \end{equation*}

  Now let $\rrhon^{(1)},\rrhon^{(2)},\rrhon^{(3)}\in H^1_0(\Omegar)$ be
  the unique solutions to 
  \begin{subequations}
    \label{eq:PDErrho}
    \begin{align}
      \div(\gamma_{\rhon} \grad \rrhon^{(1)}) 
      &\,=\, -\sum_{j=1}^3 \div(\gamma_{\rhon} \Wrhonej\nabla\xi_j)
      &&\text{in $\Omegar$} \,, \qquad
      &&\rrhon^{(1)} 
         \,=\, 0 \qquad \text{on $\di\Omegar$} \,,\label{eq:PDErrho1}\\
      \div(\gamma_{\rhon} \grad \rrhon^{(2)}) 
      &\,=\, -\sum_{j=1}^3 \gamma_{\rhon} 
        \grad\Wrhonej \cdot \nabla\xi_j
      &&\text{in $\Omegar$} \,,  \qquad
      &&\rrhon^{(2)} 
         \,=\, 0 \qquad \text{on $\di\Omegar$} \,,\label{eq:PDErrho2}\\
      \div(\gamma_{\rhon} \grad \rrhon^{(3)}) 
      &\,=\, -\sum_{j=1}^3 (\gamma_{\rhon}-\gamma_0)
        (\bfe_j\cdot \grad\xi_j)
      &&\text{in $\Omegar$} \,,  \qquad
      &&\rrhon^{(3)} 
         \,=\, 0 \qquad \text{on $\di\Omegar$} \,.\label{eq:PDErrho3}
    \end{align}
  \end{subequations}
  The uniqueness of solutions to the Dirichlet problem implies that
  \begin{equation*}
    \Wrhonxi - \sum_{j=1}^3 \xi_j\Wrhonej 
    \,=\, \rrhon^{(1)} + \rrhon^{(2)} + \rrhon^{(3)} \,,
  \end{equation*}
  and we have the following estimates for 
  $\rrhon^{(1)}, \rrhon^{(2)}$, and $\rrhon^{(3)}$.
  Using the well-posedness of~\eqref{eq:PDErrho1}
  and~\eqref{eq:Estwq2} we find that 
  \begin{equation*}
    \begin{split}
      \|\rrhon^{(1)}\|_{H^1(\Omegar)}
      &\,\leq\, C \Bigl\| \sum_{j=1}^3 \gamma_{\rhon} \Wrhonej 
      \grad \xi_j \Bigr\|_{L^2(\Omegar)}
      \,\leq\, C \max_j\|\grad\xi_j\|_{L^\infty(\Omegar)}
      \|\Wrhonej\|_{L^2(\Omegar)}\\
      &\,\leq\, C \|\bfxi\|_{C^1(\Omegar)}
      |\Drhon|^{\frac34} \,.
    \end{split}
  \end{equation*}
  Similarly, using Poincare's inequality, the weak formulation of
  \eqref{eq:PDErrho2}, 
  H\"older's inequality, Sobolev's embedding theorem (see, e.g.,
  \cite[p.~158]{GilTru01}), and \eqref{eq:Estwq3} we obtain that
  \begin{equation*}
    \begin{split}
      &\|\rrhon^{(2)}\|_{H^1(\Omegar)}^2
      \,\leq\, C \Bigl| \int_\Omegar \gamma_{\rhon} \grad \rrhon^{(2)} 
      \cdot \grad \rrhon^{(2)} \dx \Bigr|
      \,=\, C \Bigl| \int_\Omegar \sum_{j=1}^3 \gamma_{\rhon} 
      \bigl( \grad\Wrhonej \cdot \nabla\xi_j \bigr) \rrhon^{(2)} 
      \dx \Bigr|\\
      &\,\leq\, C \Bigl\|\sum_{j=1}^3\gamma_{\rhon}
      \grad\Wrhonej\Bigr\|_{L^{\frac{6}{5}}(\Omegar)} 
      \max_j\|\nabla\xi_j\|_{L^\infty(\Omegar)} 
      \|\rrhon^{(2)}\|_{L^{6}(\Omegar)}\\
      &\,\leq\, C |\Drhon|^{\frac56}
      \|\bfxi\|_{C^1(\Omegar)} 
      \|\grad \rrhon^{(2)}\|_{L^2(\Omegar)}
      \,\leq\, C |\Drhon|^{\frac56}
      \|\bfxi\|_{C^{1}(\Omegar)} 
      \|\rrhon^{(2)}\|_{H^1(\Omegar)}\,.
    \end{split}
  \end{equation*}
  For the third term $\rrhon^{(3)}$ we note that, using Poincare's
  inequality, the weak formulation of \eqref{eq:PDErrho3}, H\"older's
  inequality, and Sobolev's embedding theorem, 
  \begin{equation*}
    \begin{split}
      \|\rrhon^{(3)}\|_{H^1(\Omegar)}^2
      &\,\leq\, C \Bigl| \int_\Omegar \gamma_{\rhon} 
      \grad\rrhon^{(3)}\cdot \grad\rrhon^{(3)} \dx \Bigr|
      \,=\, C \Bigl| \int_\Omegar (\gamma_{\rhon}-\gamma_0) 
      \sum_{j=1}^3(\bfe_j\cdot\grad\xi_j) \rrhon^{(3)} \dx \Bigr|\\
      &\,\leq\, C \|\gamma_{\rhon}-\gamma_0\|_{L^\frac65(\Omegar)} 
      \max_j\|\grad\xi_j\|_{L^\infty(\Omegar)}
      \|\rrhon^{(3)}\|_{L^6(\Omegar)}\\
      &\,\leq\, C |\Drhon|^{\frac56} 
      \|\bfxi\|_{C^1(\Omegar)}
      \|\grad\rrhon^{(3)}\|_{L^2(\Omegar)}
      \,\leq\, C |\Drhon|^{\frac56} 
      \|\bfxi\|_{C^1(\Omegar)}
      \|\rrhon^{(3)}\|_{H^1(\Omegar)} \,.
    \end{split}
  \end{equation*}
  Accordingly,
  \begin{equation*}
     \Bigl\| \Wrhonxi - \sum_{j=1}^3 \xi_j\Wrhonej \Bigr\|_{H^1(\Omegar)}
     \,\leq\, C \|\bfxi\|_{C^1(\Omegar)} |\Drhon|^{\frac34} \,,
   \end{equation*}
  and, using \eqref{eq:Estwq2},
  \begin{multline*}
    \Bigl\| \grad\Wrhonxi 
    - \sum_{j=1}^3 \xi_j\grad\Wrhonej \Bigr\|_{L^2(\Omegar)}
    \,\leq\, \Bigl\| \grad\Wrhonxi 
    - \grad\Bigl(\sum_{j=1}^3 \xi_j\Wrhonej\Bigr) \Bigr\|_{L^2(\Omegar)}
    + \Bigl\| \sum_{j=1}^3 \Wrhonej \grad\xi_j \Bigr\|_{L^2(\Omegar)}\\
    \,\leq\, C \|\bfxi\|_{C^1(\Omegar)} |\Drhon|^{\frac34}
    + \|\Wrhonej\|_{L^2(\Omegar)} \max_j\|\grad\xi_j\|_{L^\infty(\Omegar)}
    \,\leq\, C \|\bfxi\|_{C^1(\Omegar)} |\Drhon|^{\frac34} \,.
  \end{multline*}
\end{proof}

Next we prove the first part of
Theorem~\ref{thm:CharacterizationPolTen}. 

\begin{proof}[Proof of Theorem~\ref{thm:CharacterizationPolTen}(a)]
  Let $\bfxi\in C^1(\Omegar)$ be defined by $\bfxi(\bfx) := \tK(s)$ for
  any $\bfx = \rK(s,\eta,\zeta)\in\Omegar$.
  Using Proposition~\ref{pro:PolTenWxi} we find that, for any 
  $\psi\in C^1(\ol{\Omegar})$,
  \begin{multline}
    \label{eq:ProofChar1}
    \frac{1}{2\ell} \int_{-\ell}^{\ell} \tK(s) 
    \cdot \Mgamma(\pK(s)) \tK(s) \psi(\pK(s))
    \ds\\
    \,=\, \frac1{|\Drhon|}\int_{\Drhon} \psi(\bfx) \dx
    + \frac1{|\Drhon|}\int_{\Drhon} 
    \bigl( \bfxi(\bfx) \cdot \grad \Wrhonxi(\bfx) \bigr) 
    \psi(\bfx) \dx
    + o(1) \,.
  \end{multline}
  Working in local coordinates, recalling
  \eqref{eq:LocalCoordsKJacobian}--\eqref{eq:LocalCoordsKGradient}, 
  and integrating by parts we obtain that
  \begin{equation*}
    \begin{split}
      &\biggl| \int_{\Drhon} 
      \bigl( \bfxi(\bfx) \cdot \grad \Wrhonxi(\bfx) \bigr) 
      \psi(\bfx) \dx \biggr|\\
      &\,=\, \biggl| \int_{-\ell}^{\ell} \int_{\Drhon'} \!\!\!
      \Bigl( \frac{\Wrhonxi}{\di s}(\rK(s,\eta,\zeta))
      +\Bigl(\tau+\thetarhonprime\Bigr)(s) \!
      \begin{bmatrix}
        \zeta \\ -\eta
      \end{bmatrix} \cdot \crossgrad 
      \Wrhonxi(\rK(s,\eta,\zeta)) \Bigr) 
      \psi(\rK(s,\eta,\zeta)) \!\dez \!\ds \biggr|\\
      &\,\leq\, \biggl| \biggl[ 
      \int_{\Drhon'} \Wrhonxi(\rK(s,\eta,\zeta)) 
      \psi(\rK(s,\eta,\zeta)) \dez \biggr]_{s=-\ell}^{\ell} \biggr|\\
      &\phantom{\,=\,}
      + \biggl| \int_{\Drhon'} \int_{-\ell}^{\ell} \Wrhonxi(\rK(s,\eta,\zeta)) 
      \frac{\di }{\di s}\psi(\rK(s,\eta,\zeta)) \ds \dez \biggr|\\
      &\phantom{\,=\,}
      + \biggl| \int_{-\ell}^{\ell} \int_{\Drhon'} 
      \Bigl(\tau+\thetarhonprime\Bigr)(s) \begin{bmatrix}
        \zeta \\ -\eta
      \end{bmatrix} \cdot \crossgrad 
      \Wrhonxi(\rK(s,\eta,\zeta)) 
      \psi(\rK(s,\eta,\zeta)) \dez \ds  \biggr|\,.
    \end{split}
  \end{equation*}
  Using the interior regularity estimate \cite[Thm.~8.24]{GilTru01}
  and \eqref{eq:Estwq2} gives
  \begin{equation*}
    \|\Wrhonxi\|_{L^\infty(\Omegar)}
    \,\leq\, C \bigl( \|\Wrhonxi\|_{L^2(\Omegar)} 
    + \|(\gamma_{\rhon}-\gamma_0)\bfxi\|_{L^4(\Omegar)} \bigr)
    \,\leq\, C |\Drhon|^{\frac14} \,.
  \end{equation*}
  Therefore, using \eqref{eq:Estwq1}--\eqref{eq:Estwq2}, 
  \begin{equation*}
     \begin{split}
      &\biggl| \int_{\Drhon} 
      \bigl( \bfxi(\bfx) \cdot \grad \Wrhonxi(\bfx) \bigr) 
      \psi(\bfx) \dx \biggr|\\
      &\,\leq\, C |\Drhon'| \|\Wrhonxi\|_{L^\infty(\Drhon)} 
      + C |\Drhon|^{\frac12} \|\Wrhonxi\|_{L^2(\Drhon)} 
      + C \rhon |\Drhon|^{\frac12} 
      \|\grad \Wrhonxi\|_{L^2(\Omegar)}\\
      &\,\leq\, C |\Drhon| |\Drhon|^{\frac14}
      + C |\Drhon|^{\frac12} |\Drhon|^{\frac34}
      + C \rhon |\Drhon|^{\frac12} |\Drhon|^{\frac12} \
      \,=\, o(|\Drhon|) \,.
    \end{split}
  \end{equation*}

  Inserting this estimate into \eqref{eq:ProofChar1}, using
  \eqref{eq:Defmu1}--\eqref{eq:Defmu2}, and letting $n\to\infty$, we
  obtain that 
  \begin{equation*}
    \frac{1}{2\ell} \int_{-\ell}^{\ell} 
    \tK(s) \cdot \Mgamma(\pK(s)) \tK(s) \psi(\pK(s))
    \ds
    \,=\, \frac{1}{2\ell} \int_{-\ell}^{\ell} 
    \psi(\pK(s)) \ds \,.
  \end{equation*}
  Since $\psi\in C^1(\ol{\Omegar})$ was arbitrary, this implies
  \eqref{eq:CharacterizationPolTen(a)}. 

  Recalling the symmetry of $\Mgamma(\pK(s))$ in
  \eqref{eq:MgammaSymmetry} and the polarization tensor
  bounds~\eqref{eq:MgammaBounds} shows that~$1$~is either the
  maximal or minimal eigenvalue of $\Mgamma(\pK(s))$ for a.e.\
  $s\in(-\ell,\ell)$, and that $\tK(s)$ is the corresponding
  eigenvector. 
\end{proof}

Next we prove the second part of
Theorem~\ref{thm:CharacterizationPolTen}. 

\begin{proof}[Proof of Theorem~\ref{thm:CharacterizationPolTen}(b)]
  Let $\bfxi' \in \Sone$, and let $\bfxi\in C^1(\Omegar,\Rd)$ be
  defined by 
  \begin{equation*}
    \bfxi(\bfx) 
    \,:=\, 
    \begin{bmatrix}
      \nK(s) & \bK(s)
    \end{bmatrix}
    \bfxi'
    \qquad \text{for any } \bfx = \rK(s,\eta,\zeta)\in\Omegar \,,
  \end{equation*}
  and let $\Wrhonxi \in H^1_0(\Omegar)$ be the corresponding solution 
  to~\eqref{eq:DefWrhoxi}. 

  The main idea of this proof is to approximate $\grad\Wrhonxi$ by the
  gradient of a product of functions involving the solution 
  $\wrhonxitheta\in H^1_0(\Br)$ of \eqref{eq:Defwrhoxi} with $\bfxi'$ 
  replaced by $\Rtheta^{-1}\bfxi'$.
  To do so we first introduce a modified corrector potential
  $\wrhonxithetatilde\in H^1_0(\Br)$ as the unique solution to 
  \begin{subequations}
    \label{eq:Defwrhoxitilde}
    \begin{align}
      \crossdiv \Bigl( (\I_2+A_\tor) \gamma_{\rhon}'
      \crossgrad \wrhonxithetatilde \Bigr) 
      &\,=\, -\crossdiv\bigl((\gamma_1-\gamma_0) 
        \chi_{\Drhon'} R_\theta^{-1}\bfxi'\bigr) &&
                                                    \text{in $\Br$} \,, \\
      \wrhonxithetatilde &\,=\, 0 && \text{on $\di\Br$} \,,
    \end{align}
  \end{subequations}
  where
  \begin{equation}
    \label{eq:DefAtor}
    A_\tor(s,\eta,\zeta) 
    \,:=\, \JrK^{-2}(s,\eta,\zeta) 
    \Bigl( \tau+\thetarhonprime \Bigr)^2(s) 
    \begin{bmatrix}
      \zeta^2 & -\eta\zeta\\ -\eta\zeta &\eta^2
    \end{bmatrix}
    \,, \qquad s\in(-L,L) \,,\; 
    (\eta,\zeta) \in \Br \,.
  \end{equation}
  The term $A_\tor$ will be used to account for the twisting of the
  cross-sections along the center curve~$K$ in the estimates below. 
  We note that for any $s\in(-\ell,\ell)$ the matrix $A_\tor(s)$
  is symmetric and positive semi-definite, and
  therefore \eqref{eq:Defwrhoxitilde} has a unique solution.
  Both $\wrhonxitheta$ and $\wrhonxithetatilde$ depend on the
  parameter $s\in(-L,L)$ although we do not indicate this through our
  notation. 

  We define
  \begin{equation*}
    \Wrhonxitilde(\rK(s,\eta,\zeta))
    \,:=\, \frhon(s) \JrK^{-1}(s,\eta,\zeta) \wrhonxithetatilde(\eta,\zeta) \,,
    \qquad s\in (-L,L) \,,\; (\eta,\zeta)\in \Br \,,
  \end{equation*}
  where $\frhon\in C^1([-L,L])$ is a cut-off function satisfying
  \begin{subequations}
    \label{eq:Propertiesfrho}
    \begin{align}
      &0\leq \frhon\leq 1 \,,\qquad 
        \frhon\chi_{(-\ell,\ell)} \,=\, \chi_{(-\ell,\ell)} \,,\\
      &\|\frhon'\|_{L^2((-L,L))} 
        \,\leq\, C |\Drhon'|^{-\frac{1}{8}} \,, \qquad
        \|\frhon(1-\chi_{(-\ell,\ell)})\|_{L^2((-L,L))} 
        \,\leq\, C|\Drhon'|^{\frac{1}{8}}
    \end{align}
  \end{subequations}
  (see \cite[Lmm.~3.6]{BerCapdeGFra09}).
  Using Proposition~\ref{pro:PolTenWxi} we find that, for any 
  $\psi\in C^1(\ol{\Omegar})$,
  \begin{equation}
    \label{eq:ProofThm4.1-1}
    \begin{split}
      \frac{1}{2\ell} \int_{-\ell}^{\ell} \bfxi(s) 
      \cdot \Mgamma(\pK(s))& \bfxi(s) \psi(\pK(s)) \ds\\
      &\,=\, \frac{1}{|\Drhon|} \int_{\Drhon} \psi \dx
      + \frac{1}{|\Drhon|} \int_{\Drhon} 
      \bigl(\bfxi\cdot\grad\Wrhonxi\bigr) \psi \dx
      + o(1)\\
      &\,=\, \frac{1}{|\Drhon|} \int_{\Drhon} \psi \dx
      + \frac{1}{|\Drhon|}
      \int_{\Drhon} \bigl(\bfxi\cdot\grad\Wrhonxitilde\bigr)\psi \dx\\
      &\phantom{\,=\,}
      + \frac{1}{|\Drhon|}
      \int_{\Drhon} \bigl(\bfxi\cdot \bigl( 
      \grad\Wrhonxi - \grad\Wrhonxitilde \bigr) \bigl)\psi \dx
      + o(1) \,.
     \end{split}
  \end{equation}

  We consider the three integrals on the right hand side
  of \eqref{eq:ProofThm4.1-1} separately. 
  Recalling~\eqref{eq:LocalCoordsKJacobian}
  and~\eqref{eq:LocalCoordsKGradient} we obtain that
  \begin{equation*}
    \begin{split}
      &\int_{\Drhon} \bigl(\bfxi\cdot\grad\Wrhonxitilde\bigr)
      \psi \dx\\
      &\,=\, \int_{-\ell}^{\ell} \int_{\Drhon'} 
      \Bigl( \bfxi(\rK(s,\eta,\zeta))
      \cdot \grad\Wrhonxitilde (\rK(s,\eta,\zeta) 
      \Bigr) 
      \psi(\rK(s,\eta,\zeta)) \JrK(s,\eta,\zeta) \dez \ds\\
      &\,=\, \int_{-\ell}^{\ell} \int_{\Drhon'} 
      \Bigl( \bfxi'
      \cdot \Bigl(\Rthetarhon(s)
      \crossgrad
      \Bigl(\JrK^{-1}(s,\eta,\zeta)\wrhonxithetatilde(\eta,\zeta)\Bigr) 
      \Bigr) \Bigr) 
      \psi(\rK(s,\eta,\zeta)) \JrK(s,\eta,\zeta) \dez \ds\\
      &\,=\, \int_{-\ell}^{\ell} \int_{\Drhon'} 
      \Bigl( \bigl( \Rthetarhon^{-1}(s)\bfxi'\bigr) 
      \cdot \crossgrad\wrhonxithetatilde(\eta,\zeta) 
      \Bigr)
      \psi(\rK(s,\eta,\zeta)) \dez \ds\\
      &\phantom{\,=\,}
      + \Ocal\Bigl( |\Drhon|^{\frac12} 
      \bigl\|\wrhonxithetatilde\bigr\|_{L^2(\Drhon')} \Bigr) \,,
    \end{split}
  \end{equation*}
  and applying \eqref{eq:Estwq2} and Lemma~\ref{lmm:Caccioppoli} gives
  \begin{equation*}
    \begin{split}
      &\int_{\Drhon} \bigl(\bfxi\cdot\grad\Wrhonxitilde\bigr)
      \psi \dx\\
      &\,=\, \int_{-\ell}^{\ell} \int_{\Drhon'} 
      \Bigl( \bigl( \Rthetarhon^{-1}(s)\bfxi'\bigr) 
      \cdot \crossgrad\wrhonxitheta(\eta,\zeta) 
      \Bigr)
      \psi(\rK(s,\eta,\zeta)) \dez \ds\\
      &\phantom{\,=\,}
      + \Ocal\Bigl( |\Drhon|^{\frac12} 
      \bigl\|\crossgrad\wrhonxitheta
      -\crossgrad\wrhonxithetatilde\bigr\|_{L^2(\Drhon)} \Bigr)
      + o(|\Drhon|)\\
      &\,=\, \int_{-\ell}^{\ell} \int_{\Drhon'} 
      \Bigl( \bigl( \Rthetarhon^{-1}(s)\bfxi'\bigr) 
      \cdot \crossgrad\wrhonxitheta(\eta,\zeta) 
      \Bigr)
      \psi(\rK(s,\eta,\zeta)) \dez \ds
      + o(|\Drhon|) \,.
    \end{split}
  \end{equation*}
  Accordingly, using \eqref{eq:PolTen2D} we obtain for the first two
  terms in \eqref{eq:ProofThm4.1-1} that 
  \begin{equation}
    \label{eq:ProofThm4.1-b_last}
    \begin{split}
      &\frac{1}{|\Drhon|} \int_{\Drhon} \psi \dx
      + \frac{1}{|\Drhon|}
      \int_{\Drhon} \bigl(\bfxi\cdot\grad\Wrhonxi\bigr)\psi \dx\\
      &\,=\, \frac{|\Drhon'|}{|\Drhon|} \int_{-\ell}^{\ell}
      \biggl( \frac{1}{|\Drhon'|} \int_{\Drhon'}
      \psi(\rK(s,\eta,\zeta)) \dez\\
      &\phantom{\,=\,\frac{|\Drhon'|}{|\Drhon|}}
      + \frac{1}{|\Drhon'|} \int_{\Drhon'} 
      \Bigl( \bigl(\Rthetarhon^{-1}(s)\bfxi'\bigr)
      \cdot \crossgrad\wrhonxitheta(\eta,\zeta) \Bigr)
      \psi(\rK(s,\eta,\zeta)) \dez \biggr)\ds 
      + o(1)\\
      &\,\to\, \frac{1}{2\ell} \int_{-\ell}^{\ell}
      \bfxi' \cdot \bigl( \Rthetarhon(s) \mgamma 
      \Rthetarhon^{-1}(s) \bigr) \bfxi' \psi(\pK(s)) \ds
    \end{split}
  \end{equation}
  as $n\to\infty$.

  For the last integral on the right hand side of
  \eqref{eq:ProofThm4.1-1} we find, using H\"older's inequality, that
  \begin{equation*}
    \biggl| \int_{\Drhon} \bigl(\bfxi\cdot \bigl( 
    \grad\Wrhonxi - \grad\Wrhonxitilde \bigr) \bigl)\psi \dx
    \biggr|
    \,\leq\, C \| \bfxi \|_{L^\infty(\Omegar)} |\Drhon|^{\frac12}
    \bigl\| \grad\Wrhonxi 
    - \grad\Wrhonxitilde \bigr\|_{L^2(\Omegar)} \,.
  \end{equation*}
  To finish the proof, we will show that 
  $\| \grad\Wrhonxi - \grad\Wrhonxitilde \bigr\|_{L^2(\Omegar)}$
  is $o\bigl(|\Drhon|^{\frac12}\bigr)$ as $n\to\infty$. 
  Then \eqref{eq:CharacterizationPolTen(b)} follows from 
  \eqref{eq:ProofThm4.1-1} and \eqref{eq:ProofThm4.1-b_last}.
  This is done in Lemma~\ref{lmm:EstI2} below.
\end{proof}

\begin{lemma}
  \label{lmm:EstI2}
  Let $\bfxi'\in\Sone$, and let $\bfxi \in C^1(\Omegar,\Stwo)$ be
  given by
  \begin{equation*}
    \bfxi(\rK(s,\eta,\zeta)) 
    \,:=\, 
    \begin{bmatrix}
      \nK(s) & \bK(s)
    \end{bmatrix}
    \bfxi'
    \qquad \text{for any }\, \bfx = \rK(s,\eta,\zeta)\in\Omegar \,.
  \end{equation*}
  Let $\Wrhonxi \in H^1_0(\Omegar)$ be the corresponding solution 
  to~\eqref{eq:DefWrhoxi}, and define
  \begin{equation}
    \label{eq:DefWrhonxitilde}
    \Wrhonxitilde(\rK(s,\eta,\zeta))
    \,:=\, \frhon(s) \JrK^{-1}(s,\eta,\zeta) 
    \wrhonxithetatilde(\eta,\zeta) \,,
    \qquad s\in (-L,L) \,,\; (\eta,\zeta)\in \Br \,,
  \end{equation}
  where $\frhon\in C^1([-L,L])$ is a cut-off
  functions satisfying \eqref{eq:Propertiesfrho}, and 
  $\wrhonxithetatilde \in H^1_0(\Br)$
  solves \eqref{eq:Defwrhoxitilde}. 
  Then,
  \begin{equation*}
     \bigl\| \grad\Wrhonxi - \grad\Wrhonxitilde \bigr\|_{L^2(\Omegar)}
    \,=\, o\bigl(|\Drhon|^{\frac12}\bigr) \,.
  \end{equation*}
\end{lemma}

In the proof of Lemma~\ref{lmm:EstI2} we use the following technical
result, which can be shown using the same arguments as in the proof of
\cite[Lmm.~3.4]{BerCapdeGFra09}. 

\begin{lemma}
  \label{lmm:EstDerivativew}
  Let $\bfxi'\in\Sone$, and let $\wrhonxithetatilde \in H^1_0(\Br)$ be the
  solution to \eqref{eq:Defwrhoxitilde}.
  Then, for a.e.~$s\in(-\ell,\ell)$, 
  \begin{equation}
    \label{eq:EstDerivativegradw}
    \biggl\| \frac{\di}{\di s} \Bigl(
    \crossgrad\wrhonxithetatilde \Bigr) \biggr\|_{L^2(\Br)}
    \,\leq\, C |\Drhon'|^{\frac12} 
    \qquad\text{and}\qquad
    \biggl\|\frac{\di\wrhonxithetatilde}{\di s} \biggr\|_{L^2(\Br)}
    \,\leq\, C |\Drhon'|^{\frac34} \,.
  \end{equation}
\end{lemma}

\begin{proof}[Proof of Lemma~\ref{lmm:EstI2}]
  Recalling \eqref{eq:DefWrhoxi} we note that 
  $\Wrhonxi\in H^1_0(\Omegar)$ fulfills
  \begin{equation}
    \label{eq:ProofEstI2-0}
    \int_{\Omegar} \gamma_{\rhon} \grad \Wrhonxi \cdot \grad \psi \dx
    \,=\, - \int_{\Omegar} (\gamma_{\rhon}-\gamma_0) \bfxi 
    \cdot \grad \psi \dx
  \end{equation}
  for all $\psi\in H^1_0({\Omegar})$.
  Furthermore, recalling~\eqref{eq:LocalCoordsKJacobian} and
  \eqref{eq:LocalCoordsKGradient} we find that 
  $\Wrhonxitilde=\frhon\JrK^{-1}\wrhonxithetatilde$ satisfies
  \begin{equation}
    \label{eq:ProofEstI2-1}
    \begin{split}
      &\int_{\Omegar} \gamma_{\rhon} \grad \Wrhonxitilde 
      \cdot \grad \psi \dx\\
       &\,=\, \int_{-L}^L \int_{\Br} \gamma_{\rhon} \Bigl( 
      \crossgrad\Wrhonxitilde \cdot \crossgrad\psi%
       + \bigl(\tK\cdot\grad\Wrhonxitilde\bigr)
      \bigl(\tK\cdot\grad\psi\bigr)
      \Bigr) \JrK \dez \ds\\
      &\,=\, \int_{-L}^L \int_{\Br} \gamma_{\rhon}
      \biggl( \chi_{(-\ell,\ell)} \Bigl( 
      \wrhonxithetatilde \crossgrad\JrK^{-1} 
      + \JrK^{-1} \crossgrad \wrhonxithetatilde \Bigr) 
      \cdot \crossgrad\psi\\
      &\phantom{\,=\,\qquad}
      + (1-\chi_{(-\ell,\ell)}) \Bigl(
      \frhon\wrhonxithetatilde\crossgrad\JrK^{-1}
      + \frhon\JrK^{-1}\crossgrad\wrhonxithetatilde \Bigr)      
      \cdot\crossgrad\psi\\
      &\phantom{\,=\,\qquad}
      + \Bigl( \wrhonxithetatilde\Bigl(\tK\cdot\grad(\frhon\JrK^{-1})\Bigr)
      + \frhon\JrK^{-1}\Bigl(\tK\cdot\grad\wrhonxithetatilde\Bigr)
       \Bigr) 
      \bigl( \tK\cdot\grad\psi \bigr) \biggr) \JrK \dez \ds \,.
    \end{split}
  \end{equation}
  Using \eqref{eq:LocalCoordsKGradient} once more we further decompose
  the last term on the right hand side of \eqref{eq:ProofEstI2-1} to
  obtain 
  \begin{multline*}
    \Bigl(\tK\cdot\grad\wrhonxithetatilde\Bigr)
    \bigl( \tK\cdot\grad\psi \bigr)
    \,=\, \JrK^{-2} \frac{\di \wrhonxithetatilde}{\di s} 
    \frac{\di \psi}{\di s}
    + A_\tor \crossgrad\wrhonxithetatilde 
    \cdot \crossgrad\psi\\
    + \frac{\di \wrhonxithetatilde}{\di s} 
    \bigl( \bfd_\tor' \cdot \crossgrad\psi \bigr)
    + \frac{\di \psi}{\di s} 
    \Bigl( \bfd_\tor' \cdot \crossgrad\wrhonxithetatilde \Bigr) \,,
  \end{multline*}
  where $A_\tor$ has been defined in \eqref{eq:DefAtor} and
  \begin{equation*}
    \bfd_\tor'(s,\eta,\zeta)
    \,:=\, \Bigl( \tau+\thetarhonprime \Bigr)(s) 
    \JrK^{-1}(s,\eta,\zeta)
    \begin{bmatrix}
      \zeta\\ -\eta
    \end{bmatrix} \,, \qquad s\in(-L,L) \,,\; 
    (\eta,\zeta) \in \Br \,. 
  \end{equation*}
  Accordingly,
  \begin{equation}
    \label{eq:ProofEstI2-4}
    \begin{split}
      &\int_{\Omegar} \gamma_{\rhon} \grad \Wrhonxitilde 
      \cdot \grad \psi \dx
      \,=\, \int_{-\ell}^\ell \int_{\Br} (\I_2+A_\tor) \gamma_{\rhon}
      \crossgrad \wrhonxithetatilde \cdot \crossgrad\psi \dez \ds\\
      &\phantom{\,=\,}
      + \int_{-L}^L \int_{\Br} \gamma_{\rhon} \biggl( 
      \chi_{(-\ell,\ell)} \wrhonxithetatilde 
      \crossgrad\JrK^{-1}\cdot\crossgrad\psi\\
      &\phantom{\,=\,\qquad}
      + (1-\chi_{(-\ell,\ell)}) \Bigl(
      \frhon\wrhonxithetatilde\crossgrad\JrK^{-1}
      + \frhon\JrK^{-1}(\I_2+A_\tor)\crossgrad\wrhonxithetatilde \Bigr)      
      \cdot\crossgrad\psi\\
      &\phantom{\,=\,\qquad}
      + \wrhonxithetatilde\Bigl(\tK\cdot\grad(\frhon\JrK^{-1})\Bigr)
      \bigl( \tK\cdot\grad\psi \bigr)
      + \frhon \JrK^{-3} \frac{\di \wrhonxithetatilde}{\di s} 
      \frac{\di \psi}{\di s}\\
      &\phantom{\,=\,\qquad}
      + \frhon\JrK^{-1} 
      \biggl(\frac{\di \wrhonxithetatilde}{\di s} 
      \bigl( \bfd_\tor' \cdot \crossgrad\psi \bigr)
      + \frac{\di \psi}{\di s} 
      \Bigl( \bfd_\tor' \cdot \crossgrad\wrhonxithetatilde \Bigr) 
      \biggr)
      \biggr) \JrK \dez \ds \,.
    \end{split}
  \end{equation}

  Now let $v\in H^1_0(\Br)$ satisfy
  \begin{equation*}
    \crossdiv\bigl( (\I_2+A_\tor) \gamma_{\rhon}' \crossgrad v \bigr)
    \,=\, -\crossdiv\bigl( 
    \JrK (\gamma_{\rhon}'-\gamma_0')
    \bigl(\Rthetarhon^{-1}\bfxi'\bigr) \bigr) 
    \qquad \text{in } \Br \,.
  \end{equation*}
  Using \eqref{eq:LocalCoordsKJacobian} we find that
  \begin{equation}
    \label{eq:PDEv}
    \begin{split}
      \crossdiv\bigl( (\I_2+A_\tor)\gamma_{\rhon}' \crossgrad v \bigr)
      &\,=\, -\crossdiv\bigl( (\gamma_{\rhon}'-\gamma_0')
      \bigl(\Rthetarhon^{-1}\bfxi'\bigr) \bigr)\\
      &\phantom{\,=\,}
      +\crossdiv\Bigl( \kappa\Bigl(\bfe_1'\cdot\Rthetarhon
      \begin{bmatrix}
        \eta\\\zeta
      \end{bmatrix}\Bigr) (\gamma_{\rhon}'-\gamma_0') 
      \bigl(\Rthetarhon^{-1}\bfxi'\bigr) \Bigr) \,.
    \end{split}
  \end{equation}
  Together with \eqref{eq:Defwrhoxi} and the uniqueness of solutions
  to the Dirichlet problem this implies that
  $\wrhonxithetatilde=v-v_1$, where $v_1\in H^1_0(\Br)$ satisfies
  \begin{equation*}
    \crossdiv\bigl( (\I_2+A_\tor) \gamma_{\rhon}' \crossgrad v_1 \bigr)
    \,=\, -\crossdiv\Bigl( \kappa\Bigl(\bfe_1'\cdot\Rthetarhon
    \begin{bmatrix}
      \eta\\\zeta
    \end{bmatrix}\Bigr) (\gamma_{\rhon}'-\gamma_0') 
    \bigl(\Rthetarhon^{-1}\bfxi'\bigr)\Bigr) 
    \qquad\text{in } \Br \,.
  \end{equation*}
  Using \eqref{eq:Estwq1} we obtain the estimate
  \begin{equation*}
    \|\crossgrad v_1\|_{L^2(\Br)}
    \,\leq\, C \kappamax\rhon\, |\Drhon'|^{\frac12} \,.
  \end{equation*}
  Accordingly \eqref{eq:PDEv} and \eqref{eq:LocalCoordsKGradient} give 
  \begin{equation}
    \label{eq:ProofEstI2-9}
    \begin{split}
      \int_{-\ell}^\ell& \int_{\Br} (\I_2+A_\tor) \gamma_{\rhon}
      \crossgrad \wrhonxithetatilde \cdot \crossgrad\psi\, \dez \ds\\
      &\,=\, \int_{-\ell}^\ell \int_{\Br}  (\I_2+A_\tor)
      \gamma_{\rhon} \crossgrad v \cdot \crossgrad\psi\, \dez \ds 
      + o\bigl(|\Drhon|^{\frac{1}{2}} \|\grad\psi\|_{L^2(\Omegar)}\bigr)\\
      &\,=\, -\int_{-\ell}^\ell \int_{\Br} 
      (\gamma_{\rhon}-\gamma_0) \bigl(\Rthetarhon^{-1}(s)\bfxi'\bigr)
      \cdot \crossgrad\psi\, \JrK \dez \ds 
      + o\bigl(|\Drhon|^{\frac{1}{2}} \|\grad\psi\|_{L^2(\Omegar)}\bigr)\\
      &\,=\, -\int_{{\Omegar}} (\gamma_{\rhon}-\gamma_0) 
      \bfxi \cdot \grad \psi \dx 
      + o\bigl(|\Drhon|^{\frac{1}{2}}\|\grad\psi\|_{L^2(\Omegar)}\bigr) \,.
    \end{split}
  \end{equation}
  Combining \eqref{eq:ProofEstI2-0}, \eqref{eq:ProofEstI2-4} and
  \eqref{eq:ProofEstI2-9} we find that, for all 
  $\psi\in H^1_0({\Omegar})$, 
  \begin{equation}
    \label{eq:BeforeIntbyParts}
    \begin{split}
      &\int_{\Omegar} \gamma_{\rhon} 
      \bigl( \grad \Wrhonxi - \grad \Wrhonxitilde \bigr)
      \cdot \grad \psi \dx\\
      &\,=\, -\int_{-L}^L \int_{\Br} \gamma_{\rhon} \biggl( 
      \chi_{(-\ell,\ell)} \wrhonxithetatilde 
      \crossgrad\JrK^{-1}\cdot\crossgrad\psi\\
      &\phantom{\,=\,}
      + (1-\chi_{(-\ell,\ell)}) \Bigl(
      \frhon\wrhonxithetatilde\crossgrad\JrK^{-1}
      + \frhon\JrK^{-1}(\I_2+A_\tor)\crossgrad\wrhonxithetatilde \Bigr)      
      \cdot\crossgrad\psi\\
      &\phantom{\,=\,}
      + \wrhonxithetatilde\Bigl(\tK\cdot\grad(\frhon\JrK^{-1})\Bigr)
      \bigl( \tK\cdot\grad\psi \bigr) 
      + \frhon \JrK^{-3} \frac{\di \wrhonxithetatilde}{\di s} 
      \frac{\di \psi}{\di s}\\
      &\phantom{\,=\,}
      + \frhon\JrK^{-1} 
      \biggl(\frac{\di \wrhonxithetatilde}{\di s} 
      \bigl( \bfd_\tor' \cdot \crossgrad\psi \bigr)
      + \frac{\di \psi}{\di s} 
      \Bigl( \bfd_\tor' \cdot \crossgrad\wrhonxithetatilde \Bigr) 
      \biggr)
      \biggr) \JrK \dez \ds\\
      &\phantom{\,=\,} 
      + o\bigl(|\Drhon|^{\frac{1}{2}}\|\grad\psi\|_{L^2(\Omegar)}\bigr) \,.
    \end{split}
  \end{equation}
  Now let $\grhon\in C^1([-L,L])$ be a cut-off function satisfying
  \begin{subequations}
    \label{eq:Propertiesgrho}
    \begin{align}
      &0\leq \grhon\leq 1 \,, \qquad 
        \supp(\grhon) = [-\ell,\ell] \,, \qquad
        \grhon\chi_{(-\frac{\ell}{2},\frac{\ell}{2})} 
        \,=\, \chi_{(-\frac{\ell}{2},\frac{\ell}{2})} \,,\\
      &\|\grhon'\|_{L^2((-L,L))} 
        \,\leq\, C |\Drhon'|^{-\frac{1}{8}} \,, \qquad
        \bigl\|\frhon(1-\grhon)\bigr\|_{L^2((-L,L))} 
        \,\leq\, C|\Drhon'|^{\frac{1}{8}} \,,
    \end{align}
  \end{subequations}
  where $\frhon$ denotes the cut-off function from
  \eqref{eq:Propertiesfrho} (see \cite[Lmm.~3.6]{BerCapdeGFra09} for a
  similar construction).
  Integrating by parts shows that the last term in the integral on
  the right hand side of \eqref{eq:BeforeIntbyParts} satisfies
  \begin{equation*}
    \begin{split}
      - \int_{-L}^L \int_{\Br}
      \gamma_{\rhon} \frhon \frac{\di \psi}{\di s} 
      \Bigl(& \bfd_\tor' \cdot \crossgrad\wrhonxithetatilde \Bigr) \dez \ds\\
      &\,=\, \int_{-L}^L \int_{\Br}
      \psi \frac{\di}{\di s} \biggl( \gamma_{\rhon} \frhon \grhon
      \Bigl( \bfd_\tor' \cdot \crossgrad\wrhonxithetatilde \Bigr) 
      \biggr) \dez \ds \\
      &\phantom{\,=\,}
      - \int_{-L}^L \int_{\Br}
      \gamma_{\rhon} \frhon (1-\grhon) 
      \Bigl( \bfd_\tor' \cdot \crossgrad\wrhonxithetatilde \Bigr) 
      \frac{\di \psi}{\di s} \dez \ds \,.
    \end{split}
  \end{equation*}
  Combining this with \eqref{eq:BeforeIntbyParts} and choosing 
  $\psi = \Wrhonxi-\Wrhonxitilde$ this implies that
  \begin{equation}
    \label{eq:ProofLemmaEstI2Final}
    \begin{split}
      &\bigl\| \grad\Wrhonxi - \grad\Wrhonxitilde \bigr\|_{L^2(\Omegar)}^2
      \,\leq\, C \biggl(
      \Bigl\| \chi_{(-\ell,\ell)}
      \wrhonxithetatilde \crossgrad\JrK^{-1} \Bigr\|_{L^2(\Omegar)}\\
      &+ \Bigl\| (1-\chi_{(-\ell,\ell)}) 
      \frhon\wrhonxithetatilde\crossgrad\JrK^{-1} \Bigr\|_{L^2(\Omegar)}
      + \Bigl\| (1-\chi_{(-\ell,\ell)})\frhon\JrK^{-1} (\I_2+A_\tor)
      \crossgrad\wrhonxithetatilde \Bigr\|_{L^2(\Omegar)}\\
      &+ \Bigl\| \wrhonxithetatilde 
      \Bigl(\tK\cdot\grad\bigl(\frhon\JrK^{-1}\bigr)\Bigr)
      \Bigr\|_{L^2({\Omegar})}
      + \Bigl\| \frhon \JrK^{-3} 
      \frac{\di \wrhonxithetatilde}{\di s} \Bigr\|_{L^2(\Omegar)}
      + \Bigl\| \frhon\JrK^{-1}
      \frac{\di \wrhonxithetatilde}{\di s} \bfd_\tor'
      \Bigr\|_{L^2(\Omegar)}\\
      &+ \Bigl\| \gamma_{\rhon} \frhon (1-\grhon) 
      \Bigl( \bfd_\tor' \cdot \crossgrad\wrhonxithetatilde \Bigr) 
      \Bigr\|_{L^2(\Omegar)}
      \biggr) 
      \bigl\|\grad\Wrhonxi-\grad\Wrhonxitilde
      \bigr\|_{L^2(\Omegar)}\\
      &+ C \Bigl\| \frac{\di}{\di s} \Bigl(
      \gamma_{\rhon} \frhon \grhon \JrK^{-1}
      \Bigl( \bfd_\tor'\cdot\crossgrad\wrhonxithetatilde \Bigr) \Bigr) 
      \Bigr\|_{L^2(\Omegar)} 
      \bigl\|\Wrhonxi-\Wrhonxitilde\bigr\|_{L^2(\Omegar)}\\
      &+ o\Bigl(|\Drhon|^{\frac{1}{2}}
      \bigl\|\grad\Wrhonxi-\grad\Wrhonxitilde
      \bigr\|_{L^2(\Omegar)}\Bigr) \,.
    \end{split}
  \end{equation}
  From \eqref{eq:ProofEstI2-0}, \eqref{eq:DefWrhonxitilde},
  \eqref{eq:Defwrhoxitilde}, and \eqref{eq:Estwq2} we immediately
  obtain that 
  \begin{equation*}
    \bigl\|\Wrhonxi-\Wrhonxitilde \bigr\|_{L^2(\Omegar)}
    \,\leq\, C |\Drhon|^{\frac34} \,.
  \end{equation*}
  Next we estimate the remaining eight terms on the right hand side of
  \eqref{eq:ProofLemmaEstI2Final} separately.
  For the first term we obtain, using \eqref{eq:Estwq2}, that
  \begin{equation*}
    \begin{split}
      \Bigl\| \chi_{(-\ell,\ell)}
      \wrhonxithetatilde \crossgrad\JrK^{-1} \Bigr\|_{L^2(\Omegar)}^2
      &\,\leq\, C \Bigl\| \wrhonxithetatilde \Bigr\|_{L^2(\Br)}^2
      \,\leq\, C |\Drhon|^{\frac32} \,.
    \end{split}
  \end{equation*}
  Similarly, using \eqref{eq:Propertiesfrho} and \eqref{eq:Estwq2} we
  find for the second term on the right hand side of
  \eqref{eq:ProofLemmaEstI2Final} that 
  \begin{equation*}
    \begin{split}
      \Bigl\|  (1-\chi_{(-\ell,\ell)}) 
      \frhon\wrhonxithetatilde\crossgrad\JrK^{-1} \Bigr\|_{L^2(\Omegar)}^2
       &\,\leq\, C \Bigl\| \wrhonxithetatilde \Bigr\|_{L^2(\Br)}^2
      \,\leq\, C |\Drhon|^{\frac32} \,.
    \end{split}
  \end{equation*}
  Applying \eqref{eq:Propertiesfrho} and \eqref{eq:Estwq1} the third
  term on the right hand side of \eqref{eq:ProofLemmaEstI2Final} can
  be estimated by
  \begin{multline*}
    \Bigl\| (1-\chi_{(-\ell,\ell)})\frhon\JrK^{-1}(\I_2+A_\tor)
    \crossgrad\wrhonxithetatilde \Bigr\|_{L^2(\Omegar)}^2\\
    \,\leq\, C \bigl\| (1-\chi_{(-\ell,\ell)})\frhon \bigr\|_{L^2((-L,L))}^2 
    \Bigl\| \crossgrad\wrhonxithetatilde \Bigr\|_{L^2(\Br)}^2
    \,\leq\, C |\Drhon'|^{\frac14} |\Drhon'| 
    \,\leq\, C |\Drhon|^{\frac54} \,.
  \end{multline*}
  For the fourth term on the right hand
  side of \eqref{eq:ProofLemmaEstI2Final} we obtain, using
  \eqref{eq:LocalCoordsKGradient}, \eqref{eq:Propertiesfrho},
  and~\eqref{eq:Estwq2} that 
  \begin{equation*}
    \begin{split}
      &\Bigl\| \wrhonxithetatilde 
      \Bigl(\tK\cdot\grad\bigl(\frhon\JrK^{-1}\bigr)\Bigr)
      \Bigr\|_{L^2({\Omegar})}^2\\
      &\,=\, \int_{-L}^L \int_{\Br}
      \Bigl|\wrhonxithetatilde(\eta,\zeta)\Bigr|^2 
      \Bigl| \JrK^{-2}(s,\eta,\zeta)\frhon'(s) 
      + \frhon\bigl(\tK\cdot\grad\JrK^{-1}\bigr) \Bigr|^2
      \JrK(s,\eta,\zeta) \dez \ds\\
      &\,\leq\, C \Bigl\|\wrhonxithetatilde\Bigr\|_{L^2(\Br)}^2 
      \bigl( \|\frhon'\|_{L^2(-L,L)}^2 + C \bigr)
      \,\leq\, C|\Drhon'|^{\frac32} |\Drhon'|^{-\frac14} 
      \,\leq\, C|\Drhon|^{\frac54} \,.
    \end{split}
  \end{equation*}
  Using \eqref{eq:EstDerivativegradw} we obtain for the fifth and
  sixth term on the right hand side of \eqref{eq:ProofLemmaEstI2Final}
  that 
  \begin{equation*}
    \Bigl\| \frhon \JrK^{-3} 
    \frac{\di \wrhonxithetatilde}{\di s} \Bigr\|_{L^2(\Omegar)}
    \,\leq\, C |\Drhon'|^{\frac32}
    \quad\text{and}\quad
    \Bigl\| \frhon\JrK^{-1}
    \frac{\di \wrhonxithetatilde}{\di s} \bfd_\tor' \Bigr\|_{L^2(\Omegar)}^2
    \,\leq\, C |\Drhon'|^{\frac32} \,.
  \end{equation*}
  Applying \eqref{eq:Propertiesgrho} and \eqref{eq:Estwq1} we find for
  the seventh term on the right hand side of
  \eqref{eq:ProofLemmaEstI2Final} that 
  \begin{multline*}
    \Bigl\| \gamma_{\rhon} \frhon (1-\grhon) 
    \Bigl( \bfd_\tor' \cdot \crossgrad\wrhonxithetatilde \Bigr) 
    \Bigr\|_{L^2(\Omegar)}\\
    \,\leq\, C \| \gamma_{\rhon} \frhon (1-\grhon) \|_{L^2(\Omegar)}
    \Bigl\| \crossgrad\wrhonxithetatilde \Bigr\|_{L^2(\Br)}
    \,\leq\, C |\Drhon'|^{\frac18} |\Drhon'|^{\frac12} \,.
  \end{multline*}
  Finally, combining \eqref{eq:Propertiesgrho},
  \eqref{eq:Propertiesfrho}, \eqref{eq:Estwq1}, 
  and \eqref{eq:EstDerivativegradw} shows that
  \begin{equation*}
    \begin{split}
      &\Bigl\| \frac{\di}{\di s} \Bigl(
      \gamma_{\rhon} \frhon \grhon \JrK^{-1}
      \Bigl( \bfd_\tor'\cdot\crossgrad\wrhonxithetatilde \Bigr) \Bigr) 
      \Bigr\|_{L^2(\Omegar)}\\
      &\,\leq\, C \biggl(
      \Bigl\| \frac{\di}{\di s} 
      \bigl( \gamma_{\rhon} \frhon \grhon \JrK^{-1} \bigr) \Bigr\|_{L^2(\Omegar)}
      \Bigl\| \bfd_\tor'\cdot\crossgrad\wrhonxithetatilde \Bigr\|_{L^2(\Omegar)}
      + \Bigl\| \frac{\di}{\di s} \Bigl(
      \bfd_\tor'\cdot\crossgrad\wrhonxithetatilde \Bigr) 
      \Bigr\|_{L^2(\Omegar)} \biggr) \\
      &\,\leq\, C \Bigl( |\Drho'|^{-\frac18} |\Drho'|^{\frac12} 
      + |\Drho'|^{\frac12} \Bigr) 
      \,\leq\, C |\Drho'|^{\frac38} \,.
    \end{split}
  \end{equation*}
  This ends the proof.
\end{proof}

\begin{remark}
  \label{rem:PolarizationTensor}
  Combining Theorems~\ref{thm:GeneralAsymptotics} and
  \ref{thm:CharacterizationPolTen} gives almost explicit asymptotic
  representation formulas for the scattered electric field $\Esrhon$
  away from the scatterer and for its far field pattern~$\Einftyrhon$
  as $n\to\infty$.
  All components of these formulas, except for the polarization
  tensors $\meps,\mmu\in\R^{2\times 2}$ of the cross-sections
  $(\Drhon')_n$ can be evaluated straightforwardly. 
  
  If we assume some more regularity and consider sequences of
  cross-sections 
  \begin{equation*}
    D_{\rhon}' \,=\, \rhon B' \,, \qquad 0<\rhon<r/2 \,,\; n\in\N \,,
  \end{equation*}
  for some Lipschitz domain $B'\tm B_1'(0)$ then the following
  integral representation for $\mgamma$, $\gamma\in\{\eps,\mu\}$, is
  well known (see, e.g., \cite{CedMosVog98,AmmKan07}). 
  Introducing 
  \begin{equation*}
    \gammatilde'(\bfx') \,:=\,
    \begin{cases}
      \gamma_1\,, & \bfx' \in B'\,,\\
      \gamma_0\,, & \bfx' \in \R^2\setminus\overline{B'} \,,
    \end{cases}
  \end{equation*}
  the polarization tensor 
  $\mgamma = (\mgamma_{i,j})_{i,j} \in \R^{2\times 2}$ corresponding
  to the cross-sections $(\Drhon')_n$ satisfies 
  \begin{equation}
    \label{eq:mgammaij}
    \mgamma_{ij} 
    \,=\, \delta_{ij} + \frac1{|B'|} \int_{B'} 
    \frac{\di \wtilde_j}{\di x'_i}(\bfx') \dx' \,,
    \qquad 1\leq i,j\leq 2 \,,
  \end{equation}
  where $\delta_{ij}$ denotes the Kronecker delta and 
  $\wtilde_j\in H^1_\loc(\R^2)$ denotes the unique solution to the
  transmission problem
  \begin{subequations}
    \label{eq:wj}
    \begin{align}
      \Delta \wtilde_j
      &\,=\, 0 \qquad \quad\text{in } \R^2\setminus\di B \,,\\
      \wtilde_j\big|_{\di B}^+ - \wtilde_j\big|_{\di B}^- 
      &\,=\, 0 \,,\\
      \gamma_0 \frac{\di\wtilde_j}{\di\bfnu}\Big|_{\di B}^+ 
      - \gamma_1 \frac{\di\wtilde_j}{\di\bfnu}\Big|_{\di B}^-
      &\,=\, -(\gamma_0-\gamma_1) \nu_j \,,\\
      \wtilde_j(\bfx') &\to 0 \qquad \quad \text{as $|\bfx'|\to\infty$} \,.
    \end{align}
  \end{subequations}
  In particular the limit in \eqref{eq:PolTen2D} is uniquely determined
  and thus no extraction of a subsequence is required in the
  Theorems~\ref{thm:GeneralAsymptotics} and
  \ref{thm:CharacterizationPolTen} for this class of cross-sections. 
  Given any specific example for $B'$, the functions $\vtilde_j$,
  $j=1,2$, can be approximated by solving the two-dimensional
  transmission problem \eqref{eq:wj} numerically, and then the
  polarization tensor $\mgamma$ can be evaluated by applying a
  quadrature rule to the two-dimensional integral in
  \eqref{eq:mgammaij}. 
  Therewith, the representation formulas \eqref{eq:GenAsyEs} and
  \eqref{eq:GenAsyEinfty} yield a very efficient tool to evaluate the
  scattered electric field due to a thin tubular scattering object and
  its electric far field pattern.

  Explicit formulas for $\mgamma$ are, e.g., available when $B'$ is an 
  ellipse (cf., e.g., \cite{AmmKan07,BruHanVog03}) or a washer (see
  \cite{CapVog04}). 
  In the latter case, the thin tubular scatterer would correspond to a
  thin pipe. 
  In the special case when $B'$ is a disk we have that
  \begin{equation*}
    \mgamma 
    \,=\, 2\frac{\gamma_0}{\gamma_1+\gamma_0} \I_2 \,,
  \end{equation*}
  where $\I_2\in\R^2$ denotes the identity matrix.
  \hfill$\lozenge$
\end{remark}

We will provide numerical results and discuss the accuracy of the
asymptotic perturbation formula established in
Theorems~\ref{thm:GeneralAsymptotics} and
\ref{thm:CharacterizationPolTen} for some specific examples in 
Section~\ref{sec:NumericalResults} below. 
Before we do so, we consider an application and utilize the asymptotic
perturbation formula to develop an efficient iterative reconstruction
method for an inverse scattering problem with thin tubular scattering
objects. 
This is the topic of the next section.

\section{Inverse scattering with thin tubular scattering objects}
\label{sec:InverseProblem}
We consider the inverse problem to recover the shape of a thin tubular
scattering object~$\Drho$ as in~\eqref{eq:DefDrho} from observations
of a single electric far field pattern $\Einftyrho$ due to an incident
field $\Ei$. 
We restrict the discussion to the special case, when the cross-section
of the scatterer is of the form~$\Drho'=\rho B'$, where $B'=B_1(0)'$
is the unit disk. 
We assume that $\rho>0$ is small with respect to the wave length, and
that this radius as well as the material parameters $\eps_1$ and
$\mu_1$ of the scattering object are known \emph{a priori}.
In this case the explicit formulas for the polarization tensors 
$\meps,\mmu\in\R^{2\times 2}$ of the cross-section from
Remark~\ref{rem:PolarizationTensor} can be used in the reconstruction
algorithm, and a possible twisting the cross-section along the base
curve does have to be taken into account. 
Accordingly, the inverse problem reduces to reconstructing the
center curve~$K$ of the scattering object $\Drho$ from observations of
the electric far field pattern $\Einftyrho$.

We suppose that the incident field is a plane wave, i.e.,
\begin{equation}
  \label{eq:PlaneWave}
  \Ei(\bfx) 
  \,=\, \bfA e^{\rmi k \bftheta \cdot \bfx} \,, \qquad \bfx\in\Rd \,,
\end{equation}
with direction of propagation $\bftheta \in \Stwo$ and polarization
$\bfA \in \Cd\setminus\{0\}$ satisfying $\bfA \perp \bftheta$. 
Other incident fields are possible without significant changes.
The corresponding solution to the direct scattering problem
\eqref{eq:ScatteringProblem} defines a nonlinear operator 
\begin{equation*}
  \Frho:\; K \mapsto \Einftyrho \,,
\end{equation*}
which maps the center curve $K$ of the scattering object $\Drho$ onto
the electric far field pattern $\Einftyrho$. 
In terms of this operator the inverse problem consists in solving the
nonlinear and ill-posed equation
\begin{equation}
  \label{eq:InverseProblem}
  \Frho(K) \,=\, \Einftyrho
\end{equation}
for the unknown center curve $K$. 
In the following we will develop a suitably regularized iterative
reconstruction algorithm for this inverse problem. 

Introducing the set of admissible parametrizations,\footnote{We drop
  the assumption that the center curve of the scatterer is
  parametrized by arc-length for the numerical realization
  reconstruction algorithm.} 
\begin{equation*}
  \Pcal
  \,:=\, \bigl\{ \bfp \in C^3([0,1],\Rd) \;\big|\; 
  \bfp([0,1]) \text{ is simple and } \bfp^\prime(t) \neq 0 
  \text{ for all } t\in [0,1] \bigr\} \,,
\end{equation*}
we identify center curves of thin tubular scattering objects as in
\eqref{eq:DefDrho} with their parametrizations, and we denote
for any $\bfp\in\Pcal$ the leading order term in the asymptotic
perturbation formula \eqref{eq:GenAsyEinfty} by 
\begin{multline}
  \label{eq:Einftylead}
  \Einftyrhotilde(\xhat) 
  \,:=\,  (k\rho)^2\pi \biggl( -
  \int_0^1 (\mu_r-1)\, e^{\rmi k (\bftheta-\xhat)\cdot \bfp(s)} \,
  (\xhat\times \I_3) \Mmu_{\bfp}(s) (\bftheta\times\bfA)\, |\bfp'(s)| \ds\\
  +\int_0^1 (\eps_r-1)\, e^{\rmi k (\bftheta-\xhat)\cdot \bfp(s)} \,
  (\xhat\times(\I_3\times \xhat)) \Meps_{\bfp}(s) \bfA\, |\bfp'(s)| \ds
  \biggr) \,,
  \qquad \xhat\in\Stwo \,.
\end{multline}
Here, $\M_{\bfp}^\gamma:=\M^\gamma\circ\bfp$, $\gamma\in\{\mu,\eps\}$,
is the parametrized form of the polarization tensor for the thin
tubular scatterer. 
The parametrized unit tangent vector field $\tp = \bfp'/|\bfp'|$ along
$\bfp\in\Pcal$ can always be completed to a continuous orthogonal
frame $(\tp,\np,\bp)$. 
For instance, if $\bfp'(t)\times \bfp''(t)\not=0$ for all~$t\in[0,1]$,
then we can choose 
\begin{equation*}
  \tp 
  \,=\, \frac{\bfp'}{|\bfp'|} \,, \qquad 
  \np 
  \,=\, \frac{(\bfp'\times \bfp'')\times \bfp'}
  {|(\bfp'\times \bfp'')\times \bfp'|} \,, \qquad 
  \bp 
  \,=\, \tp\times \np \,.
\end{equation*}
The spectral characterization of $\Mgamma_{\bfp}$ from
Theorem~\ref{thm:CharacterizationPolTen} together with the explicit
formula for the polarization tensor of a disk in
Remark~\ref{rem:PolarizationTensor} shows that, for
$\gamma\in\{\eps,\mu\}$, 
\begin{equation*}
  \Mgamma_{\bfp}(s)
  \,=\, V_{\bfp}(s) M^\gamma V_{\bfp}(s)^\top\,, \qquad s\in [0,1] \,,
\end{equation*}
where 
$M^\gamma := \diag(1,2/(\gamma_r+1),2/(\gamma_r+1)) \in\Rdd$
and $V_{\bfp} := [\tp, \np, \bp] \in C^1([0,1],\Rdd)$ is the
matrix-valued function containing the components of the orthogonal
frame $(\tp,\np,\bp)$ as its columns.

Assuming that the radius $\rho>0$ of the thin tubular scattering
object $\Drho$ is sufficiently small such that the last term on the
right hand side of the asymptotic perturbation formula
\eqref{eq:GenAsyEinfty} can be neglected, we approximate the nonlinear
operator $\Frho$ by the nonlinear operator
\begin{equation*}
  \Trho :\, \Pcal \to L^2(\Stwo,\Cd) \,, \qquad 
  \Trho(\bfp) \,:=\, \Einftyrhotilde \,.
\end{equation*}
Accordingly, we consider the nonlinear minimization problem
\begin{equation}
  \label{eq:OLS}
  \frac{\| \Trho(\bfp) - \Einftyrho \|_{L^2(\Stwo)}^2}
  {\| \Einftyrho \|_{L^2(\Stwo)}^2} \to \min
\end{equation}
to approximate a solution to the inverse problem
\eqref{eq:InverseProblem}. 
We note that due to the asymptotic character of
\eqref{eq:GenAsyEinfty} the minimum of \eqref{eq:OLS} will be non-zero
even for exact far field data. 
Below we will apply a Gau{\ss}-Newton method to a regularized version 
of \eqref{eq:OLS}, and thus we require the Fr\'echet derivative of the
operator $\Trho$.

\subsection{The Fr\'echet derivative of $\Trho$}
\label{subsec:FrechetTrho}
The following lemma concerning the Fr\'echet derivative of the mapping 
$\bfp \mapsto \Mgamma_{\bfp}$ has been established in
\cite[Lmm.~4.1]{GriHyv11}. 

\begin{lemma}
  \label{lmm:FrechetDerivativeM}
  The mapping $\bfp \mapsto \Mgamma_{\bfp}$ is Fr\'echet differentiable
  from $\Pcal$ to $C([0,1],\Rdd)$, and its Fr\'echet derivative at
  $\bfp \in \Pcal$ is given by $\bfh \mapsto (\Mgamma_{\bfp,\bfh})'$
  with 
  \begin{equation*}
    (\Mgamma_{\bfp,\bfh})'
    \,=\, V'_{\bfp,\bfh} M^\gamma V_{\bfp}^\top
    + V_{\bfp} M^\gamma (V'_{\bfp,\bfh})^\top \,,
  \end{equation*}
  where the matrix-valued function $V'_{\bfp,\bfh}$ is defined
  columnwise by 
  \begin{equation*}
    V'_{\bfp,\bfh}
    \,:=\, \frac{1}{|\bfp'|} \bigl[
    (\bfh'\cdot\nK)\nK + (\bfh'\cdot \bK)\bK,\,
    -(\bfh' \cdot \nK)\tK,\, 
    -(\bfh'\cdot \bK)\tK \bigr] \,.
  \end{equation*}
\end{lemma}

Next we consider the Fr\'echet derivative of the mapping $\Trho$.

\begin{theorem}
  \label{thm:FrechetDerivativeTrho}
  The operator $\Trho: \Pcal \to L^2(\Stwo,\Cd)$ is Fr\'echet
  differentiable and its Fr\'echet derivative at $\bfp \in \Pcal$ is
  given by $\Trho'(\bfp) :\, C^3([0,1],\Rd) \to L^2(\Stwo,\Cd)$, 
  \begin{equation}
    \label{eq:FrechetDerivativeTrho}
    \Trho'(\bfp)\bfh 
    \,=\, (k\rho)^2\pi \biggl(
    -(\mu_r-1)\sum_{j=1}^3 {T_{\rho,\mu,j}'(\bfp)\bfh} 
    + (\eps_r -1)\sum_{j=1}^3 {T_{\rho,\eps,j}'(\bfp)\bfh}
    \biggr) 
  \end{equation}
  with
  \begin{align*}
    T_{\rho,\mu,1}'(\bfp)\bfh 
    &\,=\, \int_0^1 \rmi k 
      \bigl((\bftheta-\xhat)\cdot \bfh(s)\bigr) (\xhat\times \I_3) 
      \Mmu_{\bfp}(s) (\bftheta \times \bfA) 
      e^{\rmi k (\bftheta-\xhat) \cdot \bfp(s)} |\bfp'(s)| \ds \,,\\
    T_{\rho,\mu,2}'(\bfp)\bfh 
    &\,=\, \int_0^1 (\xhat\times \I_3)
      (\Mmu_{\bfp,\bfh})'(s) (\bftheta \times \bfA)
      e^{\rmi k (\bftheta-\xhat) \cdot \bfp(s)} |\bfp'(s)| \ds \,,\\ 
    T_{\rho,\mu,3}'(\bfp)\bfh 
    &\,=\, \int_0^1 (\xhat\times \I_3) 
      \Mmu_{\bfp}(s) (\bftheta \times \bfA)
      e^{\rmi k (\bftheta-\xhat) \cdot \bfp(s)}
      \frac{\bfp'(s)\cdot \bfh'(s)}{|\bfp'(s)|} \ds \,,
  \end{align*}    
  and
  \begin{align*}
    T_{\rho,\eps,1}'(\bfp)\bfh 
    &\,=\, \int_0^1 \rmi k 
      \bigl((\bftheta-\xhat)\cdot \bfh(s)\bigr)
      \bigl(\xhat\times (\I_3 \times \xhat)\bigr)
      \Meps_{\bfp}(s) \bfA 
      e^{\rmi k (\bftheta-\xhat) \cdot \bfp(s)} |\bfp'(s)| \ds \,,\\ 
    T_{\rho,\eps,2}'(\bfp)\bfh 
    &\,=\, \int_0^1 (\xhat\times \bigl(\I_3 \times \xhat)\bigr)
      (\Meps_{\bfp,\bfh})'(s) \bfA 
      e^{\rmi k (\bftheta-\xhat) \cdot \bfp(s)} |\bfp'(s)| \ds \,,\\
    T_{\rho,\eps,3}'(\bfp)\bfh 
    &\,=\, \int_0^1 \bigl(\xhat\times (\I_3 \times \xhat)\bigr)
      \Meps_{\bfp}(s) \bfA 
      e^{\rmi k (\bftheta-\xhat) \cdot \bfp(s)}
      \frac{\bfp'(s)\cdot \bfh'(s)}{|\bfp'(s)|} \ds \,. 
  \end{align*}
\end{theorem}

\begin{proof}
  The Fr\'echet derivative and the Fr\'echet differentiability of
  $\Trho$ can be established using Taylor's theorem along the lines of 
  \cite[Thm.~4.2]{GriHyv11}, where a similar operator has been
  considered in the context of an inverse conductivity problem. 
  The proof is therefore omitted. 
\end{proof}

\subsection{Discretization and regularization}
In the reconstruction algorithm we use interpolating cubic splines
with not-a-knot conditions at the end points of the spline to
discretize center curves $K$ parametrized by $\bfp\in\Pcal$. 
Given a non-uniform partition 
\begin{equation}
  \label{eq:DefTri}
  \tri 
  \,:=\,  \{ 0 = t_1 < t_2 < \dots < t_n = 1 \} \tm [0,1] \,,
\end{equation}
we denote corresponding not-a-knot splines by~$\bfp_\tri$. 
The space of all not-a-knot splines with respect to $\tri$ is
denoted by $\Ptri\not\tm\Pcal$.

Since the inverse problem \eqref{eq:InverseProblem} is ill-posed, we
add two regularization terms to stabilize the minimization of
\eqref{eq:OLS}. 
The functional $\Psi_1: \Ptri \to \R$ is defined by
\begin{equation*}
  \Psi_1(\bfp_\tri) 
  \,:=\, \int_{0}^{1} |\kappa(s)|^2 \ds \,,
\end{equation*}
where
\begin{equation*}
  \kappa(s) 
  \,:=\, \frac{|\bfp_\tri'(s) 
    \times \bfp_\tri''(s)|}{|\bfp_\tri'(s)|^3} \,, \qquad 
  s\in[0,1] \,,
\end{equation*}
denotes the curvature of the curve parametrized by $\bfp_\tri$.
We add $\alpha_1^2\Psi_1$ with a regularization parameter
$\alpha_1>0$ as a penalty term to the left hand side of \eqref{eq:OLS}
to prevent minimizers from being too strongly entangled. 

Furthermore, we define another functional $\Psi_2: \Ptri \to \R$ by
\begin{equation*}
  \Psi_2(\bfp_\tri) 
  \,:=\, \sum_{j=1}^{n-1} 
  \Bigl| \frac{1}{n-1}\int_{0}^{1}|\bfp_\tri'(s)| \ds 
  -  \int_{t_j}^{t_{j+1}}|\bfp_\tri' (s)| \ds \Bigr|^2 \,.
\end{equation*}
Adding $\alpha_2^2\Psi_2$ with a regularization
parameter $\alpha_2>0$ as a penalty term to the left hand side of
\eqref{eq:OLS} promotes uniformly distributed control points along
the spline and therefore prevents clustering of control points during
the minimization process. 

Adding both quadratic regularization terms $\alpha_1^2\Psi_1$ and
$\alpha_2^2\Psi_2$ to the left hand side of \eqref{eq:OLS} gives the 
regularized nonlinear output least squares functional
\begin{equation}
  \label{eq:RegOLS}
  \Phi:\, \Ptri \to \R \,, \qquad 
  \Phi(\bfp_\tri)
  \,=\, 
  \frac{\bigl\| \Trho(\bfp_\tri) - \Einftyrho \bigr\|_{L^2(\Stwo)}^2}
  {\bigl\| \Einftyrho \bigr\|_{L^2(\Stwo)}^2}
  + \alpha_1^2 \psi_1(\bfp_\tri) + \alpha_2^2 \psi_2(\bfp_\tri) \,,
\end{equation}
which we will minimize iteratively.

\subsection{The reconstruction algorithm}
We assume that $2N(N-1)$ observations of the far field
$\Einftyrho\in C^\infty(\Stwo,\Cd)$ are available on an equiangular
grid of points 
\begin{equation}
  \label{eq:StwoGrid}
  \bfy_{jl}
  \,:=\,
  \bigl[ \sin\theta_j\cos\varphi_l,\,
  \sin\theta_j\sin\varphi_l,\,
  \cos\theta_j \bigr]^\top \in \Stwo \,, \qquad
  j=1,...,N-1 \,,\; l=1,...,2N \,,
\end{equation}
with $\theta_j=j\pi/N$ and $\varphi_l=(l-1)\pi/N$ for some $N\in\N$. 
Accordingly, we approximate the $L^2(\Stwo)$-norms in the cost
functional $\Phi$ from \eqref{eq:RegOLS} using a composite trapezoid
rule in horizontal and vertical direction. 
This yields an approximation $\Phi_N$ that is given by  
\begin{equation}
  \label{eq:RegOLSDiscrete}
  \Phi_N(\bfp_\tri)
  \,=\, \frac{\sum_{j=1}^{N-1} \sum_{l=1}^{2N} 
    \frac{\pi^2}{N^2} \sin(\theta_j) 
    \bigl| \bigl(\Trho(\bfp_\tri) 
    - \bfE_\rho^\infty\bigr)(\bfy_{jl}) \bigr|^2}
  {\sum_{j=1}^{N-1} \sum_{l=1}^{2N} 
    \frac{\pi^2}{N^2} \sin(\theta_j) 
    \bigl| \bfE_\rho^\infty(\bfy_{jl}) \bigr|^2}
  + \alpha_1^2 \Psi_1(\bfp_\tri)
  + \alpha_2^2 \Psi_2(\bfp_\tri) \,.
\end{equation}

We denote by $\xvec \in \R^{3n}$ the vector with
coordinates for the control points $\bfx^{(1)}, \dots, \bfx^{(n)}$ of
a not-a-knot spline $\bfp_\tri$. 
We approximate all integrals over the parameter range $[0,1]$ of
$\bfp_\tri$ in~\eqref{eq:RegOLSDiscrete} using a composite Simpson's
rule with $M=2m+1$ nodes on each subinterval of the partition $\tri$.
Accordingly, we can rewrite $\Phi_N$ in the form
\begin{equation}
  \label{eq:NLS}
  \Phi_N(\bfp_\tri) 
  \,=\, |P_N(\xvec)|^2 \,,
\end{equation}
where $P_N: \R^{3n} \to \R^{Q}$ and 
$Q = 12N(N-1)+3((M-1)(n-1)+1)+(n-1)$. 
Storing real and imaginary parts separately, $12N(N-1)$
entries of $P_N(\xvec)$ correspond to the normalized residual term in 
\eqref{eq:RegOLSDiscrete}, $3((M-1)(n-1)+1)$ entries correspond to the
penalty term $\Psi_1$, and $n-1$ entries correspond to the penalty
term $\Psi_2$. 
Consequently, we obtain a real-valued nonlinear least squares problem,
which is solved numerically using the Gau{\ss}-Newton algorithm with a
golden section line search (see, e.g., \cite{NocWri99}). 

In addition to the Fr\'echet derivative of the operator $\Trho$ this
also requires the Fr\'echet derivatives of the mappings
$\psi_1:\, \Pcal \to \R$,
\begin{equation*}
  \psi_1(\bfp) 
  \,:=\, \kappa \,,
\end{equation*}
and $\psi_{2,j}:\, \Pcal \to \R$,
\begin{equation*}
  \psi_{2,j}(\bfp) 
  \,:=\, \frac{1}{n-1} \int_{0}^{1} |\bfp'(s)| \ds 
  - \int_{t_j}^{t_{j+1}} |\bfp'(s)| \ds \,, 
\end{equation*}
$j=1,\ldots,n-1$, corresponding to the penalty terms $\Psi_1$ and
$\Psi_2$ in \eqref{eq:RegOLS}, respectively. 
A short calculation shows that at $\bfp\in\Pcal$ these are given by
$\psi_1'(\bfp) :\, C^3([0,1],\Rd) \to \R$, 
\begin{equation}
  \label{eq:FrechetDerivativePsi1}
  \psi_1'(\bfp)\bfh
  \,=\, \frac{\bfh''}{|\bfp'|^2} 
  - \frac{2 (\bfp')^\trans \bfh'}{|\bfp'|^4} \bfp''
  - \frac{(\bfp'')^\trans \bfp'}{|\bfp'|^4} \bfh'
  - \Bigl( \frac{(\bfp'')^\trans \bfh'}{|\bfp'|^4} 
  + \frac{(\bfh'')^\trans \bfp'}{|\bfp'|^4} 
  - \frac{4((\bfp'')^\trans \bfp') ((\bfp')^\trans \bfh')}
  {|\bfp'|^6} \Bigr) \bfp' \,,
\end{equation}
and $\psi_2'(\bfp) :\, C^3([0,1],\Rd) \to \R$, 
\begin{equation}
  \label{eq:FrechetDerivativePsi2}
  \psi_{2,j}'(\bfp)\bfh 
  \,=\, \frac{1}{n-1} \int_{0}^{1} \frac{(\bfp')^\trans \bfh'}
  {|\bfp'|} \ds 
  - \int_{t_j}^{t_{j+1}} \frac{(\bfp')^\trans \bfh'}
  {|\bfp'|} \ds \,, \qquad j=1,\ldots,n-1 \,.
\end{equation}

In Algorithm~\ref{alg:reconstruction} we describe the optimization
scheme that is used to minimize $|P_N|^2$ from \eqref{eq:NLS}. 
Here we denote the Jacobian of $P_N$ by $J_{P_N}$.
The algorithm uses the following heuristic stopping criterion.
If the optimal step size $s_\ell^*$ determined by the line search in
the current iteration is zero and if the value of the objective
functional $|P(\xvec_\ell)|^2$ is dominated by the normalized residual
term in \eqref{eq:RegOLSDiscrete}, then the algorithm stops. 
However, if the optimal step size $s_\ell^*$ determined by the line
search is zero but the value of the objective functional
$|P(\xvec_\ell)|^2$ is dominated by the contribution of one of the two
regularization terms $\alpha_j^2\Psi_j(\bfp_{\tri,\ell})$, 
$j\in\{1,2\}$, then we conclude that in order to further improve the
reconstruction, the corresponding regularization parameter should be
reduced.   
In this case we replace $\alpha_{j}$ by $\frac{\alpha_{j}}{2}$ and
restart the iteration using the current iterate for the initial
guess. 

\begin{Algorithm}[ht!]
  \caption{\small Reconstruction of a thin tubular scatterer $\Drho$
    with circular cross-section}
  \label{alg:reconstruction}
  Suppose that $\Ei$ (i.e., $k$, $\bftheta$, $\bfA$), $\rho$, $\eps_r$,
  $\mu_r$, and $\Einftyrho$ are given. 

  \begin{algorithmic}[1]

    \State Choose an initial guess 
    $\xvec_0 = \bigl[ \bfx^{(1)},\ldots, \bfx^{(n)} \bigr]$ for the 
    control points of a cubic not-a-knot spline 
    $\bfp_{\tri,0} \in \Ptri$ approximating the unknown center curve
    $K$ of $\Drho$. 

    \State Initialize the regularization parameters 
    $\alpha_{1},\alpha_{2}>0$, and a maximal step size $\smax>0$ for
    the line search.

    \For{$\ell = 0,1,...,\ell_{\mathrm{max}}$}

    \State Use the Fr\'echet derivatives $\Trho'$, $\psi_1'$,
    $\psi_{2,1}',\ldots, \psi_{2,n-1}'$ in
    \eqref{eq:FrechetDerivativeTrho}, \eqref{eq:FrechetDerivativePsi1} 
    and \eqref{eq:FrechetDerivativePsi2} to evaluate the Jacobian
    $J_{P_N}$ of $P_N$ from \eqref{eq:NLS}, which is then used to
    compute the Gau\ss-Newton search direction 
    \begin{equation*}
      \bfDelta_\ell
      \,:=\, -\Bigl( J_{P_N}^\top(\xvec_\ell)
      J_{P_N}(\xvec_\ell) \Bigr)^{-1}
      J_{P_N}^\top(\xvec_\ell)P_N (\xvec_\ell) \,.
    \end{equation*}

    \State Use the golden section line search to compute 
    \begin{equation*}
      s^*_\ell 
      \,=\, \argmin_{s \in [0,\smax]} 
      \Phi_N\bigl(\xvec_\ell + s \bfDelta_\ell\bigr) \,.
    \end{equation*}

    \If{$s^*_\ell>0$}
    \State Update the reconstruction, i.e.,
    \begin{equation*}
      \xvec_{\ell+1} \,=\, \xvec_\ell + s^*_\ell\bfDelta_\ell 
      \qquad\text{and}\qquad
      \ell \,=\, \ell+1 \,.
    \end{equation*}
    
    \ElsIf{$s^*_\ell=0$ \textbf{and} the value of
      $|P(\xvec_\ell)|^2$ is dominated by the contribution of
      $\alpha_j^2\Psi_j(\bfp_{\tri,\ell})$, $j\in\{1,2\}$,
      in~\eqref{eq:RegOLSDiscrete}} 

    \State Reduce the corresponding regularization parameter, i.e.,
    \begin{equation*}
      \alpha_{j} 
      \,=\, \alpha_{j} / 2 \,.
    \end{equation*}    

    \ElsIf{$s^*_\ell=0$ \textbf{and} the value of 
      $|P(\xvec_\ell)|^2$ is dominated by the residual term
      in~\eqref{eq:RegOLSDiscrete}} 

    \Return

    \EndIf

    \EndFor

    \State The entries of $\xvec_\ell$ are the coefficients of the
    reconstruction $\bfp_{\tri,\ell}$ of the unknown center curve $K$
    of $\Drho$.

  \end{algorithmic}
\end{Algorithm}

The fact that not a single partial differential equation has to be
solved during the reconstruction process makes this algorithm
extremely efficient, when compared to traditional iterative
shape reconstruction methods for inverse scattering problems for
Maxwell's equations (see, e.g., \cite{Het12,HagAreBetHet19}).

\section{Numerical results}
\label{sec:NumericalResults}
To further illustrate our theoretical findings we provide numerical
examples.
We discuss the accuracy of the approximation of the electric far field
pattern $\Einftyrho$ by the leading order term $\Einftyrhotilde$
in~\eqref{eq:Einftylead}, and we study the performance of the
regularized Gau{\ss}-Newton reconstruction scheme as outlined in
Algorithm~\ref{alg:reconstruction}. 

Recalling that the electric permittivity and the magnetic permeability
in free space are given by 
\begin{equation*}
  \eps_0 
  \,\approx\, 8.854 \times 10^{-12}\, \mathrm{Fm}^{-1}
  \qquad\text{and}\qquad
  \mu_0 
  \,=\, 4\pi \times 10^{-7}\, \mathrm{Hm}^{-1} \,,
\end{equation*}
we consider in all numerical tests an incident plane
wave~$\Ei$ as in \eqref{eq:PlaneWave} at frequency $f=100$~MHz with 
direction of propagation $\bftheta = \frac1{\sqrt{3}} [1,-1,1]^\trans$
and polarization $\bfA = [-1, \rmi, 1+\rmi]^\trans$. 
Accordingly, the wave number is given by 
$k=\omega \sqrt{\eps_0 \mu_0}\approx 2.1$, where 
$\omega=2\pi f$ denotes the angular frequency, and the wave length
is $\lambda \approx 3.0$.  

We focus on three different examples for thin tubular scattering
objects. 
\begin{figure}[t]
  \begin{subfigure}[c]{0.33\textwidth}
    \includegraphics[width=1.\textwidth]{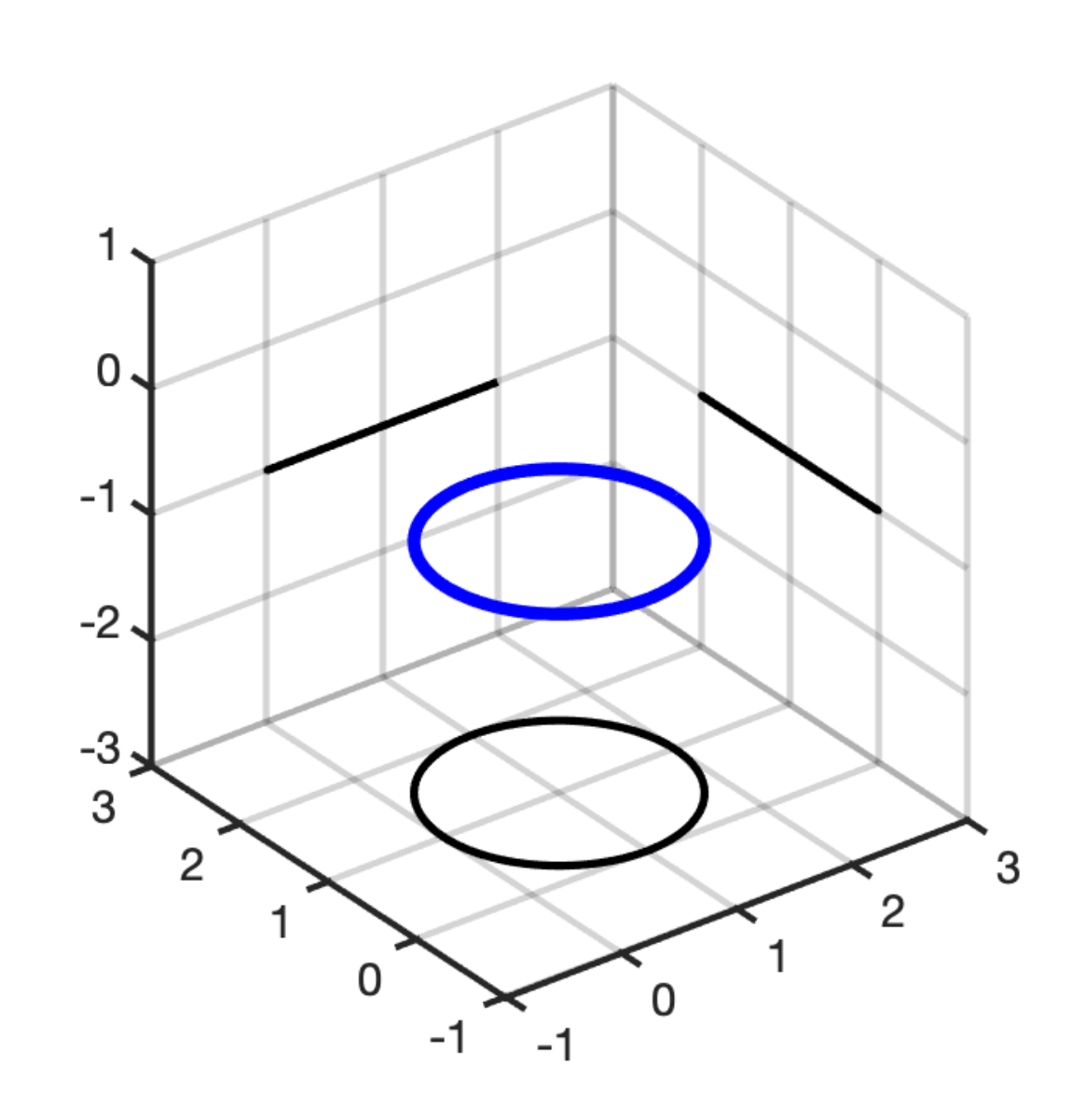}
  \end{subfigure}
  \begin{subfigure}[c]{0.33\textwidth}
    \includegraphics[width=1.\textwidth]{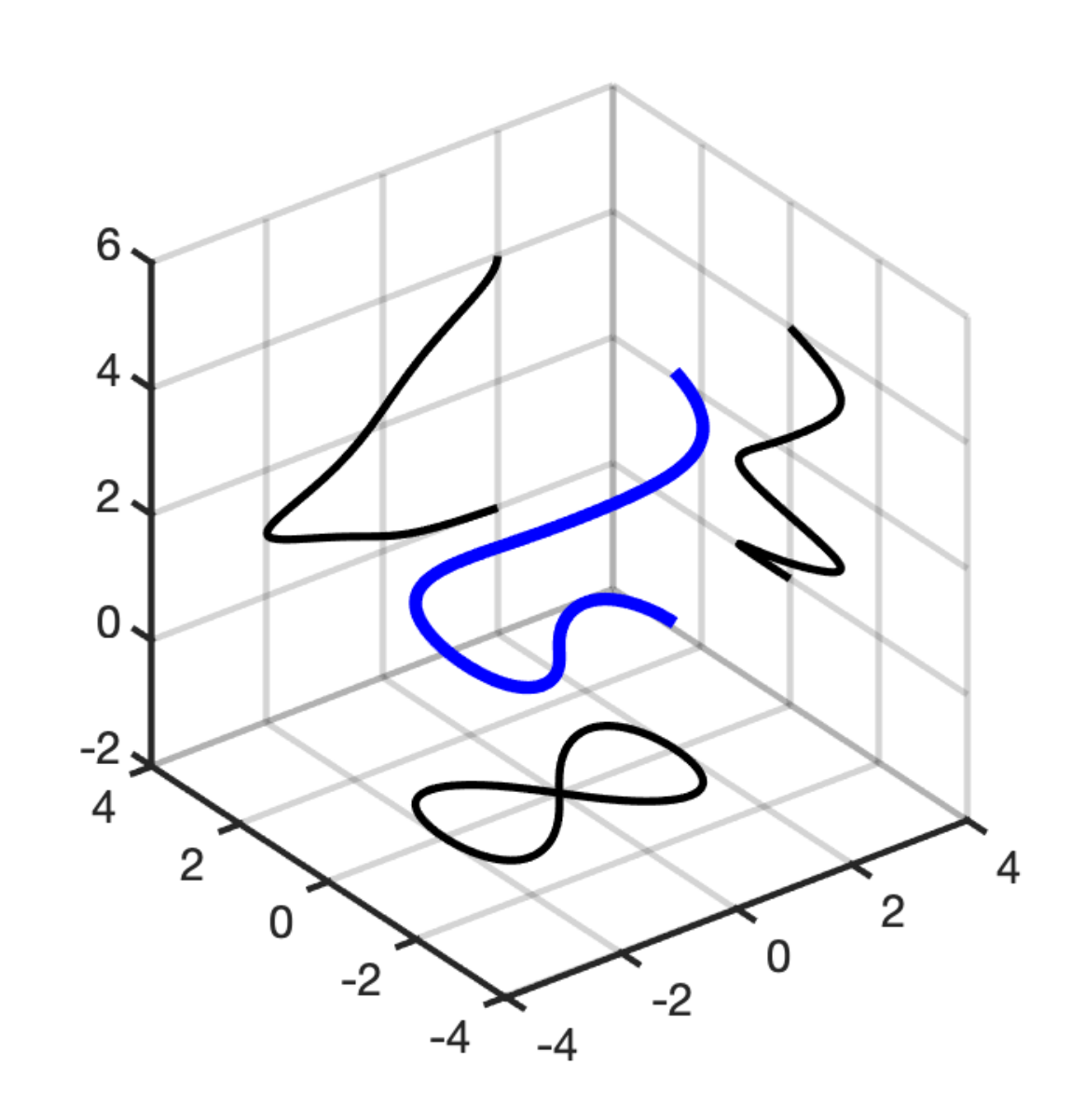}
  \end{subfigure}
  \begin{subfigure}[c]{0.33\textwidth}
    \includegraphics[width=1.\textwidth]{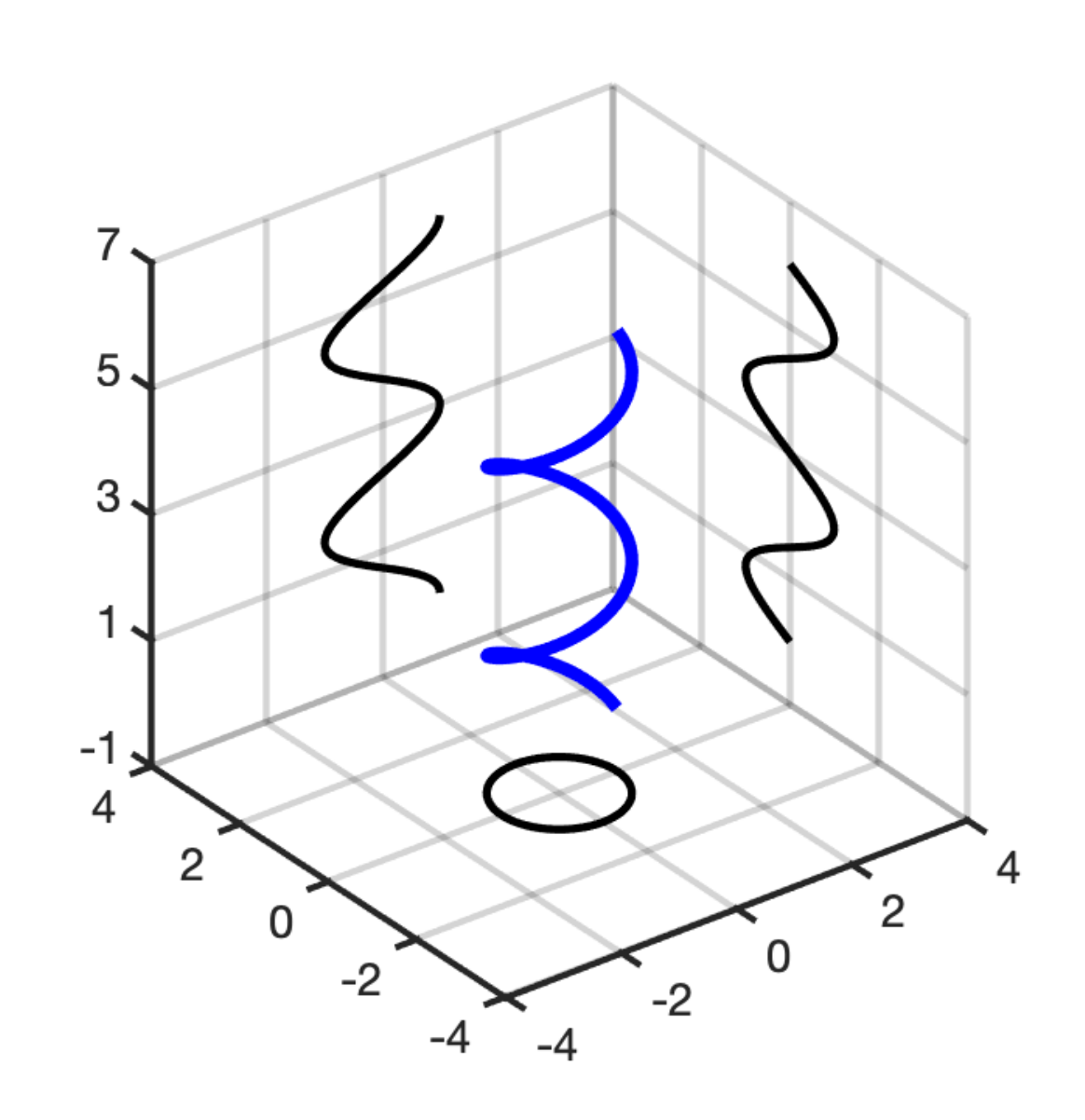}
  \end{subfigure}
  \caption{\small Center curves $K$ (solid blue) of $\Drho$ in
    Examples~\ref{exa:Geometry1} (left), \ref{exa:Geometry2}~(center),
    and \ref{exa:Geometry3}~(right).  Plots also show projections of
    $K$ onto coordinate planes (solid black).} 
  \label{fig:Geometries}
\end{figure}

\begin{example}
  \label{exa:Geometry1}
  In the first example $\Drho$ is a thin torus, where the center curve
  $K$ is a circle parametrized by $\bfp=(p_1,p_2,p_3)^\trans\in\Pcal$
  with 
  \begin{equation*}
    p_1(s) \,=\, \cos(2\pi s) + 1 \,, \quad
    p_2(s) \,=\, \sin(2\pi s) + 1 \,, \quad
    p_3(s) \,=\, -1 \,, \qquad s \in [0,1] \,,
  \end{equation*}
  as shown in Figure~\ref{fig:Geometries} (left).
  The cross-section $\Drho'$ is a disk of radius $\rho>0$, and the
  material parameters of the scattering object are described by the
  relative electric permittivity $\eps_r=2.5$ and the relative
  magnetic permeability $\mu_r=1.6$.~\hfill$\lozenge$ 
\end{example}

\begin{example}
  \label{exa:Geometry2}
  In the second example the scattering object $\Drho$ is a thin tube
  with a center curve parametrized by
  $\bfp=(p_1,p_2,p_3)^\trans\in\Pcal$ with
  \begin{equation*}
    p_1(s) 
    \,=\, 2\frac{\cos(2\pi s)}{1+\sin(2\pi s)^2} \,, \quad
    p_2(s) 
    \,=\, 4\frac{\cos(2\pi s)\sin(2\pi s)}{1+2\sin(2\pi s)^2}\,, \quad
    p_3(s) 
    \,=\, 4s^2 \,, \qquad s \in [0,1] \,,
  \end{equation*}
  as shown in Figure~\ref{fig:Geometries} (center). 
  The cross-section $\Drho'$ is a disk of radius $\rho>0$, and the
  material parameters of the scattering object in this example are 
  described by relative electric permittivity $\eps_r=1.0$ and 
  the relative magnetic permeability $\mu_r= 2.1$, i.e., there is no
  permittivity contrast.~\hfill$\lozenge$ 
\end{example}

\begin{example}
  \label{exa:Geometry3}
  In the third example the scattering object $\Drho$ is a thin tube
  with a center curve that is a two-turn helix parametrized by
  $\bfp=(p_1,p_2,p_3)^\trans\in\Pcal$ with
  \begin{equation*}
    p_1(s) \,=\, \cos(4\pi s) \,, \quad
    p_2(s) \,=\, \sin(4\pi s) \,, \quad
    p_3(s) \,=\, 6s \,, \qquad s \in [0,1] \,,
  \end{equation*}
  a shown in Figure~\ref{fig:Geometries} (right). 
  The cross-section $\Drho'$ is a disk of radius $\rho >0$, and
  the material parameters of the scattering object in this example are 
  described by relative electric permittivity $\eps_r=2.1$ and
  the relative magnetic permeability $\mu_r=1.0$, i.e., there is no
  permeability contrast.~\hfill$\lozenge$ 
\end{example}

\subsection{The accuracy of the asymptotic representation formula} 
We discuss the accuracy of the approximation of the electric far field
pattern $\Einftyrho$ by the leading order term $\Einftyrhotilde$ in
the asymptotic perturbation formula \eqref{eq:GenAsyEinfty}. 
To quantify the approximation error we consider the relative
difference 
\begin{equation}
  \label{eq:DefErr}
  \RelDiff
  \,:=\, \frac{\| \Einftyrhotilde - \Einftyrho \|_{L^2(\Stwo)}}
  {\| \Einftyrho \|_{L^2(\Stwo)}} \,.
\end{equation}
Since the exact far field pattern $\Einftyrho$ is unknown, we simulate
accurate reference far field data~$\Einftyrho$ using the C++ boundary
element library Bempp (see \cite{SmiBetArrPhi15}). 
For this purpose, we consider an integral equation formulation of the
electromagnetic scattering problem \eqref{eq:ScatteringProblem} that
is based on the  multitrace operator. 
The corresponding implementation in Bempp is described in detail
in~\cite{ScrBetBurSmiWou17}. 

To evaluate $\RelDiff$ numerically, we approximate the vector fields
$\Einftyrhotilde$ and $\Einftyrho$ on the equiangular grid on $\Stwo$
from \eqref{eq:StwoGrid} with $N=10$, and accordingly we discretize
the $L^2$-norms in \eqref{eq:DefErr} using a composite trapezoid
rule in horizontal and vertical direction. 
We discuss the following two questions:
\begin{enumerate}
\item[(i)] How many spline segments and how many quadrature points per
  spline segment are sufficient in the approximation $\bfp_\tri$ of
  the center curve $K$ that is used to evaluate $\Einftyrhotilde$
  numerically, to obtain a reasonably good approximation of
  $\Einftyrho$? 
\item [(ii)] How small does the radius $\rho>0$ of the thin tubular
  scattering object $\Drho$ have to be in order that the leading order 
  term~$\Einftyrhotilde$ in the asymptotic perturbation formula
  \eqref{eq:GenAsyEinfty} is a sufficiently good approximation
  of~$\Einftyrho$? 
\end{enumerate}

\begin{figure}[t]
  \begin{subfigure}[c]{0.32\textwidth}
    \includegraphics[width=1.\textwidth]{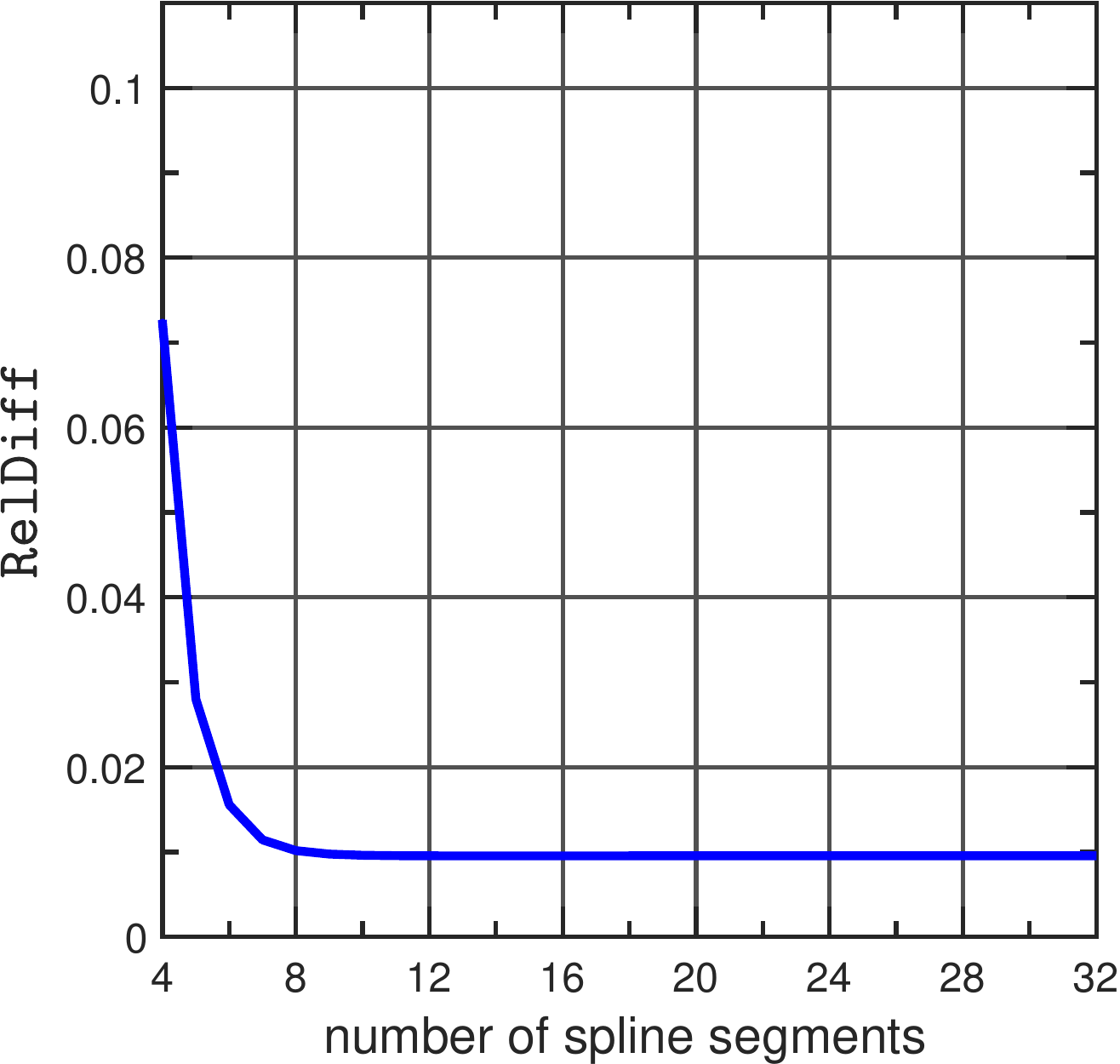}
  \end{subfigure}\,
  \begin{subfigure}[c]{0.32\textwidth}
    \includegraphics[width=1\textwidth]{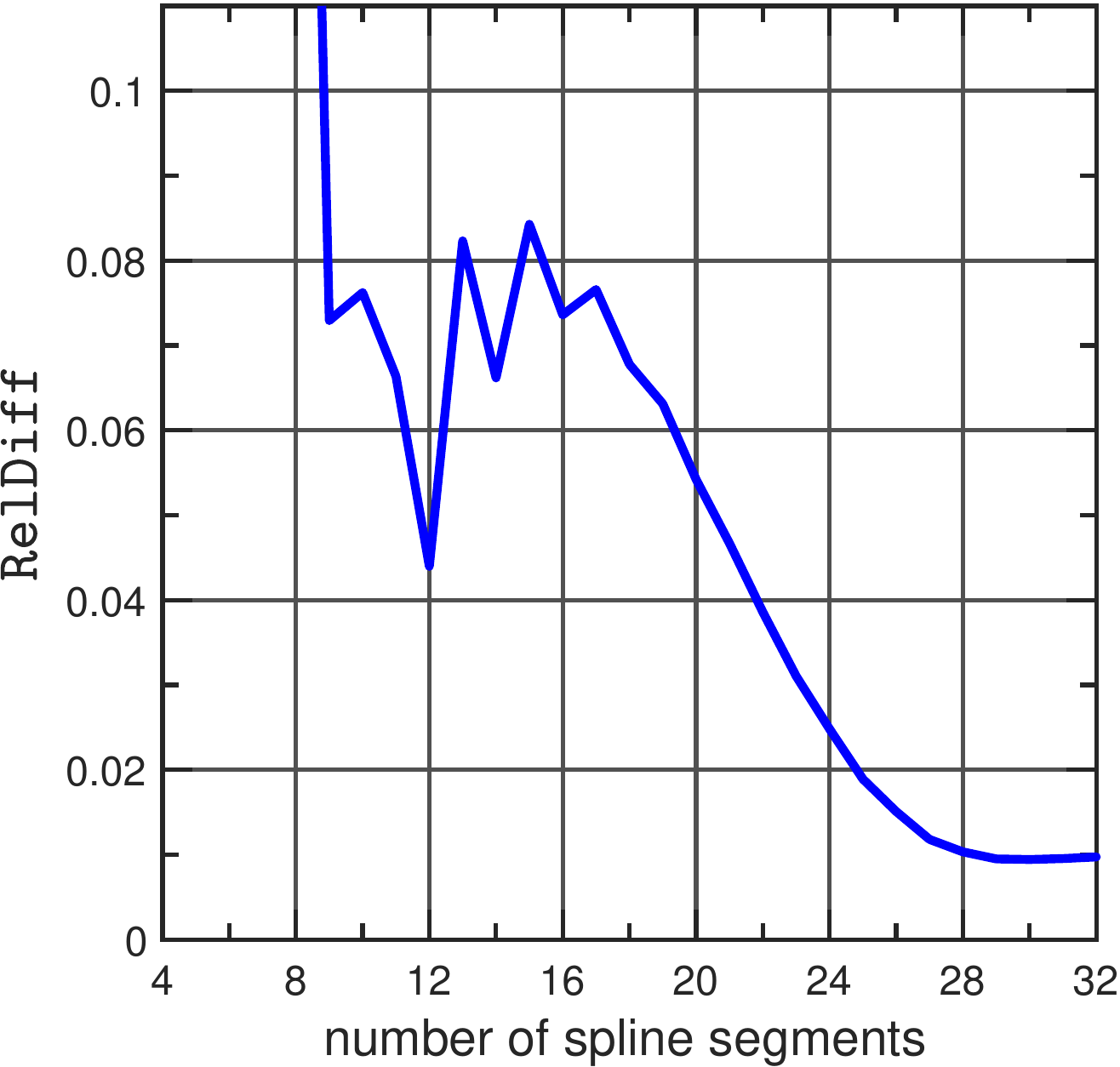}
  \end{subfigure}\,
  \begin{subfigure}[c]{0.32\textwidth}
    \includegraphics[width=1\textwidth]{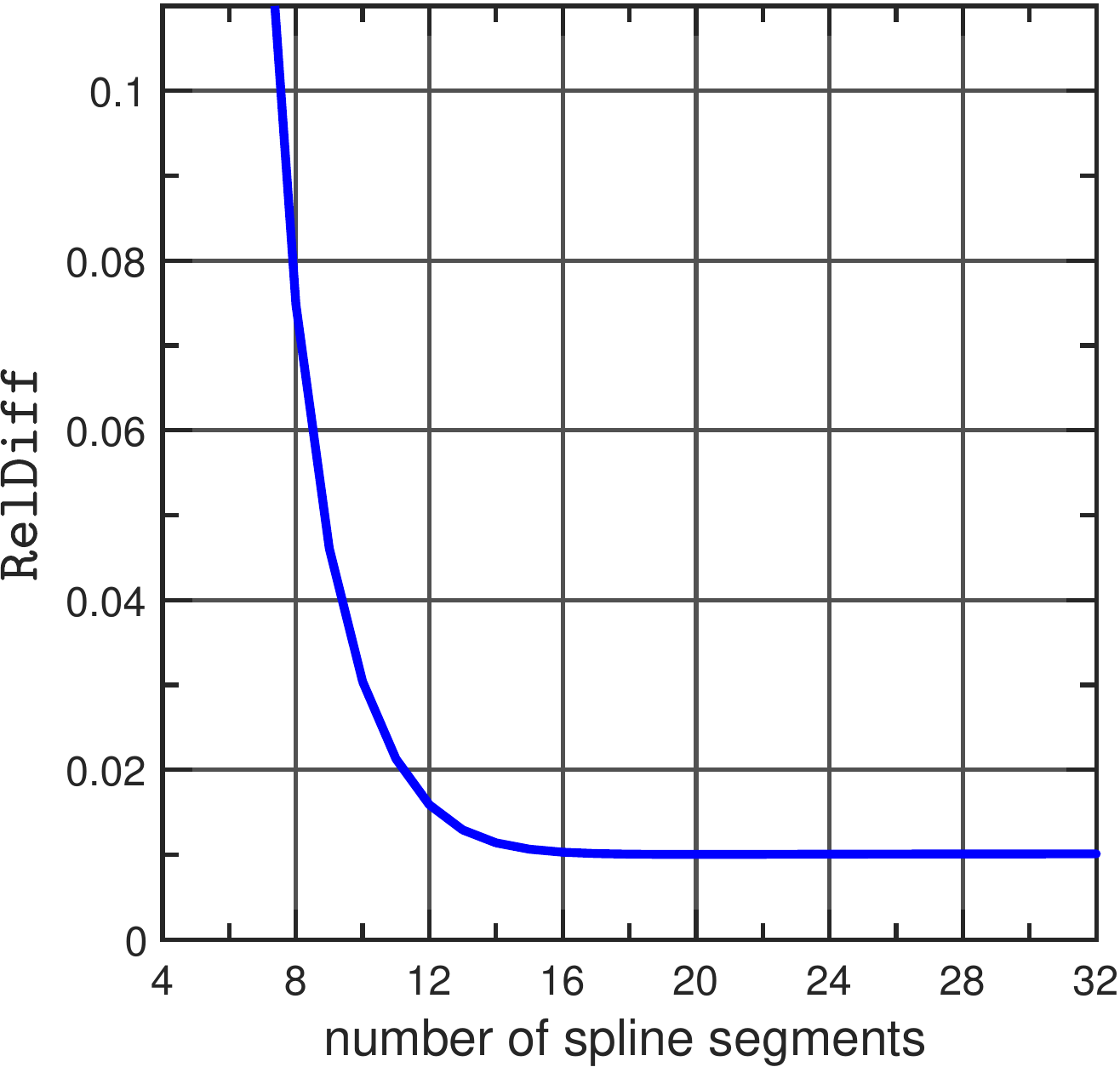}
  \end{subfigure}
  \caption{\small Relative difference $\RelDiff$ between $\Einftyrho$ and the
    leading order term $\Einftyrhotilde$ in \eqref{eq:DefErr} as a 
    function of the number of subsegments in the spline approximation
    $\bfp_\tri$ of the center curve $K$ for
    Examples~\ref{exa:Geometry1} (left), \ref{exa:Geometry2} (center),
    and \ref{exa:Geometry3} (right) with radius $\rho=0.03$.} 
  \label{fig:Accuracy1}
\end{figure}
Concerning the first question, we consider the scattering objects
in Examples~\ref{exa:Geometry1}, \ref{exa:Geometry2},
and \ref{exa:Geometry3} with radius $\rho=0.03$, and we evaluate
reference far field data $\Einftyrho$ using Bempp as described above.
In the corresponding Galerkin boundary element discretization we use a
triangulation of the boundary of the scatterer $\di\Drho$ with 
26698 triangles for Example~\ref{exa:Geometry1}, 
62116 triangles for Example~\ref{exa:Geometry2}, and
62116 triangles for Example~\ref{exa:Geometry3}. 
Then we consider a sequence of increasingly fine equidistant
partitions $\tri$ of $[0,1]$ as in \eqref{eq:DefTri}, we evaluate
$\Einftyrhotilde$ from \eqref{eq:Einftylead}, and we study
the decay of the relative difference $\RelDiff$ from \eqref{eq:DefErr}
as a function of the number of subsegments of the spline
approximations $\bfp_\tri$ of the center curves~$K$. 
We approximate the integrals in~\eqref{eq:Einftylead} using a
composite Simpson's rule with a fixed number of $M=11$ nodes on each
subinterval of $\tri$. 
The results of these tests are shown in Figure~\ref{fig:Accuracy1}. 
In each example the relative error decreases quickly until it reaches
its minimum value. 
Due to the asymptotic character of the expansion
\eqref{eq:GenAsyEinfty}, and due to numerical error in the numerical
approximation of $\Einftyrho$ obtained by Bempp, we do not expect the
relative error $\RelDiff$ to decay to zero.
A relatively low number of spline segments suffices in all three
examples to obtain less than $2\%$ relative difference. 
Of course this number depends on the shape of the center curve.
We note that while the simulation of $\Einftyrho$ using Bempp is
computationally quite demanding, the evaluation of $\Einftyrhotilde$
using \eqref{eq:Einftylead} is simple and extremely fast.

\begin{figure}[t]
  \begin{subfigure}[c]{0.32\textwidth}
    \includegraphics[width=1.\textwidth]{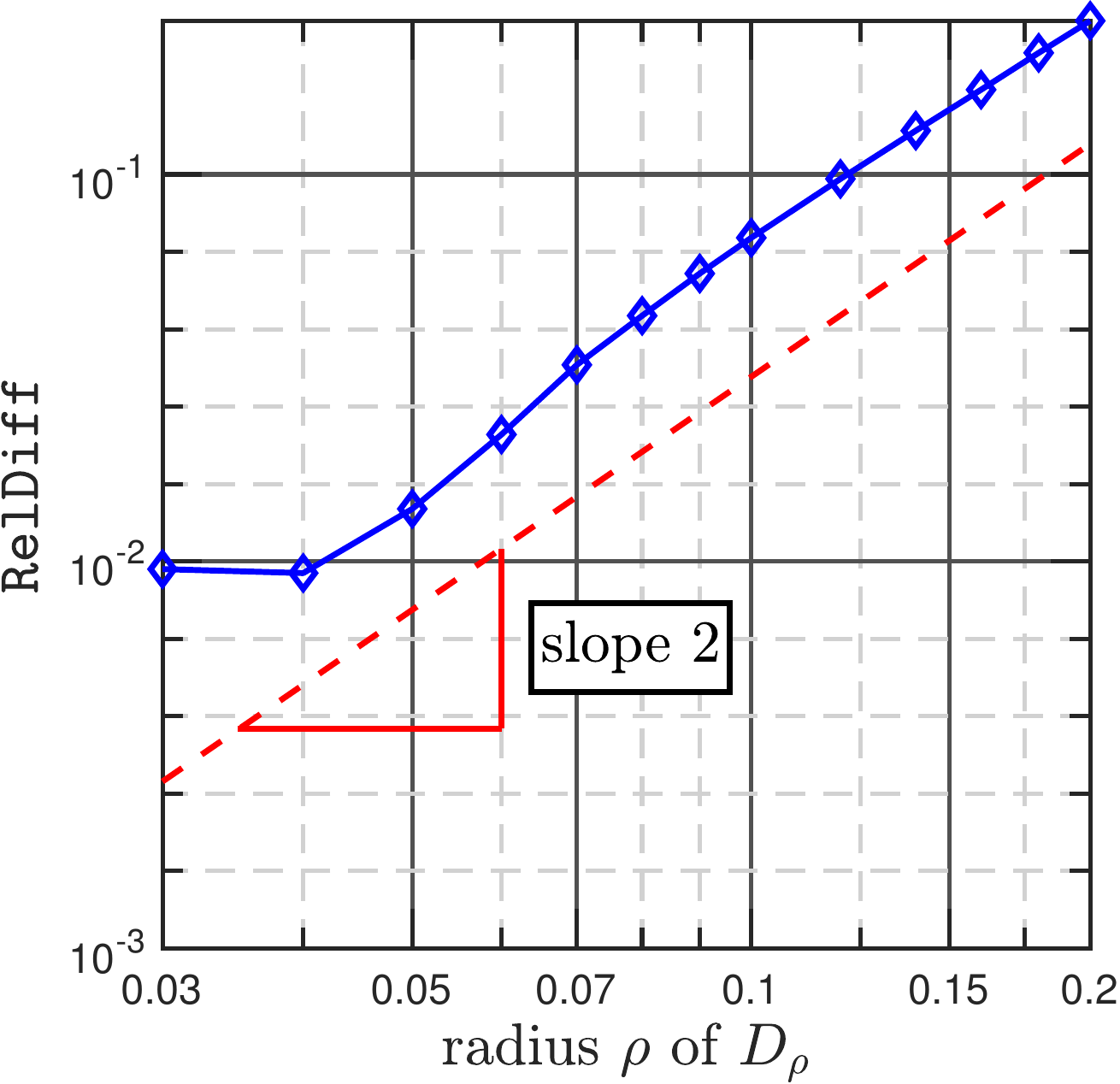}
  \end{subfigure}\,
  \begin{subfigure}[c]{0.32\textwidth}
    \includegraphics[width=1\textwidth]{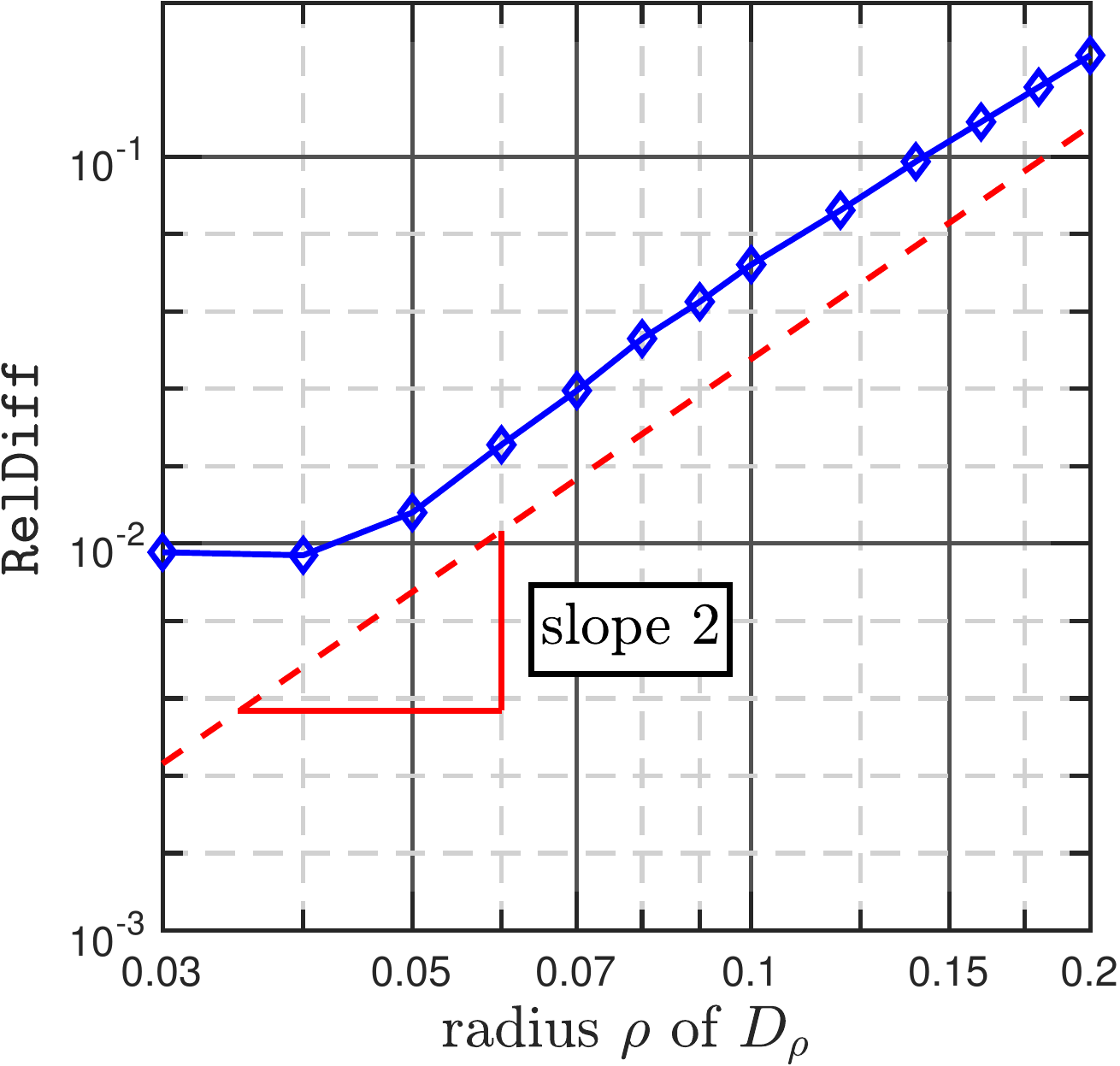}
  \end{subfigure}\,
  \begin{subfigure}[c]{0.32\textwidth}
    \includegraphics[width=1\textwidth]{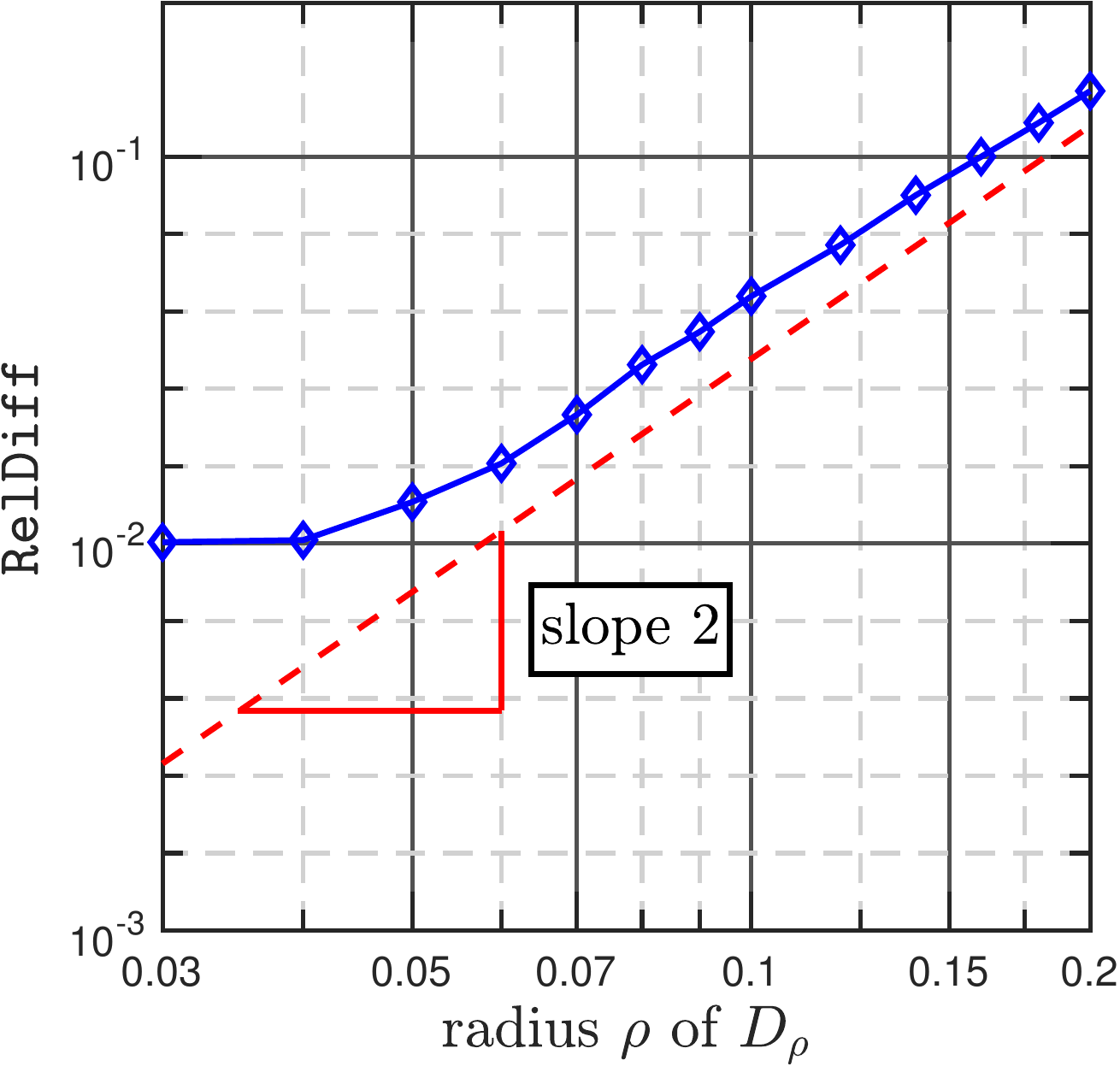}
  \end{subfigure}
  \caption{\small Relative difference $\RelDiff$ (solid blue) between
    $\Einftyrho$ and the leading order term $\Einftyrhotilde$ in
    \eqref{eq:DefErr} as a function of the radius $\rho$ of the thin
    tubular scatterer $\Drho$ for Examples~\ref{exa:Geometry1} (left),
    \ref{exa:Geometry2} (center), and \ref{exa:Geometry3} (right).
    For comparison the plot contains a line of slope $2$ (dashed red).} 
  \label{fig:Accuracy2}
\end{figure}
Concerning the second question from above we again generate reference
far field data $\Einftyrho$ for the thin tubular scattering objects in
Examples~\ref{exa:Geometry1}--\ref{exa:Geometry3} using Bempp, but now
for a whole range of radii 
\begin{equation*}
  \rho 
  \in \{ 0.03,\, 0.04,\, 0.05,\, 0.06,\, 0.07,\, 0.08,\, 0.09,\, 
  0.1,\, 0.12,\, 0.14,\, 0.16,\, 0.18,\, 0.2 \} \,.
\end{equation*}
Here we use increasingly fine triangulations of the boundaries of the
scatterers $\di\Drho$.
We also evaluate the approximations $\Einftyrhotilde$ for these values
of $\rho$ using $29$ spline segments in the spline
approximations~$\bfp_\tri$ of the center curves $K$. 
In Figure~\ref{fig:Accuracy2} we show plots of the relative difference
$\RelDiff$ as a function of $\rho$ (solid blue) for these three
examples. 
For comparison these plots contain a line of slope $2$ (dashed red).
The relative error decays approximately of order $\Ocal(\rho^2)$. 
We note that our theoretical results in
Theorem~\ref{thm:GeneralAsymptotics} do not predict any rate of
convergence.

\subsection{Reconstruction of the center curve $K$ using
  Algorithm~\ref{alg:reconstruction}}  
We return to the inverse problem and apply
Algorithm~\ref{alg:reconstruction} to reconstruct the center 
curves $K$ of the thin tubular scattering objects~$\Drho$ in
Examples~\ref{exa:Geometry1}--\ref{exa:Geometry3} from observations of
a single electric far field pattern $\Einftyrho$.
In all three examples the radius of the scattering object $\Drho$
is $\rho=0.03$.
The previous examples show that $\Einftyrho$ is well approximated by
$\Einftyrhotilde$ in this regime.
We assume that the plane wave incident field $\Ei$ (i.e., the wave
number $k$, the direction of propagation $\bftheta$, and the
polarization~$\bfA$), the shape and the radius of the cross-sections
$\Drho'$ of the scattering objects (i.e., in particular $\rho$) and
the material parameters of the scattering objects (i.e., the relative
electric permittivity $\eps_r$ and the relative magnetic permeability
$\mu_r$) are known \emph{a priori}. 

We simulate the far field data $\Einftyrho$ for each of the three
examples using Bempp, where we use triangulations of the boundaries of 
the tubes $\di\Drho$ with 
26698 triangles for Example~\ref{exa:Geometry1}, 
62116 triangles for Example~\ref{exa:Geometry2}, and
62116 triangles for Example~\ref{exa:Geometry3}. 
The values of $\Einftyrho$ are evaluated on the equiangular grid on
$\Stwo$ from \eqref{eq:StwoGrid} with $N=10$. 

We choose the following parameters in
Algorithm~\ref{alg:reconstruction}: 
\begin{itemize}
\item We use $n=30$ control points (i.e. $29$ spline segments) in
  the spline approximation $\bfp_\tri$ of the unknown center curve $K$.
\item We initialize the regularization parameters in step~2 by
  $\alpha_1 = 0.2$ and $\alpha_2 = 0.9$. 
\item We choose $\smax=1$ in the golden section line search in
  step~5 and we terminate each line search after a fixed number of $10$
  steps. 
\end{itemize}

The results are shown in Figures~\ref{fig:Recon1}--\ref{fig:Recon3}.
Here, the top-left plots show the initial guess, and the bottom-right
plots show the final reconstruction.
The remaining four plots show intermediate approximations of the
iterative reconstruction procedure. 
Each plot contains the exact center curve $K$ (solid blue) and the
current approximation $\bfp_{\tri,\ell}$ of the reconstruction
algorithm after $\ell$ iterations (solid red with stars).
Furthermore, we have included projections of these curves onto the
three coordinate planes to enhance the three-dimensional perspective.  

\begin{example} 
  \label{exa:Recon1}
  \begin{figure}[t]
    \centering
    \begin{subfigure}{.32\textwidth}
      \includegraphics[width=1\textwidth]{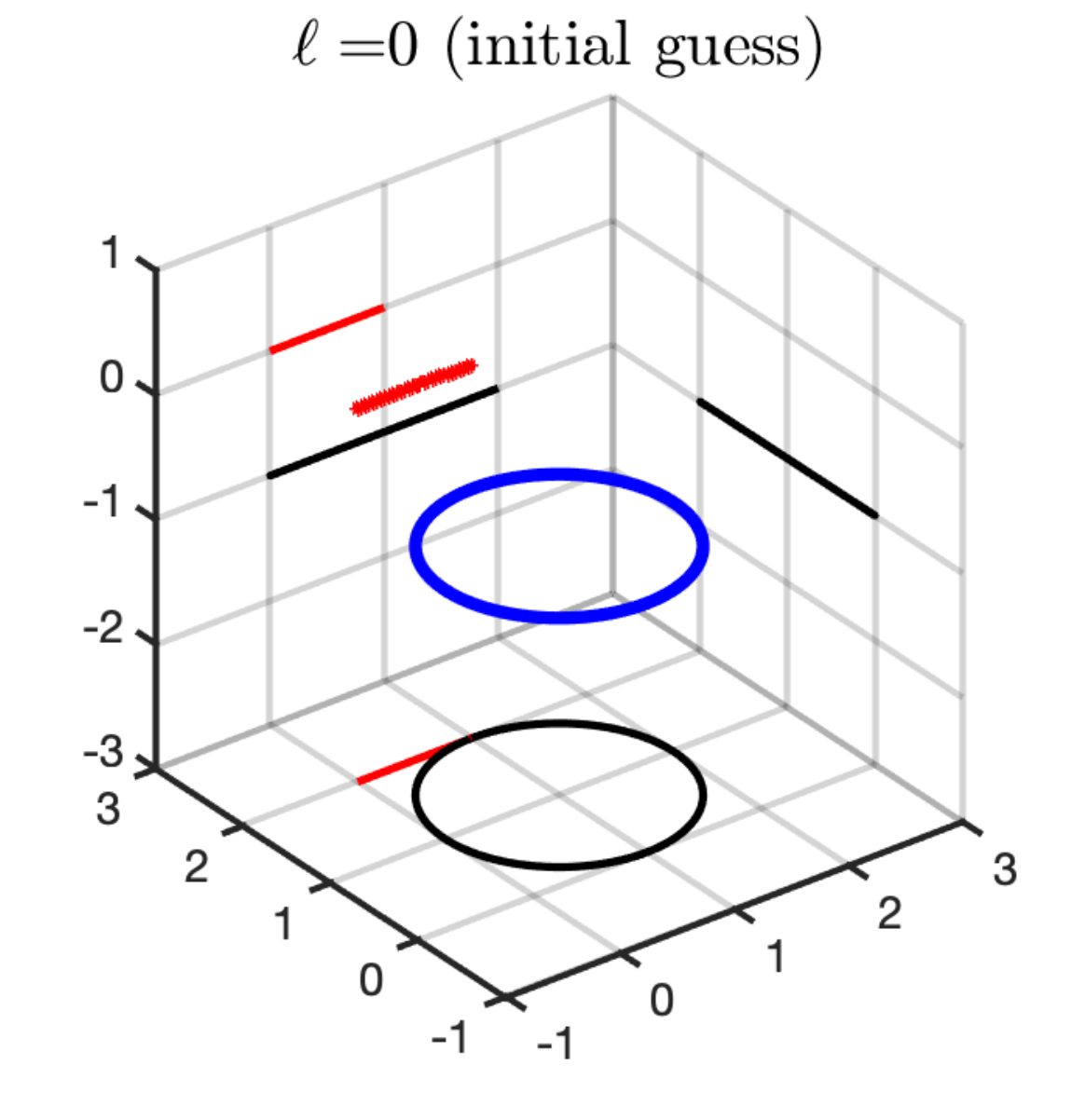}
    \end{subfigure}
    \begin{subfigure}{.32  \textwidth}           
      \includegraphics[width=1.\textwidth]{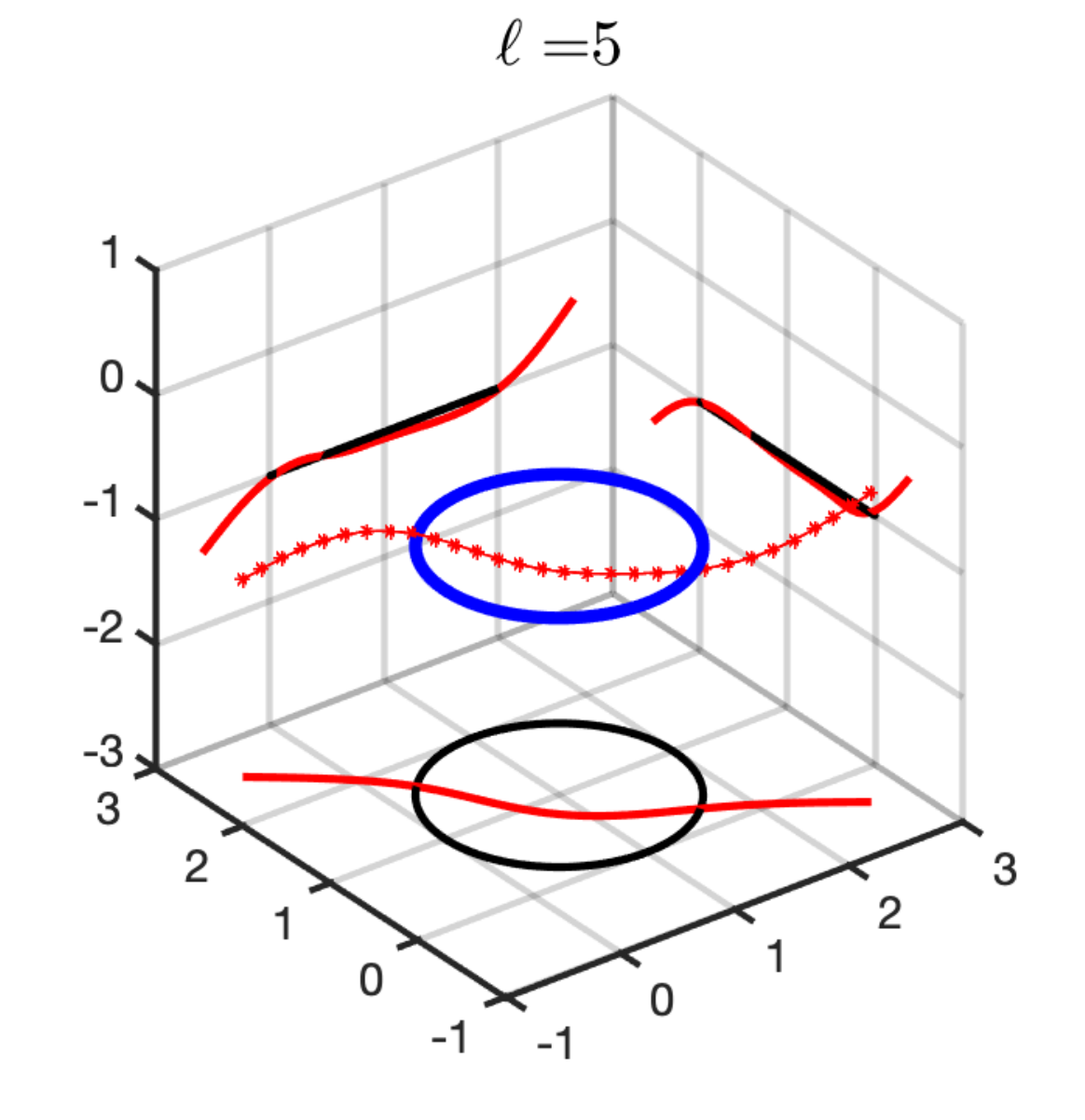}
    \end{subfigure}
    \begin{subfigure}{.32\textwidth}
      \includegraphics[width=1.\textwidth]{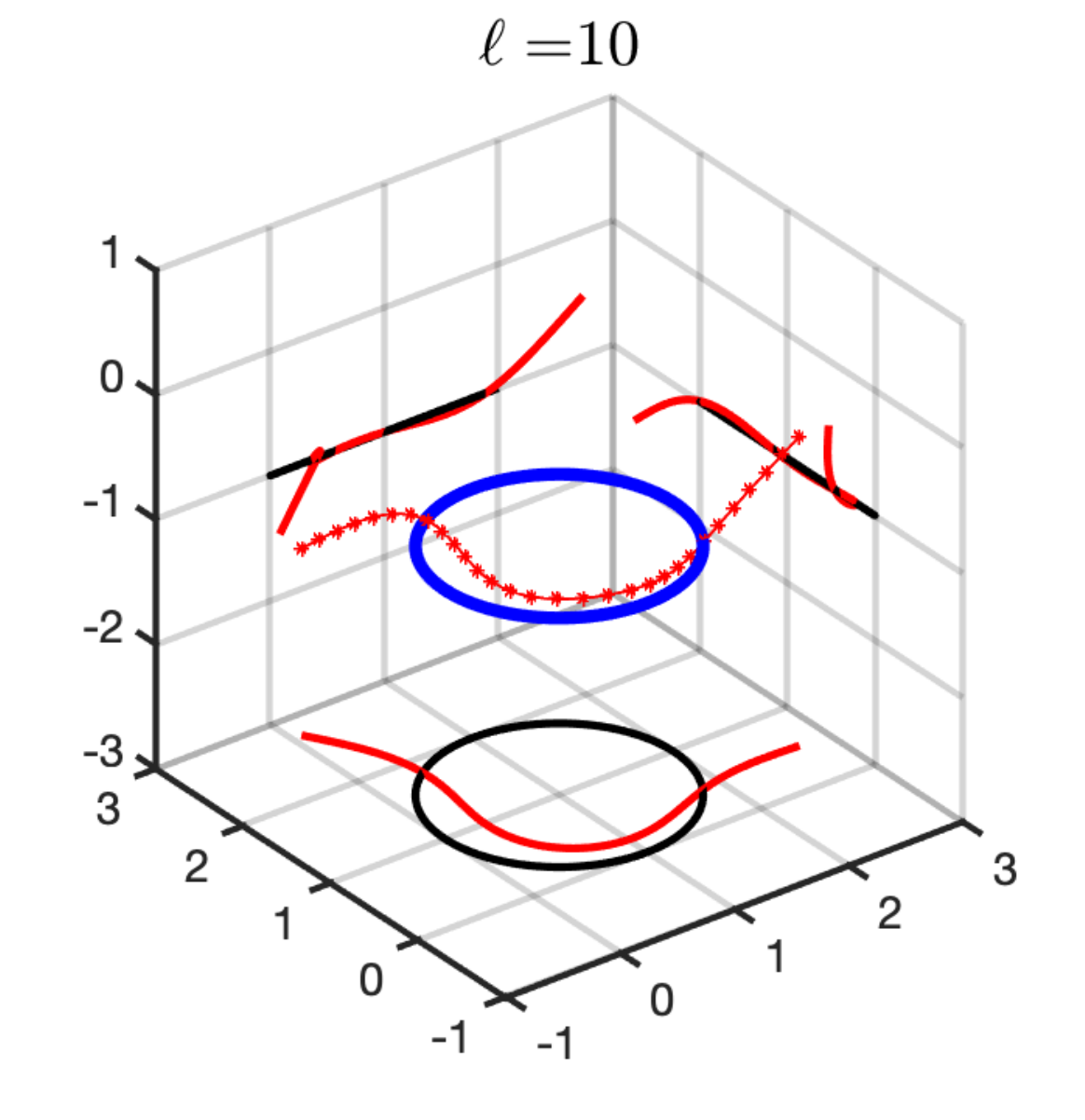}
    \end{subfigure}
    \begin{subfigure}{.32\textwidth}
      \includegraphics[width=1.\textwidth]{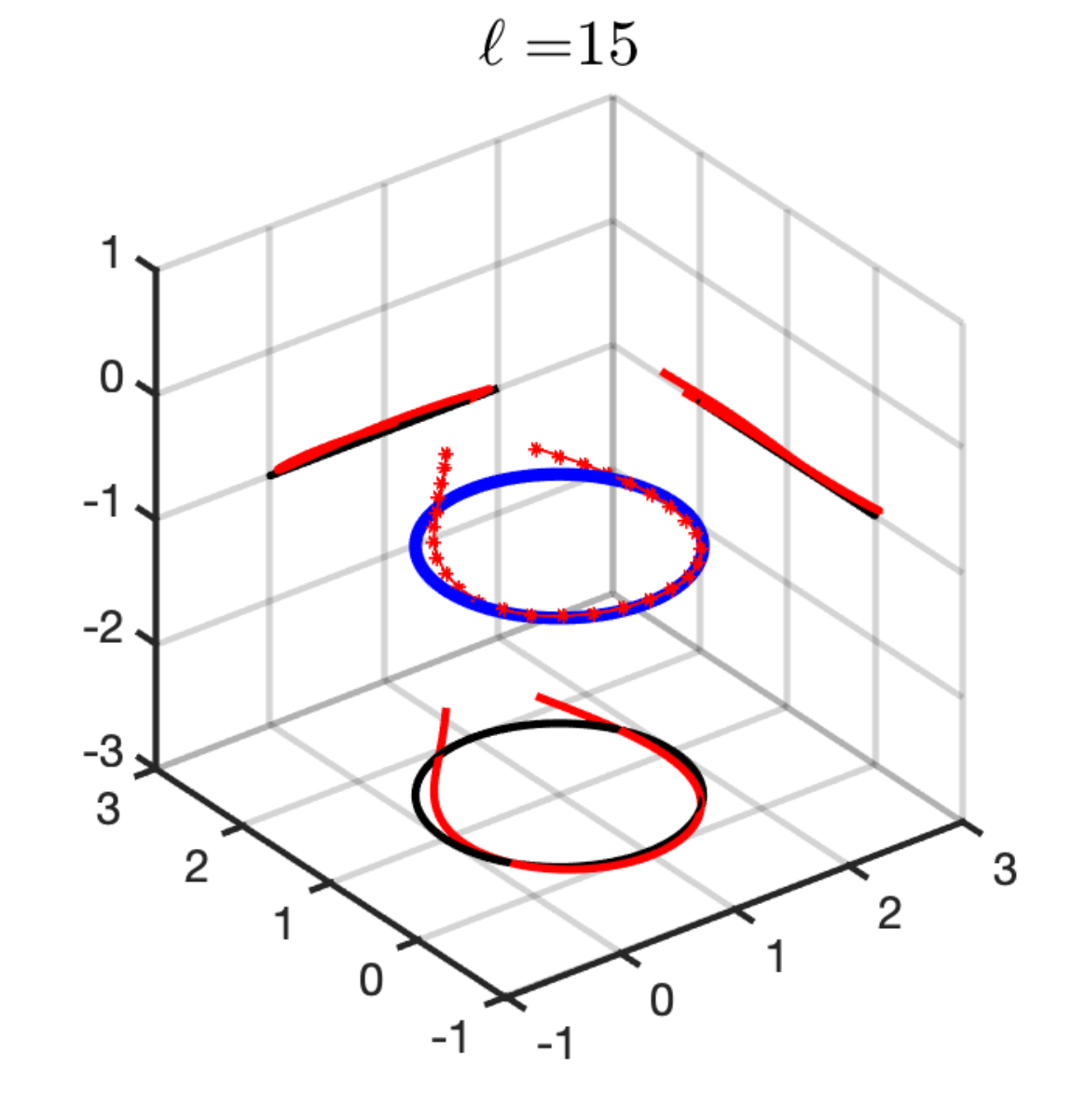}
    \end{subfigure}
    \begin{subfigure}{.32\textwidth}
      \includegraphics[width=1.\textwidth]{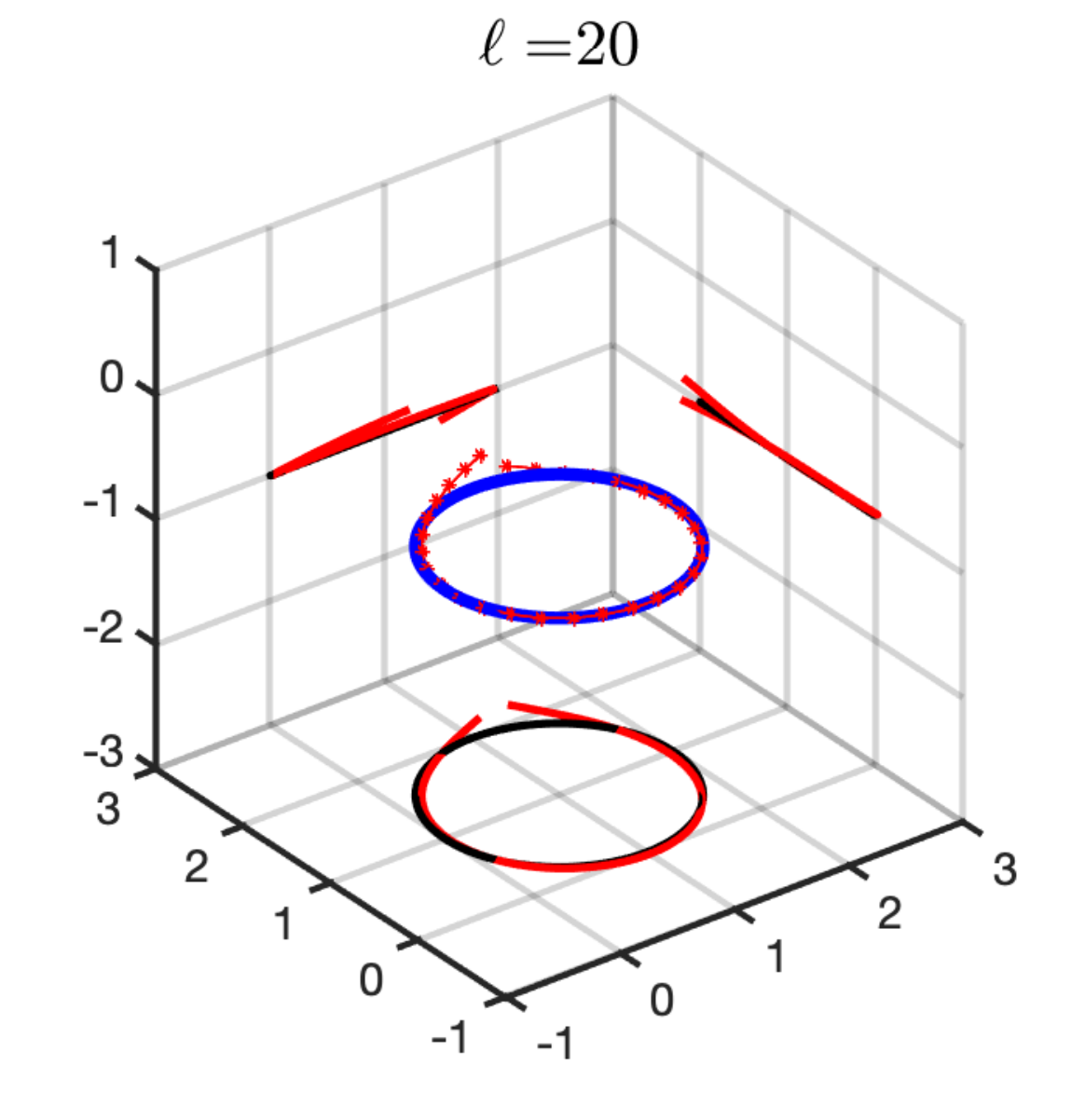}
    \end{subfigure}
    \begin{subfigure}{.32\textwidth}
      \includegraphics[width=1.\textwidth]{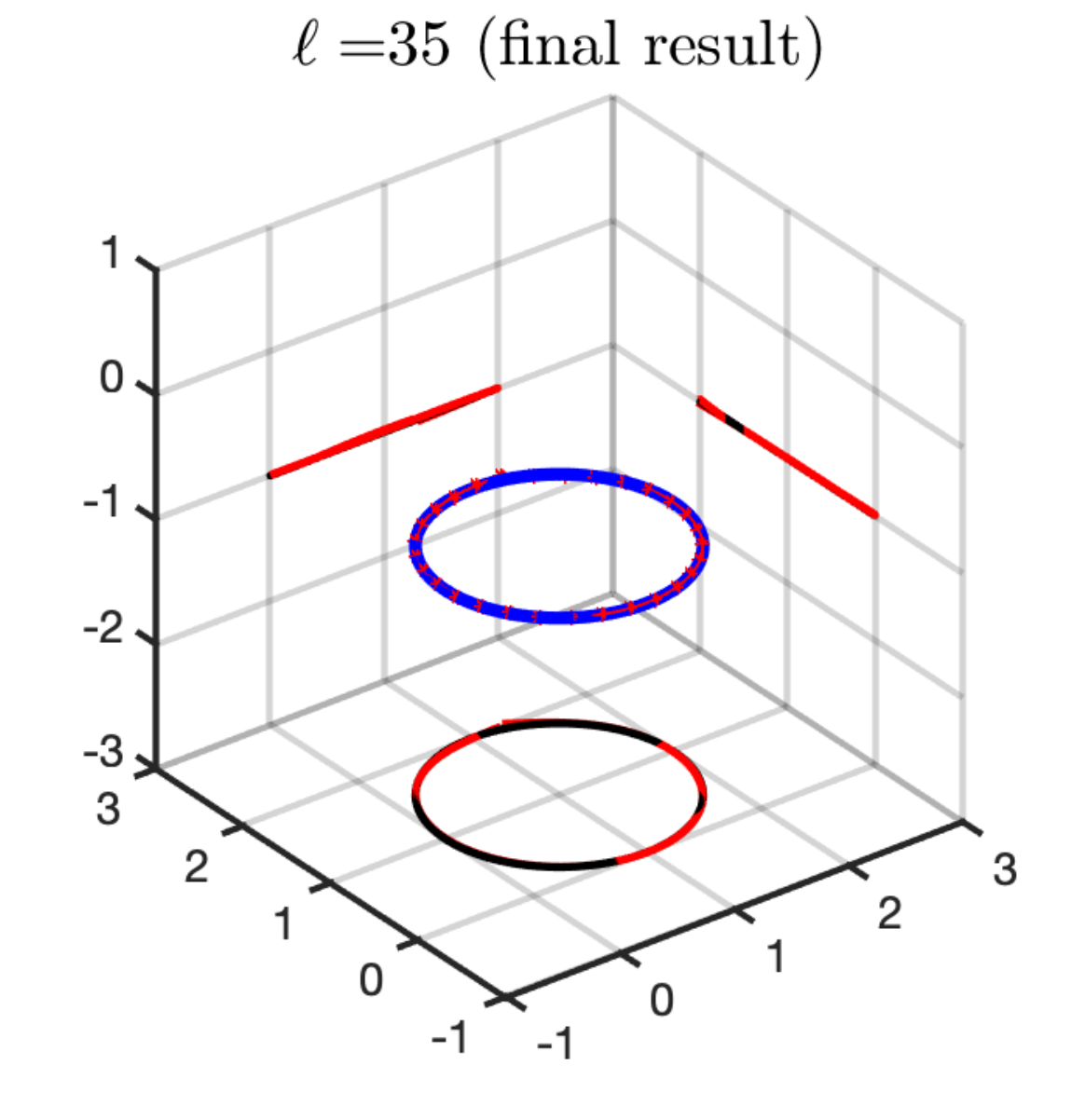}
    \end{subfigure}
    \caption{\small Reconstruction of the toroidal scatterer from
      Example~\ref{exa:Geometry1}. 
      The top-left plot shows the initial guess, and the
      bottom-right plot shows the final reconstruction.} 
    \label{fig:Recon1}
  \end{figure}
  We consider the setting from Example~\ref{exa:Geometry1}. 
  The initial guess is a straight line segment connecting the points 
  $[0,2,0]^\top$ and $[1,2,0]^\top$.
  The reconstruction algorithm stops after $35$ iterations.
  The initial guess, some intermediate steps and the final result of
  the reconstruction algorithm are shown in Figure~\ref{fig:Recon1}.
  The final reconstruction is very close to the exact center
  curve~$K$.~\hfill$\lozenge$ 
\end{example} 

\begin{example}
  \label{exa:Recon2}
  \begin{figure}[t]
    \centering
    \begin{subfigure}{.32\textwidth}
      \includegraphics[width=1\textwidth]{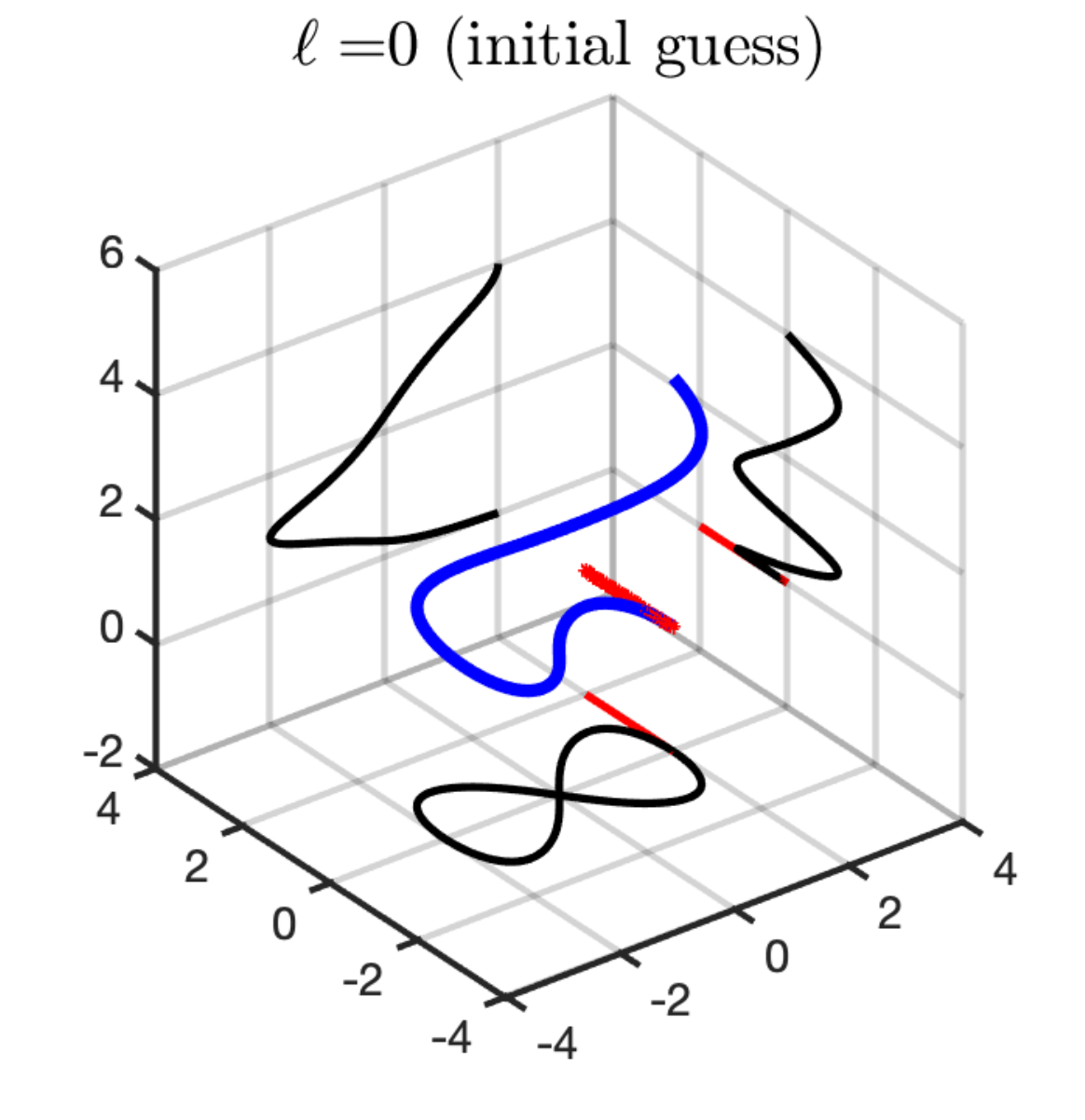}
    \end{subfigure}
    \begin{subfigure}{.32  \textwidth}           
      \includegraphics[width=1.\textwidth]{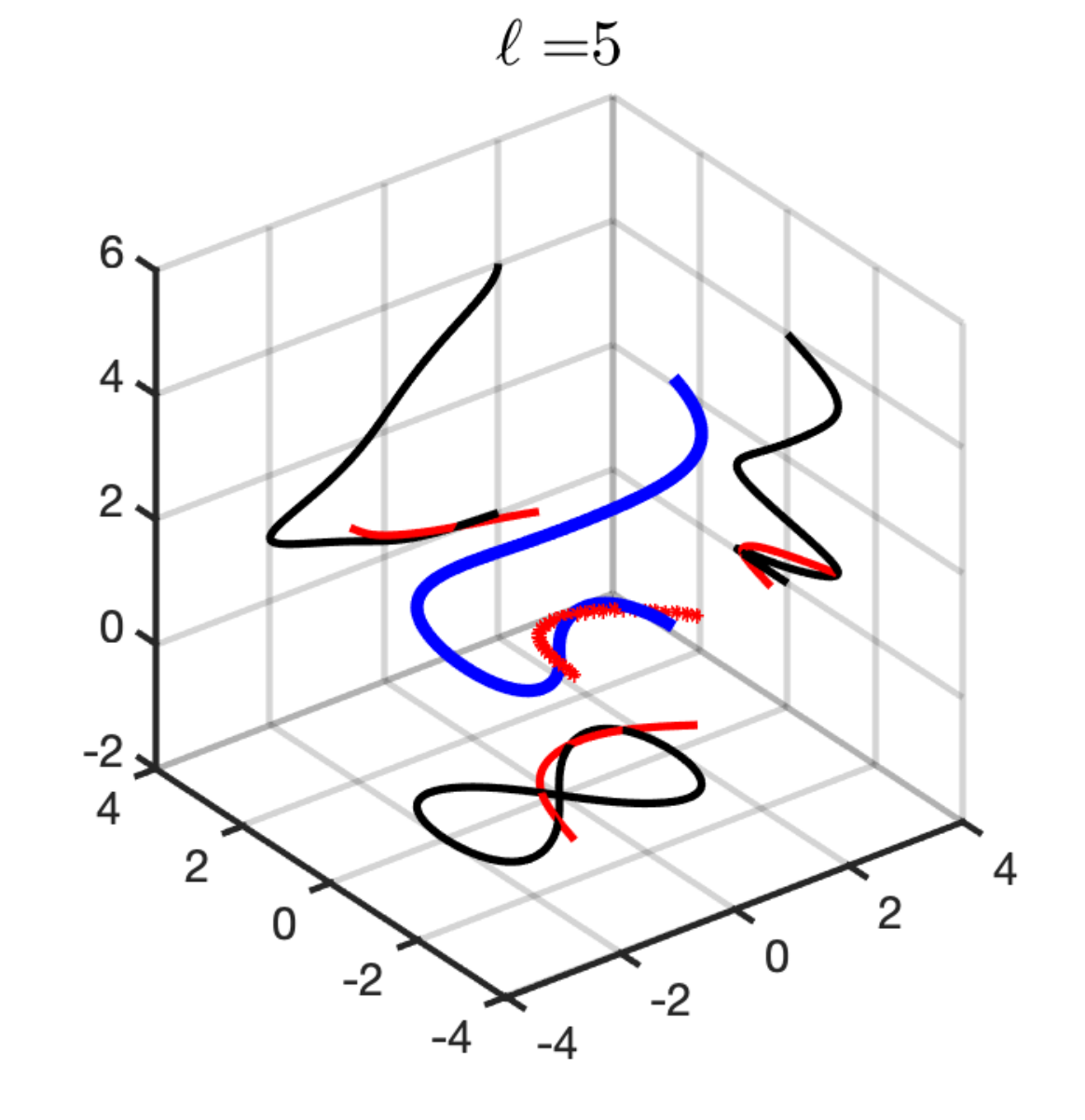}
    \end{subfigure}
    \begin{subfigure}{.32\textwidth}
      \includegraphics[width=1.\textwidth]{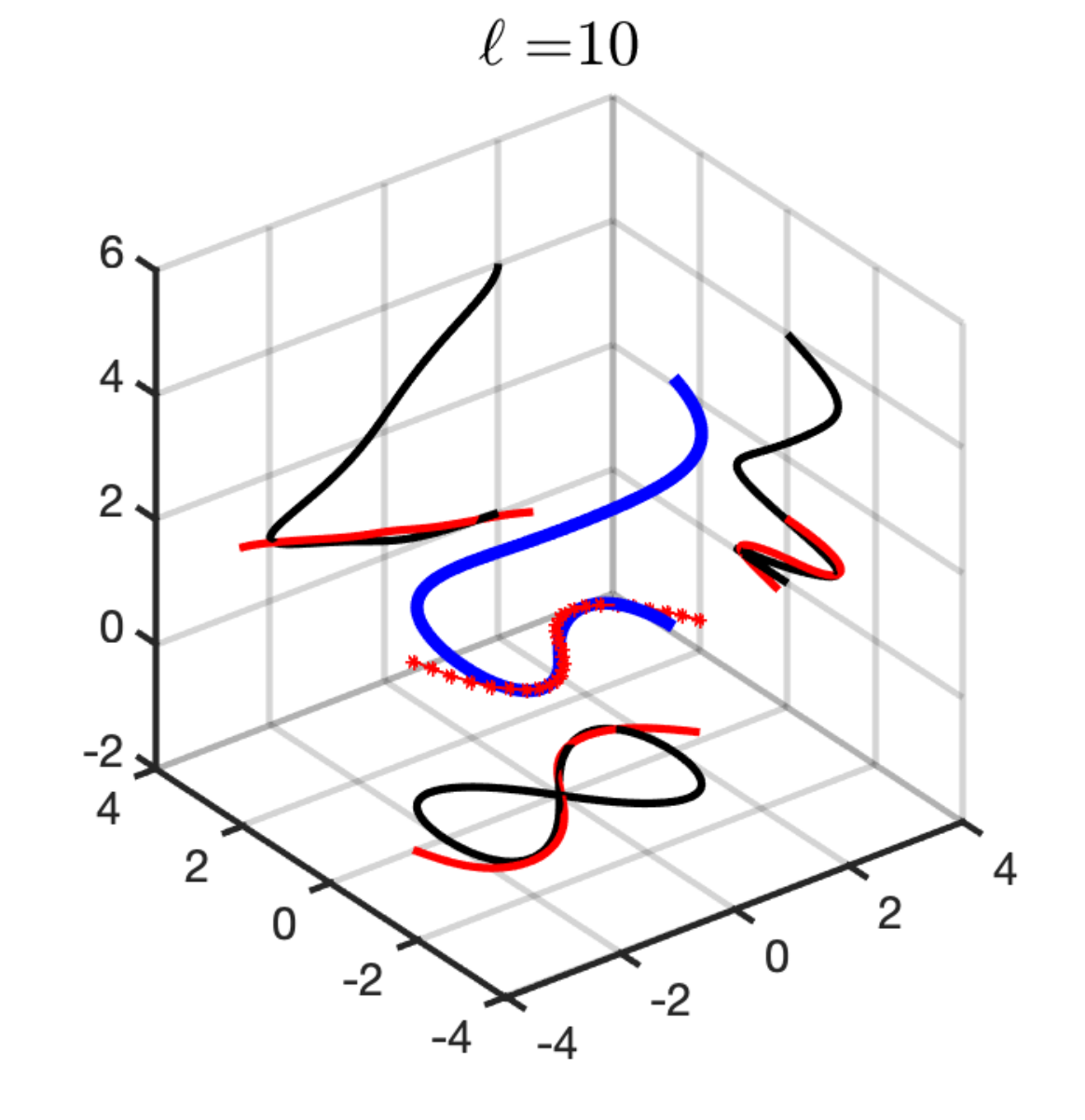}
    \end{subfigure}
    \begin{subfigure}{.32\textwidth}
      \includegraphics[width=1.\textwidth]{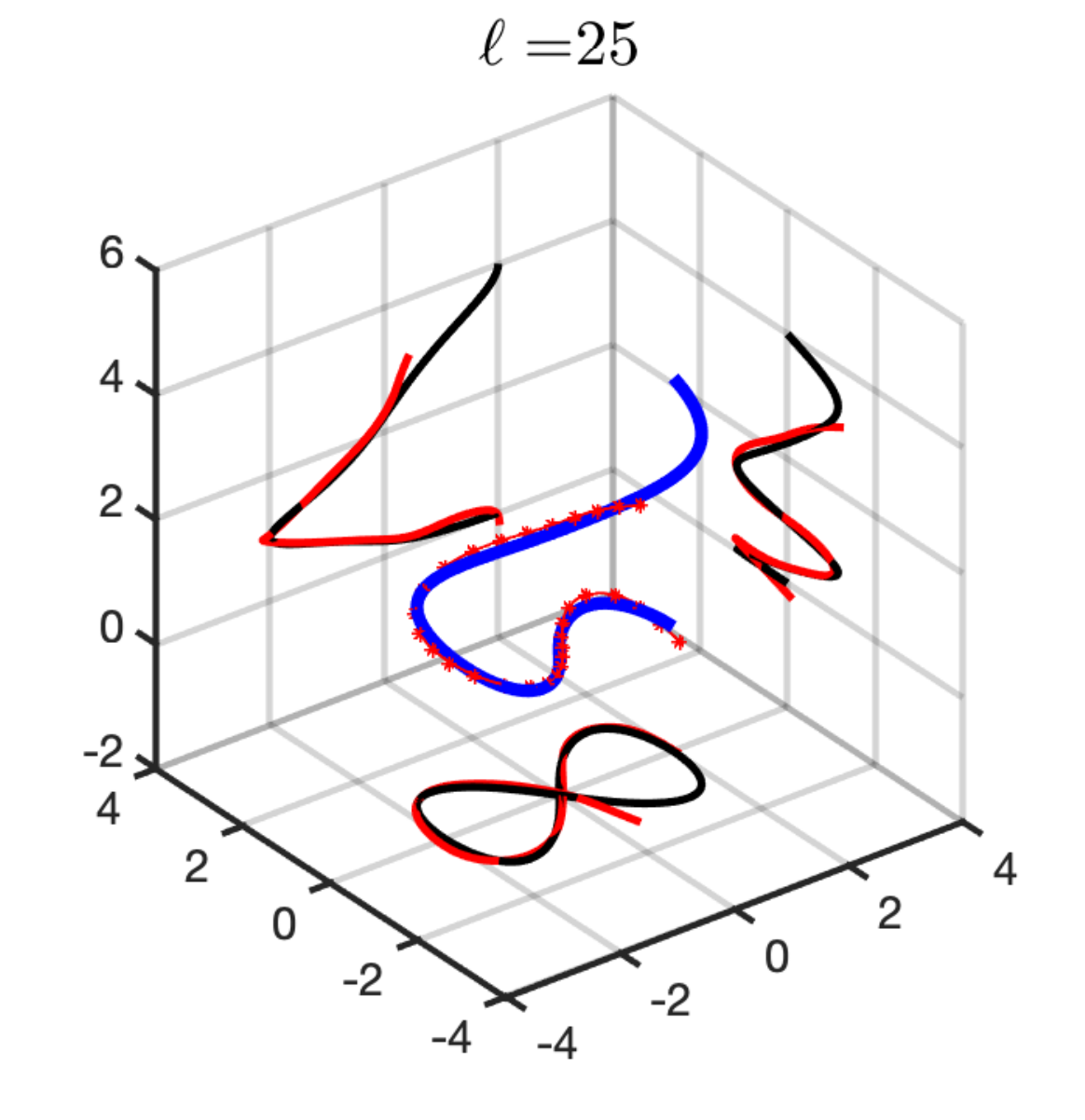}
    \end{subfigure}
    \begin{subfigure}{.32\textwidth}
      \includegraphics[width=1.\textwidth]{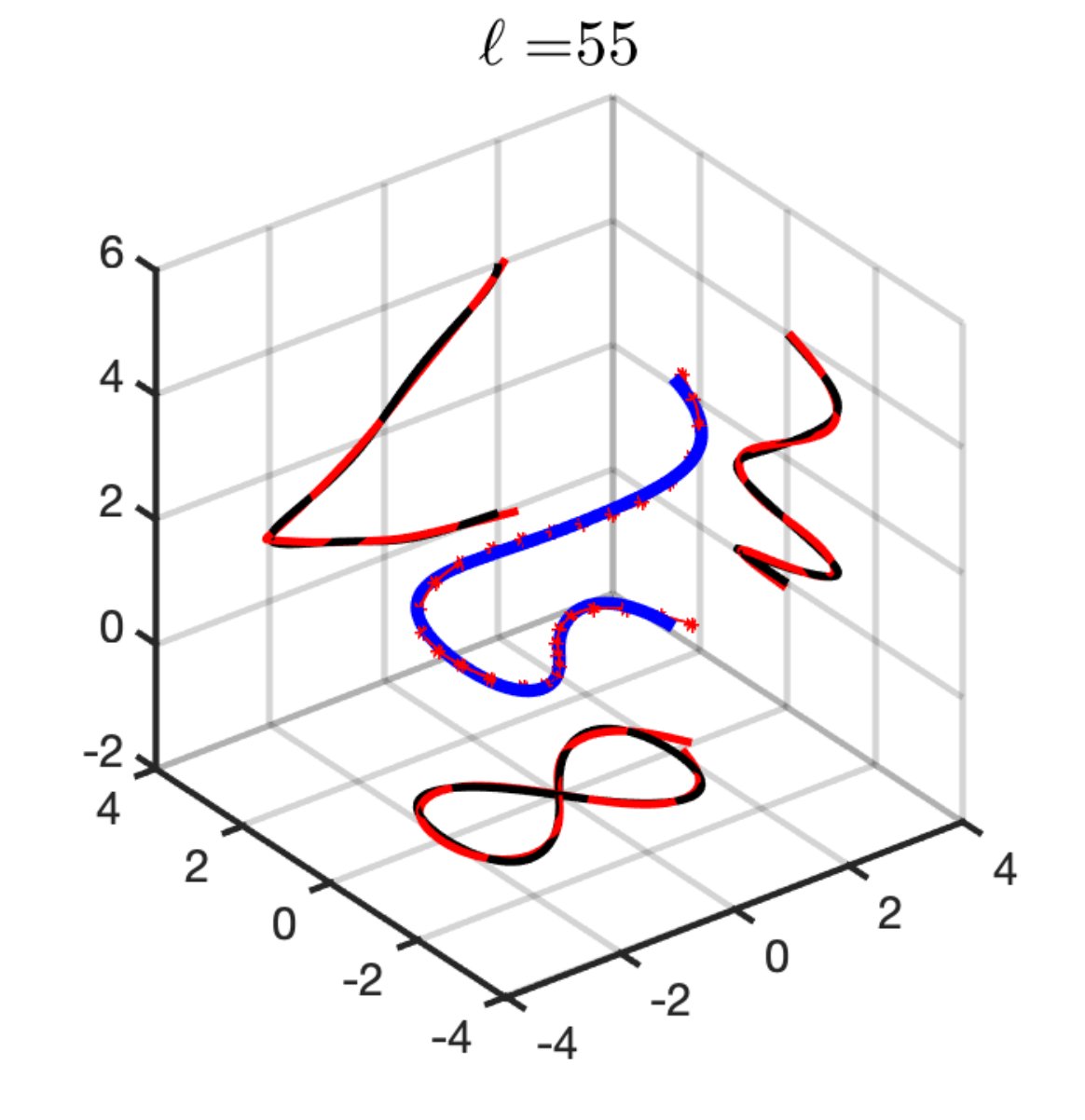}
    \end{subfigure}
    \begin{subfigure}{.32\textwidth}
      \includegraphics[width=1.\textwidth]{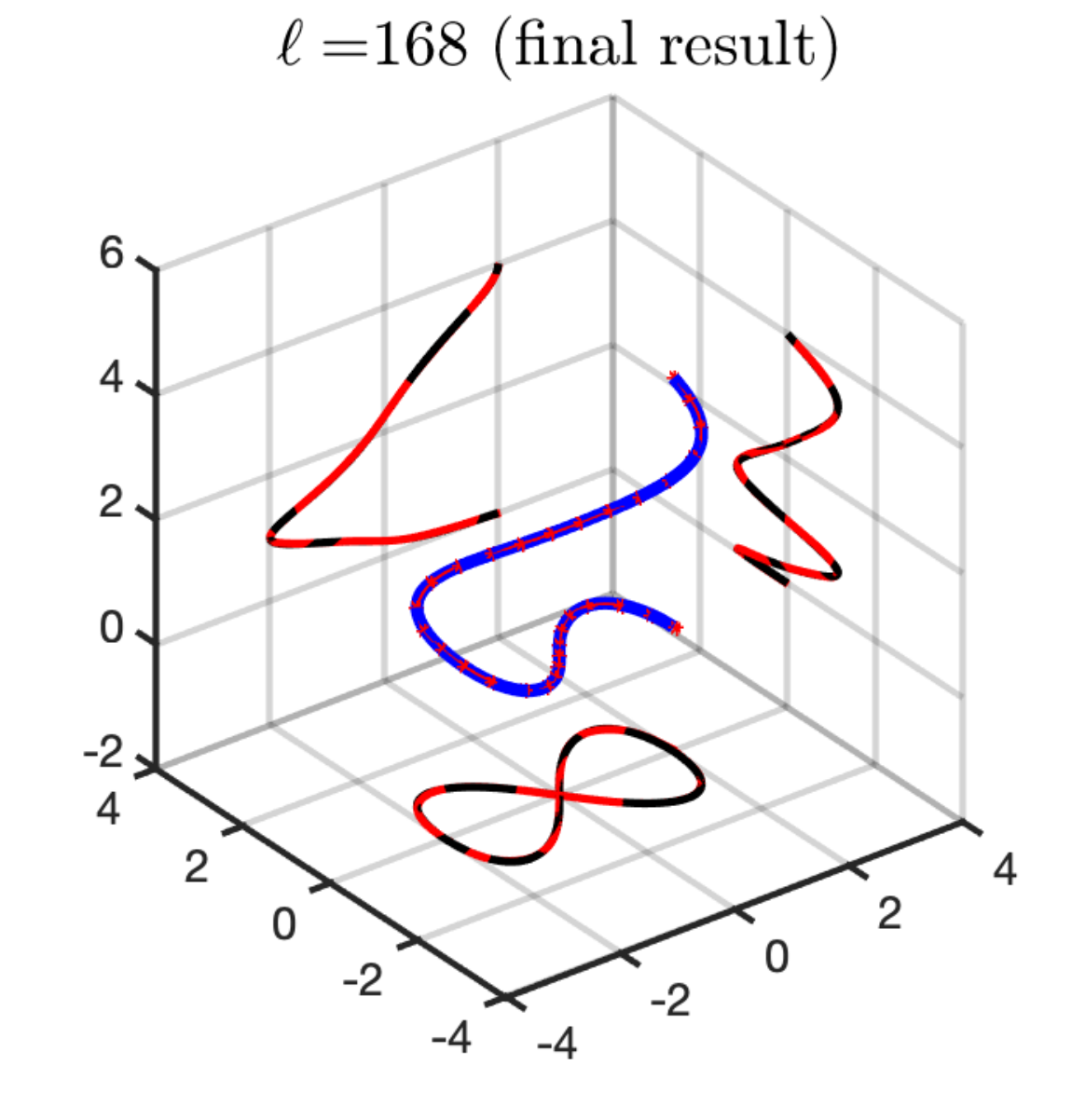}
    \end{subfigure}
    \caption{\small Reconstruction of the thin tubular scatterer from
      Example~\ref{exa:Geometry2}. 
      The top-left plot shows the initial guess, and the
      bottom-right plot shows the final reconstruction.}
    \label{fig:Recon2}
  \end{figure}
  We consider the setting from Example~\ref{exa:Geometry2}. 
  The initial guess is a straight line segment connecting the points 
  $[2,0,0]^\top$ and $[2,2,0]^\top$.
  The reconstruction algorithm stops after $168$ steps.
  The initial guess, some intermediate steps and the final result of
  the reconstruction algorithm are shown in Figure~\ref{fig:Recon2}.
  Again the final reconstruction is very close to the exact center
  curve~$K$. 
  \hfill$\lozenge$
\end{example} 

\begin{example}
  \label{exa:Recon3}
  \begin{figure}[t]
    \centering
    \begin{subfigure}{.32\textwidth}
      \includegraphics[width=1\textwidth]{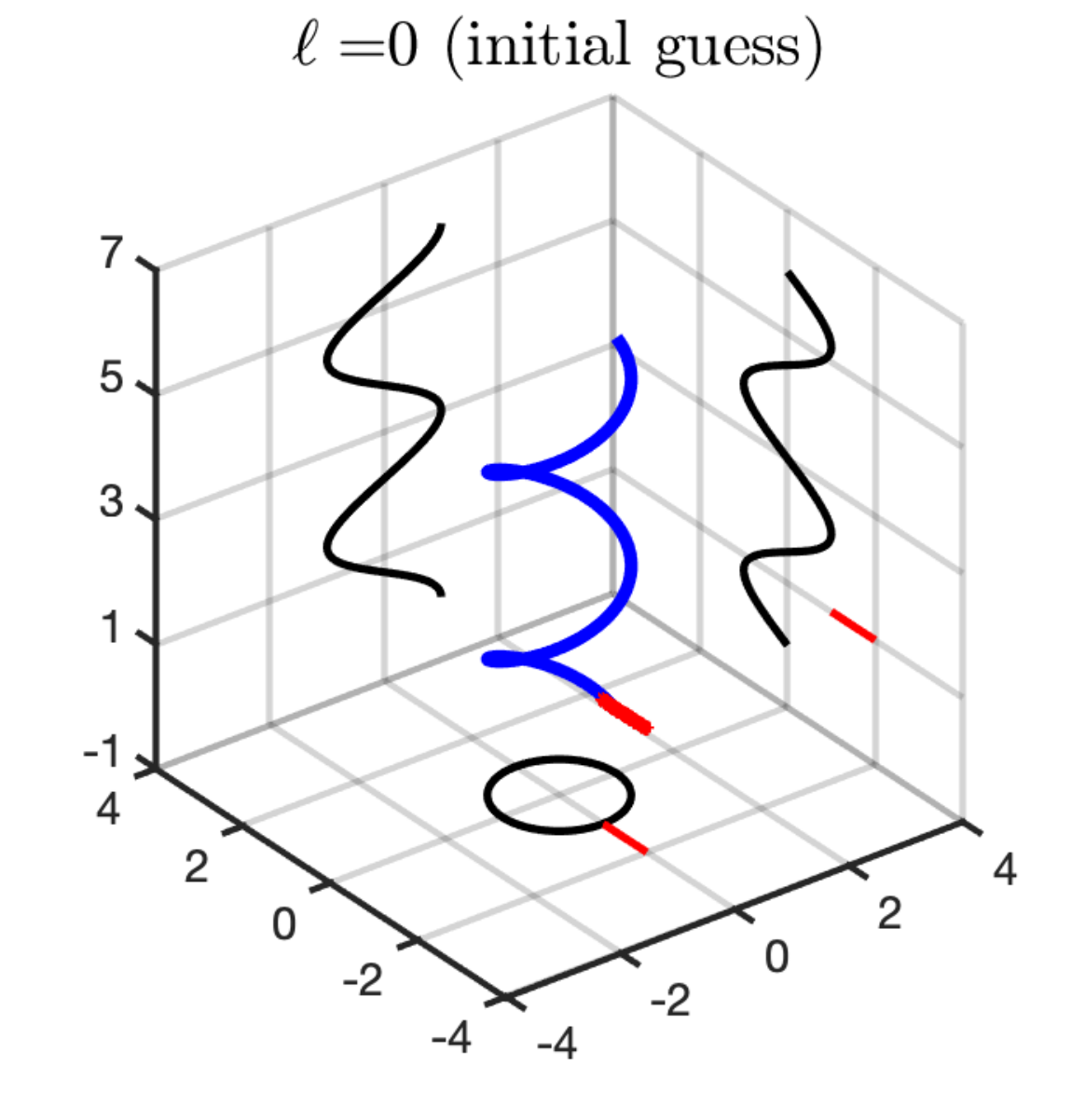}
    \end{subfigure}
    \begin{subfigure}{.32  \textwidth}           
      \includegraphics[width=1.\textwidth]{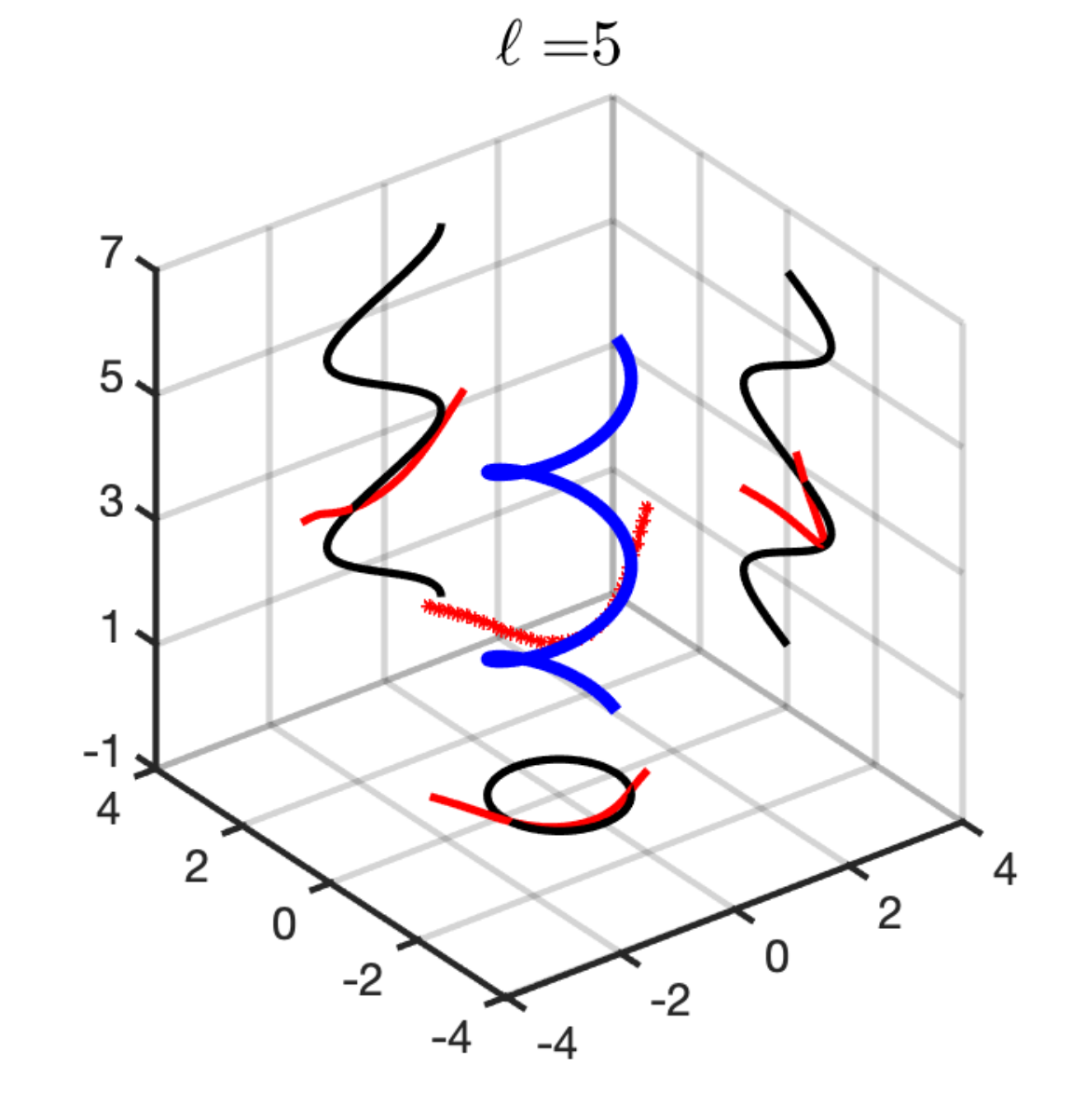}
    \end{subfigure}
    \begin{subfigure}{.32\textwidth}
      \includegraphics[width=1.\textwidth]{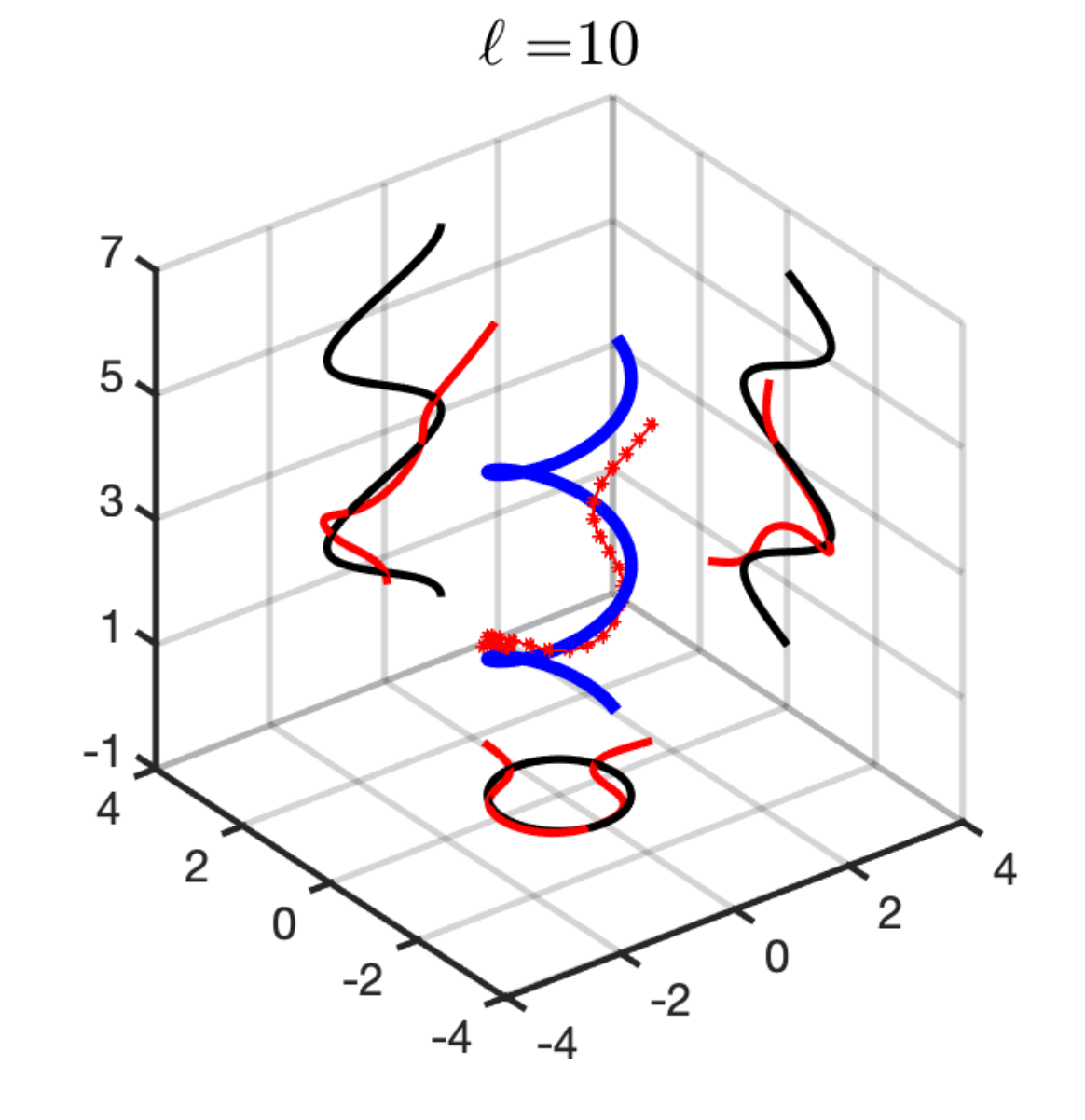}
    \end{subfigure}
    \begin{subfigure}{.32\textwidth}
      \includegraphics[width=1.\textwidth]{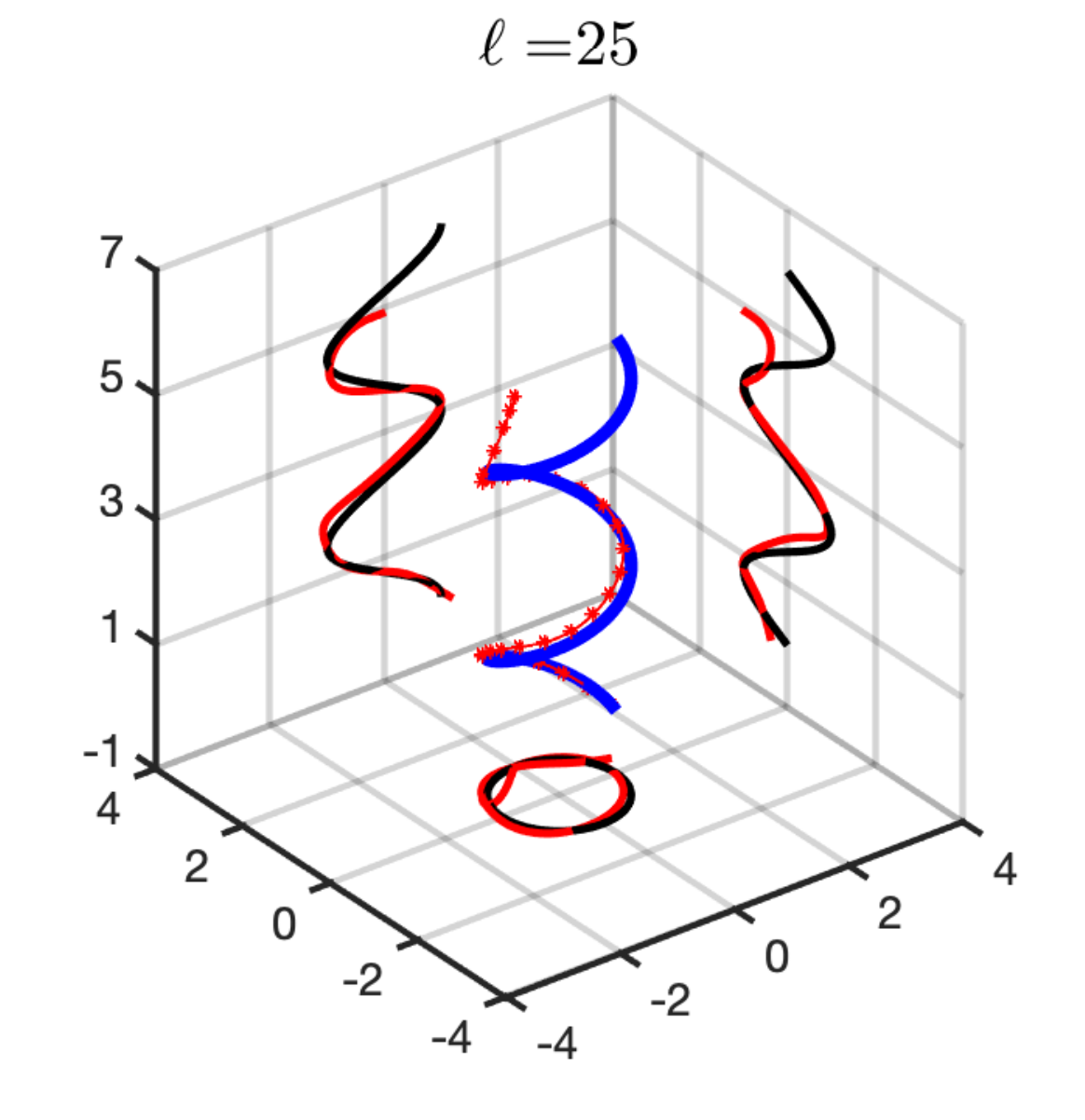}
    \end{subfigure}
    \begin{subfigure}{.32\textwidth}
      \includegraphics[width=1.\textwidth]{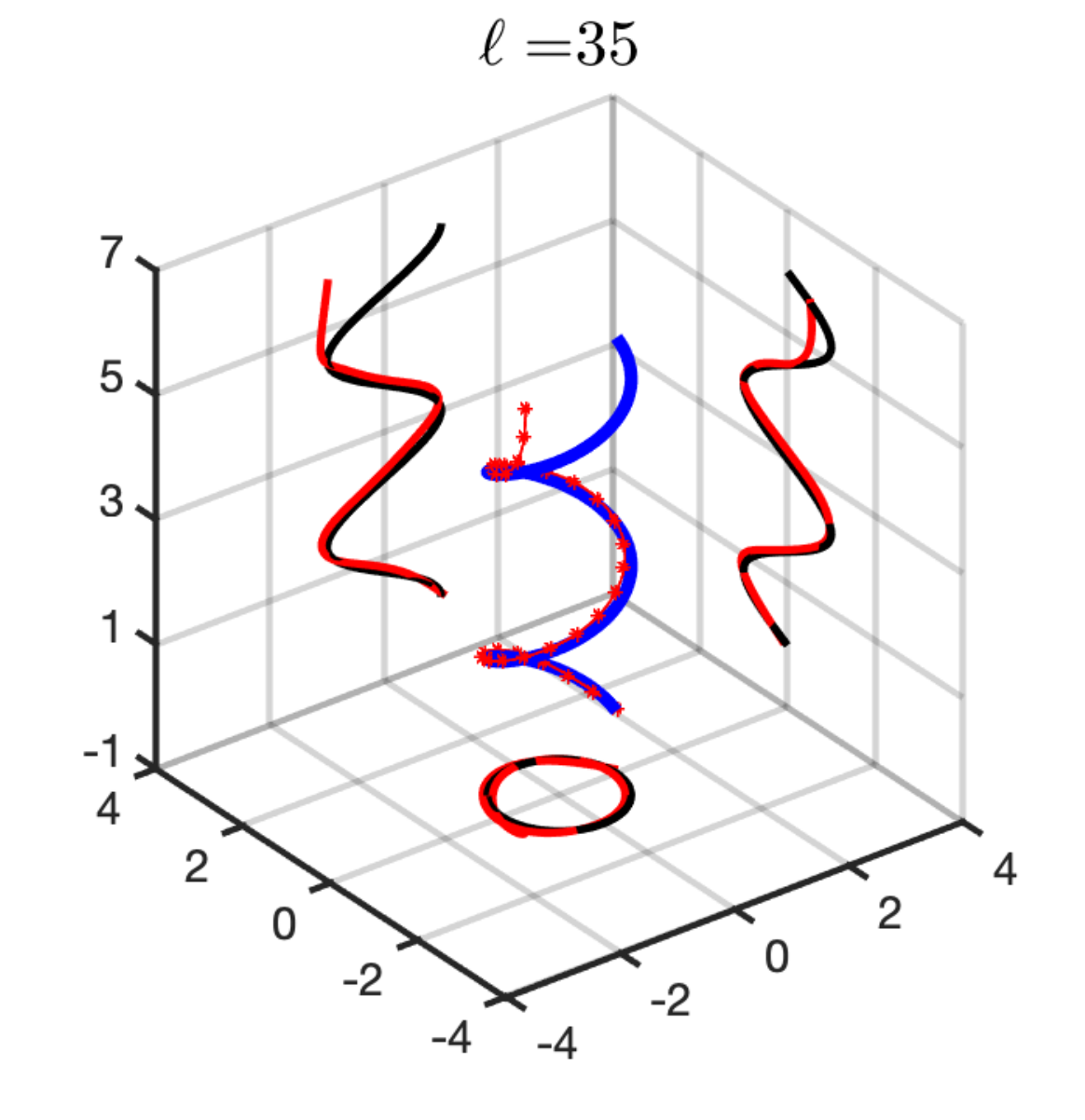}
    \end{subfigure}
    \begin{subfigure}{.32\textwidth}
      \includegraphics[width=1.\textwidth]{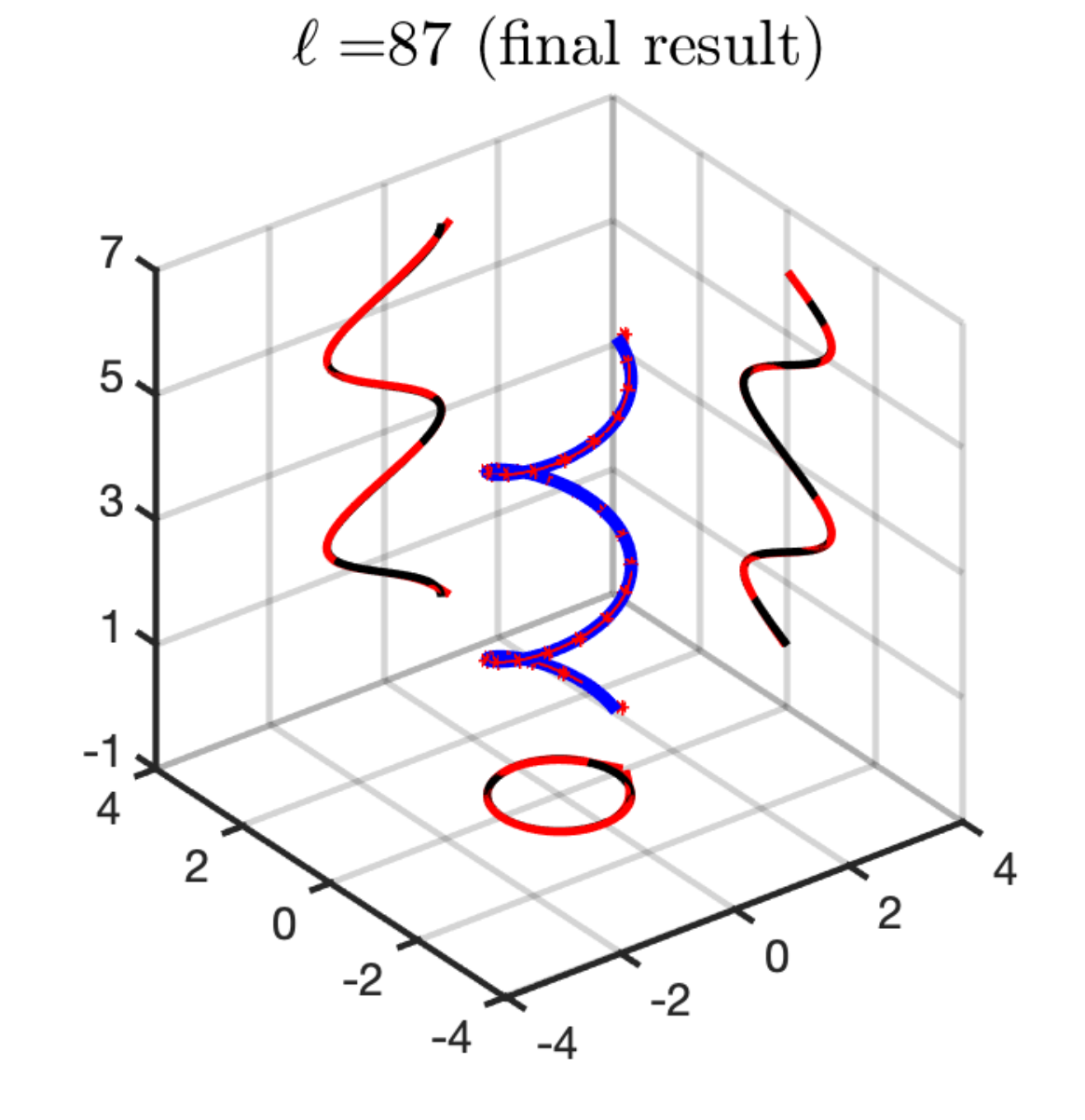}
    \end{subfigure}
    \caption{\small Reconstruction of the helical scatterer from
      Example~\ref{exa:Geometry3}. 
      The top-left plot shows the initial guess, and the
      bottom-right plot shows the final reconstruction.}
    \label{fig:Recon3}
  \end{figure}
  We consider the setting from Example~\ref{exa:Geometry3}. 
  The initial guess is a straight line segment connecting the points 
  $[0,-1,1]^\top$ and $[0,-2,1]^\top$.
  The reconstruction algorithm stops after~$87$ steps.
  The initial guess, some intermediate steps and the final result of
  the reconstruction algorithm are shown in Figure~\ref{fig:Recon3}.
  As in the previous examples, the final reconstruction is very close
  to the exact center curve $K$.
  \hfill$\lozenge$
\end{example}

In all three examples Algorithm~\ref{alg:reconstruction} provides
accurate approximations to the center curve $K$ of the unknown
scattering object $\Drho$. 
However, a suitable choice of the regularization parameters 
$\alpha_1$ and $\alpha_2$, and an initial guess $\bfp_{\tri}$
sufficiently close to the unknown center curve $K$ are crucial for a
successful reconstruction.

In our final example we study the sensitivity of the reconstruction
algorithm to noise in the far field data.

\begin{example}
  \label{exa:ReconNoisy}
  \begin{figure}[t]
    \centering
    \begin{subfigure}{.32\textwidth}
      \includegraphics[width=1.\textwidth]{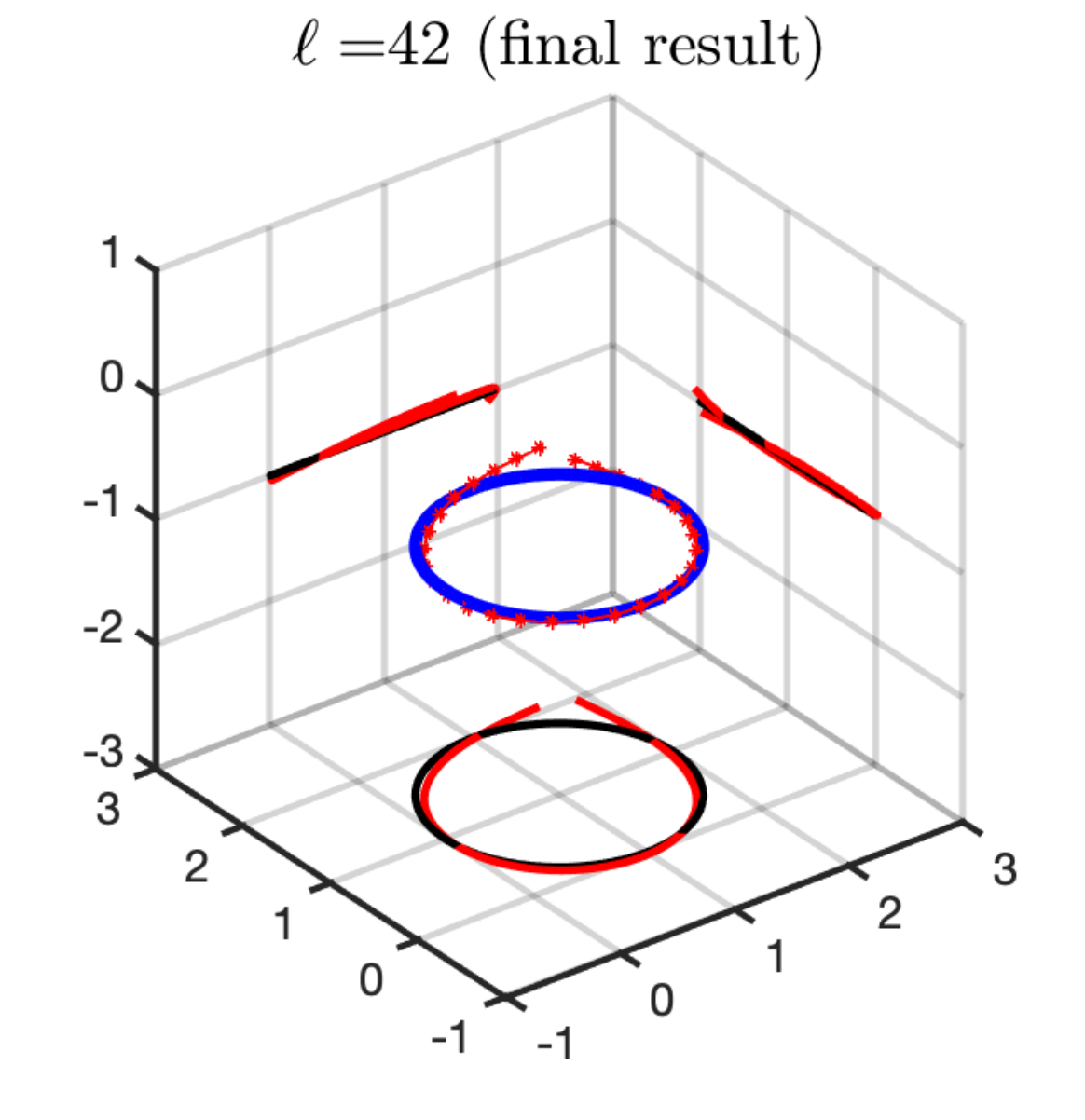}
    \end{subfigure}
    \begin{subfigure}{.32\textwidth}
      \includegraphics[width=1.\textwidth]{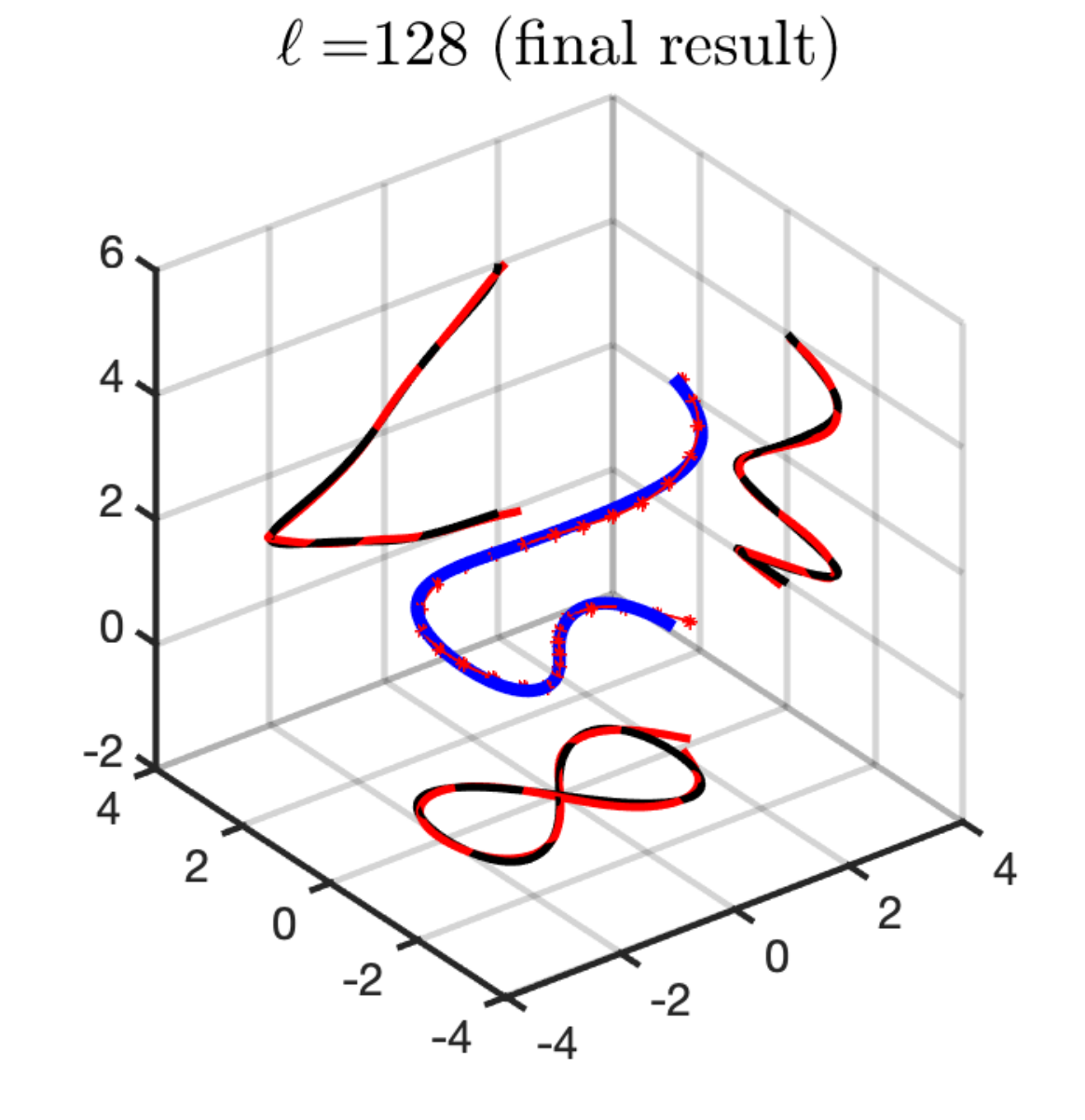}
    \end{subfigure}
    \begin{subfigure}{.32\textwidth}
      \includegraphics[width=1.\textwidth]{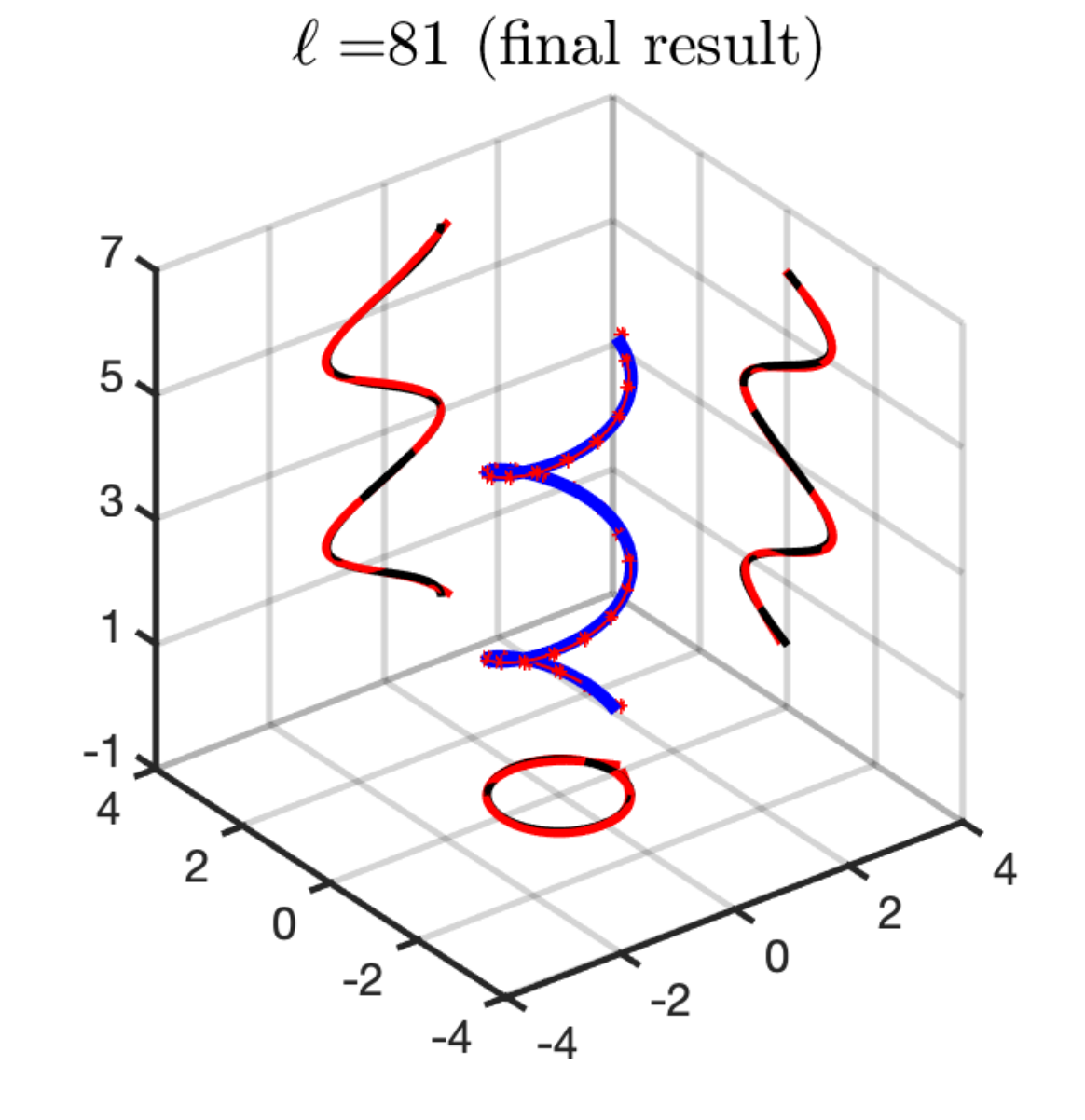}
    \end{subfigure}
    \caption{\small Reconstructions from noisy data with $30\%$ uniformly
      distributed additive noise.} 
    \label{fig:ReconNoisy}
  \end{figure}
  We repeat the previous computations but we add $30\%$ complex-valued
  uniformly distributed error to the electric far field patterns
  $\Einftyrho$ that have been simulates using Bempp.  
  We use the same initial guesses and the same initial values for the
  regularization parameters $\alpha_1$ and $\alpha_2$ as in
  Examples~\ref{exa:Recon1}--\ref{exa:Recon3}. 
  The three plots in Figure~\ref{fig:ReconNoisy} show the exact center
  curves (solid blue), the final reconstructions (solid red with
  stars), and the projections of these curves onto the coordinate
  planes.  
  The reconstruction algorithm stops after $42$ iterations for
  Example~\ref{exa:Recon1}, after $128$ iterations for
  Example~\ref{exa:Recon2}, and after $81$ iterations for
  Example~\ref{exa:Recon3}, respectively. 
  Despite the relatively high noise level, the final reconstructions
  are still very close to the exact center curves $K$. 
  
  We note that Algorithm~~\ref{alg:reconstruction} incorporates all
  available a priori information about the radius $\rho$ and the shape
  of the cross-section of the unknown scatterer, and its material
  parameters $\eps_r$ and $\mu_r$. 
  Furthermore it reconstructs a relatively low number of control
  points corresponding to the spline approximation~$\bfp_\tri$ of the
  center curve $K$ of the unknown thin tubular scattering object
  $\Drho$. 
  We also have carefully regularized the output least squares
  functional $\Phi$ in \eqref{eq:RegOLS}. 
  This might explain the good performance of the reconstruction
  algorithm even for rather noisy far field data. 
  \hfill$\lozenge$
\end{example}

\section*{Conclusions}
\label{sec:Conclusions}
The scattered electromagnetic field due to a thin tubular scattering
object in homogeneous free space can be approximated efficiently using
an asymptotic representation formula in terms of the dyadic Green's
function of the background medium, the incident electromagnetic field,
and two polarization tensors that encode the shape and the material
parameters of the thin tubular scatterering object.
In this work we have shown that, for a thin tubular scattering object
with a fixed cross-section that possibly twists along the center
curve, these three-dimensional polarization tensors can be computed
from the parametrization of the center curve and the two-dimensional
polarization tensors of the cross-section. 

For sufficiently regular cross-sections these two-dimensional
polarization tensors can be approximated by solving a two-dimensional
transmission problem for the Laplace equation. 
For ellipsoidal cross-sections explicit formulas are available. 
This gives a very efficient tool to analyze and simulate scattered
fields due to thin tubular structures. 

We have applied these results to develop an efficient iterative
reconstruction method to recover the center curve of a thin tubular
scattering object from far field observations of a single scattered
field. 
Our numerical examples illustrate the accuracy of the asymptotic
perturbation formula and the performance of the iterative
reconstruction method.

\begin{appendix}
  \section*{Appendix}
  \section{Local coordinates}
  \label{app:LocCoords}

  In this section we derive the Jacobian determinant $\JrK$ and
  compute the gradient in the local coordinate system~$\rK$
  from~\eqref{eq:LocalCoordsK}.
  The Frenet-Serret formulas state that the Frenet-Serret
  frame~$(\tK,\nK,\bK)$ from \eqref{eq:FrenetSerretFrame} satisfies 
  \begin{equation*}
    \tKprime \,=\, \kappa\nK \,, \qquad
    \bKprime \,=\, -\tau\nK \,, \qquad
    \nKprime \,=\, \tau\bK-\kappa\tK \,.
  \end{equation*}
  Therewith we find that 
  \begin{equation*}
    \begin{split}
      \frac{\di\rK}{\di s}(s,\eta,\zeta)
      &\,=\, 
      \tK(s) \Bigl( 1-\kappa(s) 
      \begin{bmatrix}
        1&0
      \end{bmatrix}
      \Rtheta(s)
      \begin{bmatrix}
        \eta\\\zeta
      \end{bmatrix}\Bigr)\\
      &\phantom{\,=\,}
      +
      \begin{bmatrix}
        \nK(s) & \bK(s)
      \end{bmatrix}
      \biggl(
      \begin{bmatrix}
        0 & -\tau(s)\\ \tau(s) & 0
      \end{bmatrix}
      \Rtheta(s) + \Rthetaprime(s) \biggr)
      \begin{bmatrix}
        \eta\\\zeta
      \end{bmatrix}
      \,,\\
      \frac{\di\rK}{\di\eta}(s,\eta,\zeta)
      &\,=\, \begin{bmatrix}
        \nK(s) & \bK(s)
      \end{bmatrix}\Rtheta(s)
      \begin{bmatrix}
        1\\0
      \end{bmatrix}\,,\\
      \frac{\di\rK}{\di\zeta}(s,\eta,\zeta)
      &\,=\, \begin{bmatrix}
        \nK(s) & \bK(s)
      \end{bmatrix}\Rtheta(s)
      \begin{bmatrix}
        0\\1
      \end{bmatrix}\,.
    \end{split}
  \end{equation*}
  Accordingly,
  \begin{equation*}
    \JrK(s,\eta,\zeta)
    \,:=\, \det D\rK(s,\eta,\zeta)
    \,=\, 1-\kappa(s)
    \begin{bmatrix}
      1&0
    \end{bmatrix}
    \Rtheta(s)
    \begin{bmatrix}
      \eta\\\zeta
    \end{bmatrix} \,. 
  \end{equation*}

  Furthermore, applying the chain rule we obtain that
  \begin{equation*}
    \begin{split}
      \frac{\di u\circ\rK}{\di s}(s,\eta,\zeta)
      &\,=\, \grad u (s,\eta,\zeta) \cdot \biggl( 
      \tK(s) \Bigl( 1-\kappa(s) 
      \begin{bmatrix}
        1&0
      \end{bmatrix}
      \Rtheta(s)
      \begin{bmatrix}
        \eta\\\zeta
      \end{bmatrix}\Bigr)\\
      &\phantom{\,=\, \grad u (s,\eta,\zeta) \cdot \Bigl(}
      +
      \begin{bmatrix}
        \nK(s) & \bK(s)
      \end{bmatrix}
      \Bigl(
      \begin{bmatrix}
        0 & -\tau(s)\\ \tau(s) & 0
      \end{bmatrix}
      \Rtheta(s) + \Rthetaprime(s) \Bigr)
      \begin{bmatrix}
        \eta\\\zeta
      \end{bmatrix} \biggr) \,,\\
      \frac{\di u\circ\rK}{\di\eta}(s,\eta,\zeta)
      &\,=\, \grad u (s,\eta,\zeta) \cdot \Bigl( 
      \begin{bmatrix}
        \nK(s) & \bK(s)
      \end{bmatrix}\Rtheta(s)
      \begin{bmatrix}
        1\\0
      \end{bmatrix} \Bigr) \,,\\
      \frac{\di u\circ\rK}{\di\zeta}(s,\eta,\zeta)
      &\,=\, \grad u (s,\eta,\zeta) \cdot \Bigl( 
      \begin{bmatrix}
        \nK(s) & \bK(s)
      \end{bmatrix}\Rtheta(s)
      \begin{bmatrix}
        0\\1
      \end{bmatrix} \Bigr) \,.
    \end{split}
  \end{equation*}
  Using the orthogonal decomposition
  \begin{equation*}
    \nabla u 
    \,=\, (\tK\cdot\grad u)\tK + (\nK\cdot\grad u)\nK 
    + (\bK\cdot\grad u)\bK 
  \end{equation*}
  and the notation
  \begin{equation*}
    \crossgrad u\circ\rK 
    \,:=\, \begin{bmatrix}
      \frac{\di u\circ\rK}{\di\eta} &
      \frac{\di u\circ\rK}{\di\zeta}
    \end{bmatrix}^\trans
  \end{equation*}
  for the two-dimensional gradient with respect to $(\eta,\zeta)$, we
  find that the gradient satisfies
  \begin{multline*}
    \grad u(\rK(s,\eta,\zeta))
    \,=\, \JrK^{-1}(s,\eta,\zeta) \biggl(
    \frac{\di u\circ\rK}{\di s} (s,\eta,\zeta)
    + \Bigl(\tau+\thetaprime\Bigr)(s)
    \begin{bmatrix}
      \zeta \\ -\eta
    \end{bmatrix} \cdot \bigl(\crossgrad u\circ\rK\bigr) (s,\eta,\zeta)
    \biggr)\tK(s)\\ 
    + \begin{bmatrix}
      \nK(s) & \bK(s)
    \end{bmatrix}
    \Rtheta(s) 
    \bigl( \crossgrad u\circ\rK\bigr)(s,\eta,\zeta) \,.
  \end{multline*}

  \section{Some estimates}
  \label{app:Estimates}
  In the following we collect some estimates that are used in the
  proof of Proposition~\ref{pro:PolTenWxi} (see also~\cite{CapVog03a}
  for \eqref{eq:Estwq1}--\eqref{eq:Estwq2}). 

  \begin{lemma}
    \label{lmm:Estwrho}
    Suppose $D\tm\Omega\tm\R^d$, $d=2,3$, let $\gamma_0,\gamma_1>0$,
    and let $\gamma\in L^\infty(\Omega)$ be defined by
    \begin{equation*}
      \gamma 
      \,:=\, 
      \begin{cases}
        \gamma_1\,, & \bfx \in D\,,\\
        \gamma_0\,, & \bfx \in \R^d\setminus\overline{D} \,,
      \end{cases}
    \end{equation*}
    Given $\bfF\in L^\infty(\Omega,\R^d)$, we denote by 
    $w\in H^1_0(\Omega)$ the unique solution to
    \begin{equation}
      \label{eq:wrho}
      \div(\gamma \grad w) 
      \,=\, \div(\chi_{D}\bfF) \quad \text{in $\Omega$} \,,\qquad
      w \,=\, 0 \quad \text{on $\di\Omega$} \,.
    \end{equation}
    Then, there exist constants $C,C_p>0$ such that
    \begin{subequations}
       
      \begin{align}
        \|\grad w\|_{L^2(\Omega)} 
        &\,\leq\, C |D|^{\frac12} 
          \|\bfF\|_{L^\infty(D)} \,,\label{eq:Estwq1}\\
        \|w\|_{L^2(\Omega)} 
        &\,\leq\, C |D|^{\frac34}
          \|\bfF\|_{L^\infty(D)} \,, \label{eq:Estwq2}\\
        \|w\|_{W^{1,p}(\Omega)} 
        &\,\leq\, C_p |D|^{\frac1p} 
          \|\bfF\|_{L^\infty(D)} \,, \qquad 1<p<2 \,.\label{eq:Estwq3}
      \end{align}
    \end{subequations}
  \end{lemma}

  \begin{proof}
    Using the weak formulation of \eqref{eq:wrho} and H\"older's
    inequality we find that
    \begin{equation*}
      \|\grad w\|_{L^2(\Omega)}^2 
      \,\leq\, C \int_\Omega \gamma \grad w \cdot \grad w \dx
      \,=\, C \int_\Omega \chi_{D} \bfF\cdot \grad w \dx
      \,\leq\, C |D|^{\frac12} \|\grad w\|_{L^2(\Omega)} 
      \|\bfF\|_{L^\infty(D)} \,.
    \end{equation*}
    This gives \eqref{eq:Estwq1}.

    Let $z \in H^1_0(\Omega)$ be the unique solution to
    \begin{equation}
      \label{eq:zrho}
      \div(\gamma_0\grad z) 
      \,=\, -w \quad \text{in $\Omega$}\,, \qquad 
      z \,=\, 0 \quad \text{on $\di\Omega$}\,.
    \end{equation}
    Elliptic regularity results (see, e.g., \cite[Thm.~8.13]{GilTru01}) 
    show that $\|z\|_{H^3(\Omega)}\leq C\|w\|_{H^1(\Omega)}$, and
    Sobolev's embedding theorem (see, e.g., \cite[p.~158]{GilTru01}) 
    gives 
    $\|\grad z\|_{L^\infty(\Omega)} \leq C
    \|z\|_{H^3(\Omega)}$. 
    Using the weak formulations of \eqref{eq:zrho} and \eqref{eq:wrho}
    we find that  
    \begin{equation*}
      \begin{split}
        \|w\|_{L^2(\Omega)}^2
        &\,=\, \int_{\Omega} \gamma_0 \grad z \cdot \grad w \dx
        \,=\, \int_{\Omega} \chi_{D} \bfF \cdot \grad z \dx
        + \int_{\Omega} (\gamma_0-\gamma) \grad z \cdot \grad w
        \dx\\
        &\,\leq\, \bigl( \|\chi_{D} \bfF\|_{L^1(\Omega)}
        + \|(\gamma_0-\gamma) \grad w\|_{L^1(\Omega)} \bigr)
        \|\grad z\|_{L^\infty(\Omega)}\\
        &\,\leq\, C \bigl(
        |D| \|\bfF\|_{L^\infty(D)}
        + |D|^{\frac12} \|\grad w\|_{L^2(\Omega)} \bigr)
        \|w\|_{H^1(\Omega)} \,.
      \end{split}
    \end{equation*}
    Applying Poincare's inequality and \eqref{eq:Estwq1} this shows
    \eqref{eq:Estwq2}. 

    Next we note that
    \begin{equation*}
      \div(\gamma_0\grad w) 
      \,=\, \div(\chi_{D} \bfF) 
      + \div\bigl((\gamma_0-\gamma)\grad w\bigr) \,.
    \end{equation*}
    If $1<p<2$, then the right hand side is in $W^{-1,p}(\Omega)$, and since
    $-\div(\gamma_0\grad \cdot)$ is an isomorphism from~$W^{1,p}_0(\Omega)$
    to $W^{-1,p}(\Omega)$ (see, e.g., \cite[p.~40]{BenLioPap78}), 
    we find using H\"older's inequality and \eqref{eq:Estwq1} that
    \begin{equation*}
      \begin{split}
        \|w\|_{W^{1,p}(\Omega)}
        &\,\leq\, C_p \|\chi_{D} \bfF\|_{L^p(\Omega)}
        + C_p \|(\gamma_0-\gamma)\nabla w\|_{L^p(\Omega)}\\
        &\,\leq\, C_p |D|^{\frac{1}{p}} \|\bfF\|_{L^\infty(D)}
        + C_p |D|^{\frac{1}{p}-\frac12} \|\nabla w\|_{L^2(\Omega)}
        \,\leq\, C_p |D|^{\frac{1}{p}} \|\bfF\|_{L^\infty(D)} \,.
      \end{split}
    \end{equation*}
    This gives \eqref{eq:Estwq3}.
  \end{proof}

  The next lemma is used in the proof of
  Theorem~\ref{thm:CharacterizationPolTen}~(b). 

  \begin{lemma}
    \label{lmm:Caccioppoli}
    Let $0<\rho<r/2$ and let $\Drho'\tm \Brho$ be open, where
    $\Brho\tm \R^2$ denotes the disk of radius $\rho$ around zero.
    Suppose that $A_0,A_1\in C^{0,1}(\Br,\R^{2\times 2})$ are
    symmetric and 
    \begin{equation*}
      c^{-1}
      \,\leq\, \bfxi' \cdot A_j(\bfx') \bfxi'
      \,\leq\, c \qquad
      \text{for all } \bfx'\in\Br \,,\; \bfxi'\in \Stwo \,,
      \text{ and } j=1,2 \,,
    \end{equation*}
    with some constant $c>0$, and let $\bfF\in C^{0,1}(\Br,\R^2)$.
    We define $A_\rho,\Atilde_\rho\in C^{0,1}(\Br,\R^{2\times 2})$ by
    \begin{equation*}
      A_\rho(\bfx') \,:=\,
      \begin{cases}
        A_1\,, & \bfx' \in \Drho\,,\\
        A_0\,, & \bfx' \in \Br\setminus\overline{\Drho} \,,
      \end{cases}
      \qquad\text{and}\qquad
      \Atilde_\rho(\bfx') \,:=\,
      \begin{cases}
        A_1(0)\,, & \bfx' \in \Drho\,,\\
        A_0(0)\,, & \bfx' \in \Br\setminus\overline{\Drho} \,,
      \end{cases}
    \end{equation*}
    and we consider the unique solutions 
    $\wrho,\wrhotilde\in H^1_0(\Br)$ to 
    \begin{subequations}
       
      \begin{align}
        \div(A_\rho\wrho) 
        &\,=\, \div(\chi_{\Drho} \bfF) 
          \quad \text{in } \Br \,, \qquad 
          \wrho = 0 \quad \text{on } \di\Br \,, \label{eq:Cacc-wrho}\\
        \div(\Atilde_\rho\wrhotilde) 
        &\,=\, \div(\chi_{\Drho} \bfF) 
          \quad \text{in } \Br \,, \qquad 
          \wrhotilde = 0 \quad \text{on } \di\Br \,.
      \end{align}
    \end{subequations}
    Then,
    \begin{equation}
      \label{eq:Caccioppoli}
      \bigl\| \grad \wrho - \grad \wrhotilde \bigr\|_{L^2(\Br)}
      \,=\, o(|\Drho|^{\frac12}) \qquad \text{as } \rho\to 0 \,.
    \end{equation}
  \end{lemma}

  \begin{proof}
    Using \eqref{eq:Cacc-wrho} we find that
    \begin{equation*}
      \div(\Atilde_\rho \grad \wrho)
      \,=\, \div(\chi_{\Drho} \bfF(0)) 
      + \div\bigl( \chi_{\Drho} (\bfF-\bfF(0)) \bigr)
      + \div\bigl( (\Atilde_\rho-A_\rho) \grad\wrho \bigr)
      \qquad \text{in } \Br \,.
    \end{equation*}
    Therefore, introducting $\Omega' := B_{\rho^{1/4}}'(0)$ we can
    write $\wrho = \wrhotilde + v_1 + v_2 + v_3$, where 
    $v_1,v_2,v_3\in H^1_0(\Br)$ are the unique solutions to
    \begin{subequations}
       
      \begin{align}
        \div(\Atilde_\rho v_1) 
        &\,=\, \div\bigl( \chi_{\Drho} (\bfF-\bfF(0)) \bigr)
        &\text{in } \Br \,, 
        &&v_1 = 0 \quad \text{on } \di\Br \,,\label{eq:Cacc-v1}\\
        \div(\Atilde_\rho v_2) 
        &\,=\, \div\bigl( \chi_{\Omega'} (\Atilde_\rho-A_\rho) 
          \grad\wrho \bigr)
        &\text{in } \Br \,, 
        &&v_2 = 0 \quad \text{on } \di\Br \,, \label{eq:Cacc-v2}\\
        \div(\Atilde_\rho v_3) 
        &\,=\, \div\bigl( (1-\chi_{\Omega'}) (\Atilde_\rho-A_\rho) 
          \grad\wrho \bigr)
        &\text{in } \Br \,, 
        &&v_3 = 0 \quad \text{on } \di\Br \,. \label{eq:Cacc-v3}
      \end{align}
    \end{subequations}
    Using \eqref{eq:Estwq1} and the Lipschitz continuity of $\bfF$
    we find that
    \begin{equation}
      \label{eq:Cacc-Estv1}
      \|\grad v_1\|_{L^2(\Br)}
      \,\leq\, C \| \bfF-\bfF(0) \|_{L^\infty(\Drho')} |\Drho'|^{\frac12}
      \,\leq\, C \rho |\Drho'|^{\frac12}
      \,=\, o(|\Drho'|^{\frac12}) \,.
    \end{equation}
    Similarly, the well-posedness of \eqref{eq:Cacc-v2},
    \eqref{eq:Estwq1}, and the Lipschitz continuity of $A_0$ and
    $A_1$ show that
    \begin{equation}
      \label{eq:Cacc-Estv2}
      \begin{split}
        \|\grad v_2\|_{L^2(\Br)}
        &\,\leq\, C \bigl\| (\Atilde_\rho-A_\rho) 
        \chi_{\Omega'} \bigl\|_{L^\infty(\Br)} 
        \|\grad\wrho\|_{L^2(\Br)}\\
        &\,\leq\, C \Bigl( \|A_1-A_1(0)\|_{L^\infty(\Drho')} 
        + \|A_0-A_0(0)\|_{L^\infty(\Omega')} \Bigr) 
        |\Drho'|^{\frac12}\\
        &\,\leq\, C \bigl(\rho+\rho^{\frac14}\bigr) |\Drho'|^{\frac12}
        \,=\, o(|\Drho'|^{\frac12}) \,.
      \end{split}
    \end{equation}
    Next let $\hrho\in C^1([0,r])$ be a cut-off function satisfying 
    \begin{equation}
      \label{eq:Propertieshrho}
      0\leq \hrho\leq 1 \,,\qquad 
      \chi_{(0,\,\rho^{\frac12})} \hrho \,=\, 0\,,\qquad
      \chi_{(\rho^{\frac14},\,1)} \hrho 
      \,=\, \chi_{(\rho^{\frac14},\,1)}\,, \qquad
      \|\hrho'\|_{L^\infty((0,\,r))} \,\leq\, C \rho^{-\frac14} \,.
    \end{equation}
    (see \cite[Lmm.~3.6]{BerCapdeGFra09} for a similar construction).
    Using the weak formulation of \eqref{eq:Cacc-wrho} and integrating
    by parts shows that
    \begin{equation*}
      \begin{split}
        0 
        &\,=\, \int_{\Drho'} \bfF \cdot \grad (\hrho^2\wrho) \dx'
        \,=\, \int_{\Br} A_\rho \grad\wrho 
        \cdot \grad(\hrho^2\wrho) \dx'\\
        &\,=\, \int_{\Br} A_\rho \grad\wrho 
        \cdot \bigl( \hrho \grad(\hrho\wrho) 
        + \hrho\wrho \grad\hrho \bigr) \dx'\\
        &\,=\, \int_{\Br} A_\rho \grad(\hrho\wrho)
        \cdot \grad(\hrho\wrho) \dx'
        - \int_{\Br} A_\rho \wrho^2 \grad\hrho \cdot \grad\hrho \dx' \,.
      \end{split}
    \end{equation*}
    Accordingly,
    \begin{equation*}
      \| \grad(\hrho\wrho) \|_{L^2(\Br)}^2
      \,\leq\, C \|\grad\hrho\|_{L^\infty(\Br)}^2 
      \| \wrho \|_{L^2(\Br)}^2 \,,
    \end{equation*}
    and applying \eqref{eq:Propertieshrho} and \eqref{eq:Estwq2} 
    gives 
    \begin{equation*}
      \| \grad(\hrho\wrho) \|_{L^2(\Br)}
      \,\leq\, C \rho^{-\frac14}
      |\Drho'|^{\frac34} 
      \,\leq\, C |\Drho'|^{-\frac18} |\Drho'|^{\frac34} 
      \,=\, o(|\Drho'|^{\frac12}) \,,
    \end{equation*}
    where we used that $\Drho'\tm\Brho$ and thus 
    $|\Drho'|\leq \pi \rho^2$.
    Combining \eqref{eq:Cacc-v2} with \eqref{eq:Estwq1} we obtain that
    \begin{equation}
      \label{eq:Cacc-Estv3}
      \|v_2\|_{L^2(\Br)}
      \,\leq\, C \| \grad\wrho \|_{L^2(\Br\setminus\ol{\Omega'})}
      \,\leq\, C \| \grad(\hrho\wrho) \|_{L^2(\Br)}
      \,=\, o(|\Drho'|^{\frac12}) \,.
    \end{equation}
    Finally, \eqref{eq:Cacc-Estv1}, \eqref{eq:Cacc-Estv2}, and
    \eqref{eq:Cacc-Estv3} give \eqref{eq:Caccioppoli}.
  \end{proof}

\end{appendix}

\section*{Acknowledgments}
Funded by the Deutsche Forschungsgemeinschaft (DFG, German Research
Foundation) -- Project-ID 258734477 -- SFB 1173.

{\small

}

\end{document}